\documentclass[dvipsnames]{amsart}

\usepackage{amsmath,amssymb,amsthm}
\usepackage{fullpage}
\usepackage{todonotes}

\usepackage{soul}

\usepackage{color}
\usepackage{hyperref}
\hypersetup{
	colorlinks=true,
	linkcolor=blue,
	citecolor=blue,
	urlcolor=blue
}

\usepackage{xcolor}

\newtheorem{theoremAlph}{Theorem}
\newtheorem{corollaryAlph}[theoremAlph]{Corollary}

\newtheorem{theorem}{Theorem}[section]
\newtheorem{lemma}[theorem]{Lemma}	
\newtheorem{proposition}[theorem]{Proposition}
\newtheorem{corollary}[theorem]{Corollary}
\theoremstyle{definition}
\newtheorem{definition}[theorem]{Definition} 
\newtheorem{remark}[theorem]{Remark}	

\newtheorem{question}[theorem]{Question}

\theoremstyle{definition} 
\newtheorem*{ack}{Acknowledgements}

\numberwithin{equation}{section}

\newcommand{\C}{\mathbb{C}}
\newcommand{\R}{\mathbb{R}}
\newcommand{\N}{\mathbb{N}}
\newcommand{\Z}{\mathbb{Z}}
\newcommand{\II}{\mathrm{I\!I}}
\newcommand{\Ric}{\mathrm{Ric}}
\newcommand{\BE}{\widetilde{\mathrm{Ric}}}
\newcommand{\Hess}{\mathrm{Hess}}
\newcommand{\bigslant}[2]{{\raisebox{.2em}{$#1$}\left/\raisebox{-.2em}{$#2$}\right.}}

\makeatletter
\DeclareRobustCommand{\rvdots}{%
	\vbox{
		\baselineskip4\p@\lineskiplimit\z@
		\kern-\p@
		\hbox{.}\hbox{.}\hbox{.}
}}
\makeatother

\makeatletter
\newcommand{\listlabel}[2]{#2\def\@currentlabel{#2}\label{#1}}
\makeatother

\title{Surgery and positive Bakry--Émery Ricci curvature} 
\date{}

\author[P.~Reiser]{Philipp Reiser$^{*}$}
	\address[Reiser]{Department of Mathematics, University of Fribourg, Switzerland.}
	\email{\href{mailto:philipp.reiser@unifr.ch}{philipp.reiser@unifr.ch}}

\author[F.~Tripaldi]{Francesca Tripaldi}
    \address[Tripaldi]{Department of Pure Mathematics, University of Leeds, UK.}
    \email{\href{mailto:f.tripaldi@leeds.ac.uk}{f.tripaldi@leeds.ac.uk}}

 \thanks{$^{*}$The author acknowledges funding by the SNSF-Project 200020E\textunderscore 193062 and the DFG-Priority programme SPP 2026.}

 \keywords{Bakry-Émery Ricci curvature, surgery, positive Ricci curvature, 5-manifolds}

 \subjclass[2020]{53C20, 53C21, 54E35, 57R65}

\begin{document}

\maketitle

\begin{abstract}
    We consider the problem of preserving weighted Riemannian metrics of positive Bakry-Émery Ricci curvature along surgery. We establish two theorems of this type: One for connected sums, and one for surgeries along higher-dimensional spheres. In contrast to known surgery results for positive Ricci curvature, these results are local, i.e.\ we only impose assumptions on the weighted metric locally around the sphere along which the surgery is performed. As application we then show that all closed, simply-connected spin 5-manifolds admit a weighted Riemannian metric of positive Bakry-Émery Ricci curvature. By a result of Lott, this also provides new examples of manifolds with a Riemannian metric of positive Ricci curvature.
\end{abstract}

\section{Introduction}

Surgery is an essential tool in differential topology which was introduced by Milnor \cite{Mi61} to eliminate certain homotopy classes of embedded spheres on a given manifolds. Recall that for an $n$-dimensional manifold $M^n$ and an embedding $\varphi\colon S^p\times D^{q+1}\hookrightarrow M$ of the product of the $p$-sphere with the $(q+1)$-disc with $n=p+q+1$, the manifold $M_\varphi$ obtained from $M$ by \emph{surgery along $\varphi$} is given by
\[ M_\varphi=M\setminus\varphi(S^p\times D^{q+1})^\circ\cup_{S^p\times S^q}(D^{q+1}\times S^p). \]
We also say that $M_\varphi$ is obtained by \emph{$p$-surgery} from $M$.

In the presence of a lower curvature bound, surgery offers a promising attempt to construct a wide class of manifolds satisfying this lower curvature bound, provided it can be preserved along a surgery operation. This was shown to be possible for positive scalar curvature by Schoen--Yau \cite{SY79} and Gromov--Lawson \cite{GL80a} whenever $q\geq 2$, which, in conjunction with index theory of Dirac operators, eventually led to a full classification of closed, simply-connected manifolds that admit a Riemannian metric of positive scalar curvature \cite{St92}.

For positive Ricci curvature, it is not known whether a surgery result in the same generality holds. In this context, we highlight the following questions:
\begin{question}\label{Q:conn_sum}
    Does the connected sum $M_1\# M_2$ of two closed $n$-manifolds $M_1$ and $M_2$ admit a Riemannian metric of positive Ricci curvature whenever both $M_1$ and $M_2$ admit such a metric (assuming at least one of $M_1$ and $M_2$ is simply-connected)? As a special case, does the connected sum $M_1\# \C P^{\frac{n}{2}}$ or $M_1\# (S^m\times S^{n-m})$ admit a Riemannian metric of positive Ricci curvature whenever $M_1$ admits such a metric?
\end{question}
\begin{question}\label{Q:dim-5}
    Which closed, simply-connected 5-manifolds admit a Riemannian metric of positive Ricci curvature?
\end{question}

Note that the connected sum operation is a particular instance of $0$-surgery. By the theorem of Bonnet--Myers, a connected sum $M_1\# M_2$ cannot admit a Riemannian metric of positive Ricci curvature when both $M_1$ and $M_2$ are not simply-connected. In all other cases, Question \ref{Q:conn_sum} is open.
A systematic study of the connected sum problem was initiated by Burdick \cite{Bu19a,Bu19,Bu20,Bu20a}, who, based on work by Perelman \cite{Pe97}, introduced the notion of \emph{core metrics}. These are Riemannian metrics of positive Ricci curvature that contain an embedded round hemisphere of the same dimension as the manifold (see Definition \ref{D:core} below). Burdick then showed that the connected sum of manifolds with core metrics admits a Riemannian metric of positive Ricci curvature. While this offers a promising approach towards answering Question \ref{Q:conn_sum}, it is not well understood which manifolds among the known examples of manifolds with a Riemannian metric of positive Ricci curvature admit core metrics, see Section \ref{S:conn_sums} for a full list. 

Question \ref{Q:dim-5} is of special interest, since it is known that \emph{all} closed, simply-connected 5-manifolds admit a Riemannian metric of positive scalar curvature \cite{GL80a}, which is a consequence of the aforementioned surgery result. At the same time, there exists a particularly simple classification of these manifolds by Smale \cite{Sm62} and Barden \cite{Ba65}. However, the number of known examples admitting a Riemannian metric of positive Ricci curvature is relatively small, see Subsection \ref{SS:5-manifolds} for a full list. In particular, among the known examples which are spin and have torsion in their homology, the second Betti number is at most $8$. On the other hand, it was shown by Sha--Yang \cite{SY91} that  all closed, simply-connected 5-manifolds with torsion-free homology admit a Riemannian metric of positive Ricci curvature, which was obtained by establishing a surgery result for \emph{higher surgeries}, i.e.\ $p$-surgeries with $p\geq 1$. This technique was subsequently extended and generalised by Wraith \cite{Wr97,Wr98} and the first named author \cite{Re23}, which provided new examples in dimensions at least 6. Nevertheless, all these results require strong geometric assumptions for the metrics involved, thus limiting their possible range of applications.

The purpose of this article is to study surgery in the context of a modified Ricci tensor, and in particular to address Questions \ref{Q:conn_sum} and \ref{Q:dim-5} in this setting.

\begin{definition}
    Let $(M^n,g,e^{-f})$ be an $n$-dimensional \emph{weighted Riemannian manifold}, i.e.\ $g$ is a Riemannian metric on $M$ and $f\colon M\to\R$ is a smooth function. Then for $q\in(0,\infty]$ the \emph{$q$-Bakry-Emery Ricci tensor} $\BE_q$ of $(M,g,e^{-f})$ is defined by
    \[ \BE_q=\Ric^g+\Hess(f)-\frac{1}{q}df\otimes df. \]
\end{definition}

The tensor $\BE_q$ was first introduced by Bakry and Émery \cite{BE85} in the context of diffusion processes. It also appears naturally in other settings, such as Ricci flow, general relativity and the study of Ricci limit spaces, see e.g.\ \cite{Lo03}, \cite{MW23}, \cite{WW09} and the references therein. It was shown by Lott \cite{Lo03} that if a closed $n$-manifold $M$ admits a weighed Riemannian metric $(g,e^{-f})$ of $\BE_q>0$, then the product $M\times S^p$ admits a Riemannian metric of positive Ricci curvature for all $p\geq \max\{2,q\}$ (see Proposition \ref{P:BE_bdle} below for a generalisation of this result). The metric on $M\times S^p$ is constructed in such a way that one can collapse the sphere $S^p$ to a point while preserving $\Ric>0$, which shows that $(M,g)$ is a collapsed Gromov--Hausdorff limit of Riemannian manifolds of $\Ric>0$, and, in particular, the metric measure space $(M,g,e^{-f}d\mathrm{vol}_g)$ satisfies the synthetic curvature condition $\mathrm{CD}(0,n+p)$.

Lott's result also shows that the existence of a weighted Riemannian metric of $\BE_q>0$ leads to examples in the Riemannian case, and, similarly as in the Riemannian case, we obtain that the fundamental group of a closed manifold with a weighted Riemannian metric of $\BE_q>0$ is finite. We further discuss the relation between $\BE_q>0$ and $\Ric>0$ in Appendix \ref{A:BE_Ric} below. It is worth noting that there is no difference known between the class of manifolds admitting a weighted Riemannian metric of $\BE_q>0$ and the class of manifolds admitting a Riemannian metric of $\Ric>0$.

In our first main result we consider gluing of two weighted Riemannian manifolds of $\BE_q>0$ along isometric boundary components. This generalises a corresponding gluing result of Perelman \cite{Pe97} in the Riemannian case, see also \cite{BWW19}, \cite{RW23}, the survey article \cite{Ke24a}, and Theorem \ref{T:gluing_Riem} below. We denote by $\II$ the second fundamental form and by $H^f$ the \emph{weighted mean curvature} defined by $H^f=H-\nu(f)$, where $H$ is the mean curvature and $\nu$ the outward unit normal of the boundary.

\begin{theoremAlph}
    \label{T:GLUING}
    Let $(M_1,h_1,e^{-f_1})$ and $(M_2,h_2,e^{-f_2})$ be two weighted Riemannian $n$-manifolds with $\BE_q>0$ for some $q\in(0,\infty]$, and suppose there exists an isometry $\phi\colon \partial_c M_1\to\partial_c M_2$ between two boundary components $\partial_c M_1\subseteq \partial M_1$ and $\partial_c M_2\subseteq \partial M_2$ such that $f_1|_{\partial_c M_1}=f_2\circ\phi$. If
    \begin{enumerate}
        \item $H_{\partial_c M_1}^{f_1}+H_{\partial_c M_2}^{f_2}\circ\phi\geq 0$, and
        \item $\II_{\partial_c M_1}+\phi^*\II_{\partial_c M_2}\geq 0$,
    \end{enumerate}
    then there exists a metric $h$ and a smooth function $f$ on $M_1\cup_{\phi} M_2$, which agree with $h_i$ and $f_i$ on $M_i$ outside an arbitrarily small neighbourhood of $\partial_c M_i$, such that $(M_1\cup_\phi M_2,h,e^{-f})$ has $\BE_q>0$.
\end{theoremAlph}

This result was independently also obtained by Ketterer \cite{Ke24}, who additionally proved a converse in terms of the curvature-dimension condition $\mathrm{CD}(K,N)$ (see \cite[Theorem 1.4]{Ke24}). We note that our proof of Theorem~\ref{T:GLUING}, which is based on Perelman's work \cite{Pe97}, differs from Ketterer's proof, which is based on a construction of Kosovskii \cite{Ko02}, see also \cite{Sc12}.

Theorem \ref{T:GLUING} motivates the following generalisation of core metrics to the weighted setting:
\begin{definition}\label{D:weighted_core}
    Let $q\in(0,\infty]$. A weighted metric $(g,e^{-f})$ of $\BE_q>0$ on an $n$-dimensional manifold $M$ is called a \emph{weighted core metric} with respect to $q$, if there exists an isometric embedding $\varphi\colon D^n\hookrightarrow M$, where we consider $D^n$ as equipped with the induced metric of a hemisphere in the round sphere of radius $1$, such that $f$ is constant on $\varphi(D^n)$.
\end{definition}
In particular, we obtain a weighted core metric with respect to any $q\in(0,\infty]$ from a core metric by choosing a constant weight function.

Similar arguments as in the Riemannian case \cite{Bu19} using Theorem \ref{T:GLUING} instead of Perelman's gluing theorem now directly show that the connected sum of manifolds admitting weighted core metrics with respect to $q$ admits a weighted Riemannian metric of $\BE_q>0$. In fact, we can prove the following more general result:

\begin{theoremAlph}\label{T:CONN_SUMS}
    Let $q\in(0,\infty]$ and let $M_i^n$, $i=0,\dots,\ell$ be closed manifolds such that
    \begin{enumerate}
        \item $M_0$ admits a weighted Riemannian metric of $\BE_q>0$,
        \item $M_1,\dots,M_\ell$ admit a weighted core metrics with respect to $q$.
    \end{enumerate}
    Then the connected sum $M_0\#\dots\# M_\ell$ admits a weighted Riemannian metric of $\BE_q>0$.
\end{theoremAlph}

In particular, Theorem \ref{T:CONN_SUMS} answers the second part of Question \ref{Q:conn_sum} affirmatively if one replaces $\Ric>0$ by $\BE_q>0$, since both complex projective spaces and products of spheres $S^m\times S^{n-m}$ with $m,n-m\geq 2$ admit core metrics by \cite{Bu19,Bu20}, \cite{Re24}.

As pointed out by Erik Hupp, the case where $q=2$ and all $M_i$ with $i\geq 1$ are given by $\C P^2$ in Theorem~\ref{T:CONN_SUMS} also follows from the construction in \cite{HNW23}.

The idea for the proof of Theorem \ref{T:CONN_SUMS} is as follows. For simplicity we assume $\ell=1$. We then remove a small neighbourhood of a point in $M_0$ and a small neighbourhood of the hemisphere in $M_1$ and attach a cylinder $[0,t_0]\times S^{n-1}$ equipped with a weighted Riemannian metric of $\BE_q>0$ that connects the two pieces. To ensure that we can glue the cylinder to $M_0$ and $M_1$ using Theorem \ref{T:GLUING}, we will define a warped product metric on $[0,t_0]\times S^{n-1}$ whose warping function has derivative close to $0$ at $t=0$ (to glue with $M_1$) and close to $1$ at $t=t_0$ (to glue with $M_0$). Clearly this can be achieved by a convex function with suitable boundary conditions. However, the Ricci curvatures of such a metric in $t$-direction are strictly negative. To obtain a weighted Riemannian metric of $\BE_q>0$ we then carefully construct a weight function that compensates the negative contribution of the Ricci curvature while satisfying condition (1) of Theorem \ref{T:GLUING} at the gluing areas. This choice of functions is based on the construction in \cite{Re23}. It is worth noting that for the overall construction we change the weighted Riemannian metrics on $M_0$ and $M_1$ only in arbitrarily small neighbourhoods of a point in $M_0$ and of the embedded hemisphere in $M_1$, respectively.

Next, we consider higher surgeries.
\begin{theoremAlph}\label{T:HIGHER_SURG}
    Let $(M^n,g,e^{-f})$ be a weighted Riemannian manifold with $\BE_\infty>0$ and let $\varphi\colon S^a\times D^{b+1}\hookrightarrow M$, $a+b+1=n$, be an embedding such that $\varphi(S^a\times\{0\})$ is a round, totally geodesic sphere on which $f$ is constant with vanishing normal derivative. If $a,b\geq 2$, then $M_\varphi$ admits a weighted metric of $\BE_\infty>0$.
\end{theoremAlph}
The weighted Riemannian metric constructed in Theorem \ref{T:HIGHER_SURG} coincides with $(g,e^{-f})$ outside an arbitrarily small neighbourhood of the gluing area. For the proof we consider a doubly warped submersion metric on the cylinder $[0,t_0]\times S^a\times S^b$, together with a weight function that is constant along slices $\{t\}\times S^a\times S^b$. The goal is then to transition between a weighted metric that collapses each sphere $\{0\}\times S^a\times\{x\}$ to a point (to obtain the space $D^{a+1}\times S^b$), and a weighted metric that at $t=t_0$ can be glued to $M\setminus\varphi(S^a\times D^{b+1})^\circ$. This results in a system of differential inequalities (to obtain $\BE_\infty>0$) for the warping functions and the weight function with boundary conditions at $t=0,t_0$, for which we will construct explicit solutions.

The main improvement of Theorem \ref{T:HIGHER_SURG} compared to the known surgery results for $\Ric>0$ is that Theorem~\ref{T:HIGHER_SURG} is \emph{local}, that is, we only need to impose conditions on the central sphere $\varphi(S^a\times\{0\})$. In contrast, the surgery results for $\Ric>0$ in \cite{SY91}, \cite{Wr98}, \cite{Re23} all require the diameter of the discs $\varphi(\{x\}\times D^{b+1})$ for all $x\in S^p$ to be sufficiently large compared to the size of the sphere $\varphi(S^a\times\{0\})$.

We apply Theorem \ref{T:HIGHER_SURG} to closed, simply-connected spin 5-manifolds. Any such manifold can be obtained from the sphere $S^5$ by a sequence of surgeries since the $5$-dimensional spin bordism group $\Omega_5^{Spin}$ is trivial. However, in general, the corresponding embeddings will not satisfy the hypotheses of Theorem \ref{T:HIGHER_SURG}. To obtain round and totally geodesic embeddings $S^2\hookrightarrow S^5$, one can for example consider intersections with $S^5$ of linear $3$-dimensional subspaces of $\R^6$. We will see in Section \ref{S:5-mfds} below that this results in manifolds $M^5$ with second homology group given by $H_2(M)\cong (\Z/n)^2$ with $n$ odd. To obtain more general homology groups, we will give a procedure to slightly shift a linear subspace in a given direction, while preserving all properties required to apply Theorem \ref{T:HIGHER_SURG}. A careful analysis of the possible linkings of 2-spheres in $S^5$ we can produce in this way, together with Smale's classification of closed, simply-connected spin 5-manifold \cite{Sm62}, results in the following.
\begin{theoremAlph}\label{T:dim-5}
    All closed, simply-connected spin 5-manifolds admit a weighted Riemannian metric of ${\BE_\infty>0}$. 
\end{theoremAlph}
This answers Question \ref{Q:dim-5} in the spin case if one replaces $\Ric>0$ by $\BE_\infty>0$. We note that, in combination with Theorem \ref{T:CONN_SUMS}, we also obtain a partial result in the non-spin case, see Theorem \ref{T:5-mfds} below. Moreover, the same techniques as in the proof of Theorem \ref{T:dim-5} can be applied to highly-connected $(4m+1)$-manifolds, i.e.\ closed, $(2m-1)$-connected manifolds of dimension $(4m+1)$, see Theorem \ref{T:highly-conn} below.

Since $\BE_\infty>0$ on a closed manifold implies $\BE_q>0$ for all $q$ sufficiently large, Theorem \ref{T:dim-5}, together with Lott's results \cite{Lo03}, has the following consequence:
\begin{corollaryAlph}\label{C:5-mfs}
    Let $M$ be a closed, simply-connected spin 5-manifold. Then there exists $q\in\N$ such that $M\times S^q$ admits a Riemannian metric of positive Ricci curvature.
\end{corollaryAlph}



This article is laid out as follows. In Section \ref{S:prelim} we introduce weighted Riemannian manifolds and recall constructions and curvature formulae for metrics on a cylinder. We then proceed by proving Theorems \ref{T:GLUING}, \ref{T:CONN_SUMS}, \ref{T:HIGHER_SURG}, and \ref{T:dim-5} in Sections \ref{S:gluing}, \ref{S:conn_sums}, \ref{S:higher_surg}, and \ref{S:5-mfds}, respectively. Finally, in Appendix \ref{A:BE_Ric}, we compare the conditions $\BE_q>0$ and $\Ric>0$ and collect results that allow to construct Riemannian metrics of $\Ric>0$ from weighted Riemannian metrics of $\BE_q>0$.

\begin{ack}
    F.T. would like to thank the Centro di Ricerca Matematica Ennio De Giorgi and the Scuola Normale Superiore for their hospitality and support. She also acknowledges the support and hospitality of the Department of Mathematics of Fribourg during her visit. P.R. would like to thank Christian Ketterer and David Wraith for helpful conversations. Finally, both authors would like to thank David Wraith for his suggestions on an earlier version of this article, Manuel Krannich for his explanations on smoothing theory, and Erik Hupp for making us aware of the article \cite{HNW23}.
\end{ack}

\section{Preliminaries}\label{S:prelim}

In this section, we will present the main definitions about weighted Riemannian manifolds to fix the notation. We will also present the formulae for different types of curvature on a Riemannian manifold of the type $M=I\times X$ with the metric $h=dt^2+g_t$, where $X$ a Riemannian manifold, $I\subset\mathbb{R}$ an interval and $g_t$ denotes a smoothly varying metric on $\lbrace t\rbrace\times X$ for each $t\in I$. These explicit computations will be used in the various steps needed to prove Theorems \ref{T:GLUING}--\ref{T:HIGHER_SURG}.

For a hypersurface $N\subseteq M$ in a Riemannian manifold $(M,g)$ with (local) unit normal field $\nu$, we denote by $\II(u,v)=g(\nabla_u\nu,v)$ its \emph{second fundamental form} and by $H=\mathrm{tr}_g\II$ its \emph{mean curvature}. When $N$ is the boundary of $M$, we choose $\nu$ to be the outward pointing unit normal.

\subsection{Weighted Riemannian manifolds}

In this section, we establish basic facts on weighted Riemannian manifolds. For further background literature, we refer to \cite{Lo03}, \cite{WW09}, and the references therein.

\begin{definition}
    Given a smooth manifold $M$, a Riemannian metric $g$ on $M$, and a smooth function $f\colon M\to\mathbb{R}$, we call the triple $(M,g,e^{-f})$ a \textit{weighted Riemannian manifold}.
\end{definition}
If $d\textrm{vol}_g$ denotes the Riemannian volume measure, then one can view a weighted Riemannian manifold $(M,g,e^{-f})$ as a Riemannian manifold $(M,g)$ equipped with a measure $e^{-f}d\mathrm{vol}_g$.

\begin{definition}\label{weighted q Ricci}
    For a given $q\in(0,\infty]$, we define the \emph{$q$-Bakry-Emery-Ricci tensor} $\BE_q$ of $(M,g,e^{-f})$ as
    \begin{align*}
        \BE_q=\Ric^g+\Hess(f)-\frac{1}{q}df\otimes df\,,
    \end{align*}
    where $\Ric^g$ stands for the Ricci tensor of $(M,g)$. Using the convention $\frac{1}{\infty}=0$, the $\infty$-Bakry-Emery-Ricci tensor is given by
    \begin{align*}
        \BE_\infty=\Ric^g+\Hess(f)\,.
    \end{align*}
\end{definition}
Note that $\BE_q$ depends on both the metric $g$ and the weight function $f$. Whenever the metric and weight function are not clear from the context, we will indicate the dependence by writing $\BE{}_q^{g,f}$. We will also refer to $\BE_q$ as the \emph{weighted Ricci curvatures} of $(M,g,e^{-f})$.

For $q'\geq q$ we have that $\BE_q>0$ implies $\BE_{q'}>0$. Since $\BE_q=\Ric$ for all $q$ whenever $f$ is constant, a Riemannian manifold of $\Ric>0$ satisfies $\BE_q>0$ for all $q$ with respect to a constant weight function.

\begin{definition}\label{weighted mean curv}
    Let $(M^n,g,e^{-f})$ be a weighted Riemannian manifold an let $N^{n-1}\subseteq M$ be an embedded hypersurface. For $x\in N$, let $\nu\in T_xM$ be a unit normal to $N$. Then the \emph{weighted mean curvature} $H^f$ at $x$ with respect to $\nu$ is defined by
    \[ H^f=H-g(\nu,\nabla f). \]
    Just like for the mean curvature, we will choose $\nu$ as the outward pointing unit normal when $N=\partial M$.
\end{definition}

\subsection{Weighted Riemannian metrics on a cylinder}\label{SS:cylinder}

In this section we establish curvature formulae for weighted Riemannian metrics on a cylinder, which we will need in the proof of Theorems \ref{T:GLUING}--\ref{T:HIGHER_SURG}.

Let us consider a product $M^n=I\times X^{n-1}$, where $I$ is an interval, and a metric $h$ on $M$ given by
\[ h=dt^2+g_t, \]
where $g_t$ is a smoothly varying family of Riemannian metrics on $X$. We will set $g_t'=\frac{\partial}{\partial t}g_t$ and $g_t''=\frac{\partial^2}{\partial t^2}g_t$.

\begin{lemma}[{\cite[Lemma 2.1]{Re24}}]
    \label{L:CURV_FORM}
    The second fundamental form of a slice $\{t\}\times X$ with respect to the unit normal $\partial_t$ is given by
    \[ \II_{\{t\}\times X}=\frac{1}{2}g_t'. \]
Furthermore, the Ricci curvatures of the Riemannian manifold $(M,h)$ at $(t,x)\in M$ are given as follows:
	\begin{align*}
		&\Ric^h(\partial_t,\partial_t)=-\frac{1}{2}\mathrm{tr}_{g_t}g_t^{\prime\prime}+\frac{1}{4}\lVert g_t^\prime\rVert_{g_t}^2,\\
		&\Ric^h(v,\partial_t)\;=-\frac{1}{2}v(\mathrm{tr}_{g_t}g_t')+\frac{1}{2}\sum_{i}(\nabla^{g_t}_{e_i}g_t')(v,e_i),\\
		&\Ric^h(u,v)\;\;=\Ric^{g_t}(u,v)-\frac{1}{2}g_t^{\prime\prime}(u,v)+\frac{1}{2}\sum_{i=1}^{n-1} g_t^\prime(u,e_i)g_t^\prime(v,e_i)-\frac{1}{4}g_t^\prime(u,v)\mathrm{tr}_{g_t}g_t^\prime.
		\end{align*}
    Here $u,v\in T_xX$ and $(e_i)$ is an orthonormal basis of $T_xX$ with respect to $g_t$.
\end{lemma}

Now let $f\colon M\to \R$ be a smooth function and we set $f_t=f(t,\cdot)\colon X\to \R$. The following Lemma, together with Lemma \ref{L:CURV_FORM}, provides the weighted Ricci curvatures of the weighted Riemannian manifold $(M,h,e^{-f})$.

\begin{lemma}\label{L:HESSIAN}
    Given $x\in M$, for any $u,v\in T_xX$ we have
    \begin{align*}
       & \Hess^h(f)(\partial_t,\partial_t)=f_t'',\\
        &\Hess^h(f)(u,\partial_t)=\frac{1}{2} u(f_t'),\\
        &\Hess^h(f)(u,v)=\frac{1}{2}f_t'g_t'(u,v)+\Hess^{g_t}(f_t)(u,v)
    \end{align*}
    and
    \begin{align*}
        &(df\otimes df)(\partial_t,\partial_t)=(f_t')^2,\\
        &(df\otimes df)(u,\partial_t)=f_t'df_t(u),\\
        &(df\otimes df)(u,v)=df_t(u)df_t(v)\,,
    \end{align*}
    where we are using the shorthand notation
    \begin{align*}
        f_t'=\tfrac{\partial}{\partial t}f\ \text{ and }\ f_t''=\tfrac{\partial^2}{\partial t^2}f\,.
    \end{align*}
\end{lemma}
\begin{proof}
    We extend $u$ and $v$ to local vector fields around $x\in X$, and then constantly to local vector fields around $(t,x)\in M$. First we calculate the Levi-Civita connection of $h$. Since $[u,\partial_t]=[v,\partial_t]=0$ and $[u,v]\in TX$, we obtain from the Koszul formula, where we denote by $\nabla^t$ the Levi-Civita connection of the metric $g_t$:
    \begin{align*}
        &\nabla_{\partial_t}\partial_t=0\quad ,\quad        h(\nabla_{u}\partial_t,v)=h(\nabla_{\partial_t}u,v)=\frac{1}{2}g_t'(u,v),\\
        h(\nabla_{u}\partial_t,&\partial_t)=h(\nabla_{\partial_t}u,\partial_t)=0\quad ,\quad
        \nabla_u v=-\frac{1}{2}g_t'(u,v)\partial_t+\nabla^t_u v.
    \end{align*}
    We have $h( \nabla f,\partial_t)=\partial_t f=f_t'$ and $h(\nabla f,u)=u(f)=u(f_t)=g_t(\nabla f_t,u)$. Thus,
    \[\nabla f=f'_t\partial_t+\nabla f_t. \]
    It follows that
    \begin{align*}
        \Hess^h(f)(u,v)&=h(\nabla_u\nabla f,v)=h(\nabla_u(f_t'\partial_t)+\nabla_u\nabla f_t,v)\\ \phantom{\Hess^h(f)(u,v)}&=h(u(f_t')\partial_t,v)+f_t'h(\nabla _u\partial_t,v)+\Hess^{g_t}(f_t)(u,v)
        =\frac{1}{2}f_t'g_t'(u,v)+\Hess^{g_t}(f_t)(u,v),\\
        \Hess^h(f)(u,\partial_t)&=h(\nabla_u\nabla f,\partial_t)=h(u(f_t')\partial_t+f_t'\nabla_u\partial_t+\nabla_u\nabla f_t,\partial_t)\\ 
        \phantom{\Hess^h(f)(u,\partial_t)}&=u(f'_t)-\frac{1}{2}g_t'(u,\nabla f_t)=u(f_t')-\frac{1}{2}\partial_t u(f_t)=\frac{1}{2}u(f_t'),\\
        \Hess^h(f)(\partial_t,\partial_t)&=h(\nabla_{\partial_t}\nabla f,\partial_t)=\partial_t h(\nabla f,\partial_t)=f_t''.
    \end{align*}
\end{proof}

We now consider the special case where $h$ is given by a doubly warped submersion metric and $f$ is constant on the hypersurfaces $\{t\}\times X$. This case will be important in the proof of Theorems \ref{T:CONN_SUMS} and \ref{T:HIGHER_SURG}.

For $a,b\in\N$ we consider the projection $S^a\times S^b\xrightarrow{\pi} S^a$. The vertical distribution $\mathcal{V}=\ker\pi_*$ is then simply given by $TS^b$. Let $\mathcal{H}\subseteq T(S^a\times S^b)$ be a distribution that is complementary to $\mathcal{V}$, so we have a decomposition
\[ T(S^a\times S^b)=\mathcal{H}\oplus\mathcal{V}. \]
The projection onto $\mathcal{H}$, which we call the \emph{horizontal distribution}, according to this composition will again be denoted by $\mathcal{H}$. Then for any $\alpha,\beta>0$ we define the metric
\[ g_\mathcal{\alpha,\beta}=\alpha^2\mathcal{H}^*\pi^*ds_a^2+\beta^2 ds_b^2 \]
and obtain a Riemannian submersion $(S^a\times S^b, g_{\mathcal{\alpha,\beta}})\to (S^a, \alpha^2 ds_a^2)$ with totally geodesic fibres isometric to $(S^b,\beta^2ds_b^2)$. Conversely, for any Riemannian submersion $(S^a\times S^b,g)\to (S^a,\alpha^2ds_a^2)$ with totally geodesic fibres isometric to $(S^b,\beta^2 ds_b^2)$ we define $\mathcal{H}=\mathcal{V}^\perp$ and obtain that $g$ is given by $g_{\alpha,\beta}$, see also \cite[Theorem 9.59]{Be87}.

Let $A$ denote the $A$ tensor of $g_{1,1}$ (for the definition and basic properties of the $A$-tensor we refer to \cite[Section 9.C]{Be87}). As in \cite[9.33]{Be87}, for a horizontal vector $u$ and vertical vector $v$ in $T_x(S^a\times S^b)$ we set
\begin{align*}
    (A_u,A_u)&=\sum_{i}g_{1,1}(A_u u_i,A_u u_i),\\
    (Av,Av)&=\sum_ig_{1,1} (A_{u_i}v,A_{u_i}v),\\
    ((\check{\delta}A)u,v)&=-\sum_ig_{1,1}(\nabla^{g_{1,1}}_{u_i}A)_{u_i}u,v),
\end{align*}
where $(u_i)$ is an orthonormal basis of $\mathcal{H}_x$.

\begin{lemma}\label{L:Ricci_doubly_warped_sub}
    Let $I$ be an interval and let $\alpha,\beta\colon I\to(0,\infty)$, $f\colon I\to\R$ be smooth functions. For $a,b\in\N$ and a horizontal distribution $\mathcal{H}\subseteq T(S^a\times S^b)$ as above consider the metric
    \[ h=dt^2+g_{\alpha(t),\beta(t)}=dt^2+\alpha(t)^2\mathcal{H}^*\pi^*ds_a^2+\beta(t)^2ds_b^2. \]
    We consider $f$ as a function defined on $I\times S^a\times S^b$ by composing it with the projection onto the first factor. Then, for $q\in(0,\infty]$, the weighted Ricci curvatures of the weighted Riemannian manifold $(I\times S^a\times S^b,h,e^{-f})$ are given by
    \begin{align*}
        \BE_q(\partial_t,\partial_t)&=-a\frac{\alpha''}{\alpha}-b\frac{\beta''}{\beta}+f''-\frac{1}{q}{f'}^2 ,\\
        \BE_q(\tfrac{u}{\alpha},\tfrac{u}{\alpha})&=-\frac{\alpha''}{\alpha}+(a-1)\frac{1-{\alpha'}^2}{\alpha^2}-b\frac{\alpha'\beta'}{\alpha\beta}+f'\frac{\alpha'}{\alpha}-2\frac{\beta^2}{\alpha^4}(A_u,A_u),\\
        \BE_q(\tfrac{v}{\beta},\tfrac{v}{\beta})&=-\frac{\beta''}{\beta}+(b-1)\frac{1-{\beta'}^2}{\beta^2}-a\frac{\alpha'\beta'}{\alpha\beta}+f'\frac{\beta'}{\beta}+\frac{\beta^2}{\alpha^4}(Av,Av),\\
        \BE_q(\tfrac{u}{\alpha},\tfrac{v}{\beta})&= -\frac{\beta}{\alpha^3}((\check{\delta}A)u,v),\\
        \notag\BE_q(\partial_t,\tfrac{u}{\alpha})&=\BE_q(\partial_t,\tfrac{v}{\beta})=0.
    \end{align*}
    Here $u$ and $v$ are unit horizontal and vertical vectors, respectively, and $A$ denotes the $A$-tensor.
\end{lemma}
\begin{proof}
    Let $A^{\alpha,\beta}$ denote the $A$-tensor of $g_{\alpha,\beta}$. Then, for horizontal vectors $u_1,u_2$ and a vertical vector $v$ we have
    \[ A^{\alpha,\beta}_{u_1}u_2=A_{u_1}u_2,\quad A^{\alpha,\beta}_{u_1}v=\frac{\beta^2}{\alpha^2}A_{u_1}v, \quad (\check{\delta}A)^{\alpha,\beta}=\frac{1}{\alpha^2}\check{\delta}A\,,\]
    see e.g.\ \cite[Lemma 9.69]{Be87}. Here we are using that the Levi--Civita connection, and therefore the $A$-tensor, does not change under scalar multiplication, so that $A^{\alpha,\beta}=A^{1,\frac{\beta}{\alpha}}$. It follows that
    \begin{align*}
        (A^{\alpha,\beta}_u,A^{\alpha,\beta}_u)_{\alpha,\beta}&=\frac{\beta^2}{\alpha^2}(A_u,A_u),\\
        (A^{\alpha,\beta}v,A^{\alpha,\beta}v)_{\alpha,\beta}&=\frac{\beta^4}{\alpha^2}(Av,Av),\\
        ((\check{\delta}A)^{\alpha,\beta}u,v)_{\alpha,\beta}&=\frac{\beta^2}{\alpha^2}((\check{\delta}A)u,v).
    \end{align*}
    The claim now follows from Lemmas \ref{L:CURV_FORM} and \ref{L:HESSIAN}, together with the formulae for the Ricci curvatures of the metric $g_{\alpha,\beta}$ in \cite[Proposition 9.36]{Be87}.
\end{proof}

In the case where the horizontal distribution is given by $TS^a$, we obtain that $g_{\alpha,\beta}$ is a product metric and the $A$-tensor vanishes. We therefore obtain the following as a consequence.

\begin{lemma}\label{L:Ricci_doubly_warped}
    Let $I$ be an interval and let $\alpha,\beta\colon I\to(0,\infty)$, $f\colon I\to\R$ be smooth functions. For $a,b\in\N$ define the metric $g_{\alpha,\beta}$ on $I\times S^a\times S^b$ by
    \[ h=dt^2+\alpha(t)^2ds_a^2+\beta(t)^2ds_b^2.\]
    We consider $f$ as a function defined on $I\times S^a\times S^b$ by composing it with the projection onto the first factor. Then, for $q\in(0,\infty]$, the weighted Ricci curvatures of the weighted Riemannian manifold $(I\times S^a\times S^b,h,e^{-f})$ are given by
    \begin{align*}
        \BE_q(\partial_t,\partial_t)&=-a\frac{\alpha''}{\alpha}-b\frac{\beta''}{\beta}+f''-\frac{1}{q}{f'}^2 ,\\
        \BE_q(\tfrac{u}{\alpha},\tfrac{u}{\alpha})&=-\frac{\alpha''}{\alpha}+(a-1)\frac{1-{\alpha'}^2}{\alpha^2}-b\frac{\alpha'\beta'}{\alpha\beta}+f'\frac{\alpha'}{\alpha},\\
        \BE_q(\tfrac{v}{\beta},\tfrac{v}{\beta})&=-\frac{\beta''}{\beta}+(b-1)\frac{1-{\beta'}^2}{\beta^2}-a\frac{\alpha'\beta'}{\alpha\beta}+f'\frac{\beta'}{\beta},\\
        \notag\BE_q(\partial_t,\tfrac{u}{\alpha})&=\BE_q(\partial_t,\tfrac{v}{\beta})=\BE_q(\tfrac{u}{\alpha},\tfrac{v}{\beta})=0.
    \end{align*}
    Here $u$ and $v$ are unit tangent vectors of $(S^a,ds_a^2)$ and $(S^b,ds_b^2)$, respectively.
\end{lemma}

Using Lemma \ref{L:Ricci_doubly_warped}, one can calculate the Ricci curvatures of a metric obtained by adding a third warping function.

\begin{lemma}\label{L:Ricci_triple_warped}
    Let $I_1,I_2$ be intervals and let $\alpha,\beta\colon I_1\times I_2\to(0,\infty)$ and $\gamma\colon I_1\to(0,\infty)$ be smooth functions. For $a,b\in\N$ define the metric $h$ on $I_1\times I_2\times S^a\times S^b$ by
    \[ h=dt^2+\gamma(t)^2ds^2+\alpha(t,s)^2ds_a^2+\beta(t,s)^2ds_b^2, \]
    where $I_1$ and $I_2$ are parametrized by $t$ and $s$, respectively. Then the Ricci curvatures of the metric $h$ are given as follows:
    \begin{align*}
        \Ric(\partial_t\;,\partial_t\;)&=-a\frac{\alpha_{tt}}{\alpha}-b\frac{\beta_{tt}}{\beta}-\frac{\gamma''}{\gamma},\\
        \Ric(\partial_t\;,\tfrac{\partial_s}{\gamma})&=\frac{1}{\gamma}\left( -a\frac{\alpha_{ts}}{\alpha}-b\frac{\beta_{ts}}{\beta}+a\frac{\alpha_s\gamma'}{\alpha\gamma}+b\frac{\beta_s\gamma'}{\beta\gamma} \right),\\
        \Ric(\tfrac{\partial_s}{\gamma},\tfrac{\partial_s}{\gamma})&=\frac{1}{\gamma^2}\left(-a\frac{\alpha_{ss}}{\alpha}-b\frac{\beta_{ss}}{\beta} \right)-\frac{\gamma''}{\gamma}-a\frac{\alpha_t\gamma'}{\alpha\gamma}-b\frac{\beta_t\gamma'}{\beta\gamma},\\
        \Ric(\tfrac{u}{\alpha}\;,\tfrac{u}{\alpha}\;)&= \frac{1}{\gamma^2}\left( -\frac{\alpha_{ss}}{\alpha}-(a-1)\frac{\alpha_s^2}{\alpha^2} -b\frac{\alpha_s\beta_s}{\alpha\beta}\right)-\frac{\alpha_{tt}}{\alpha}+(a-1)\frac{1-\alpha_t^2}{\alpha^2}-b\frac{\alpha_t\beta_t}{\alpha\beta}-\frac{\alpha_t\gamma'}{\alpha\gamma},\\
        \Ric(\tfrac{v}{\beta}\;,\tfrac{v}{\beta}\;)&=\frac{1}{\gamma^2}\left( -\frac{\beta_{ss}}{\beta}-(b-1)\frac{\beta_s^2}{\beta^2} -a\frac{\alpha_s\beta_s}{\alpha\beta}\right)-\frac{\beta_{tt}}{\beta}+(b-1)\frac{1-\beta_t^2}{\beta^2}-a\frac{\alpha_t\beta_t}{\alpha\beta}-\frac{\beta_t\gamma'}{\beta\gamma},\\
        \Ric(\partial_t\;,\tfrac{u}{\alpha}\;)&=\Ric(\partial_t,\tfrac{v}{\beta})=\Ric(\frac{\partial_s}{\gamma},\tfrac{u}{\alpha})=\Ric(\tfrac{\partial_s}{\gamma},\tfrac{v}{\beta})=\Ric(\tfrac{u}{\alpha},\tfrac{v}{\beta})=0.
    \end{align*}
    Here $u$ and $v$ are unit tangent vectors of $(S^a,ds_a^2)$ and $(S^b,ds_b^2)$, respectively.
\end{lemma}
\begin{proof}
    We write $h$ as $dt^2+g_t$ with
    \[ g_t=\gamma(t)^2\left(ds^2+\frac{\alpha(t,s)^2}{\gamma(t)^2}ds_a^2+\frac{\beta(t,s)^2}{\gamma(t)^2}ds_b^2\right).  \]
    Like in the proof of Lemma \ref{L:HESSIAN}, the Levi--Civita connection $\nabla^{t}$ of the metric $g_t$ is given by
    \begin{align*}
        \nabla^{t}_{\partial_s}\partial_s&=0,\\
        \nabla^{t}_{\partial_s}u&=\nabla^{t}_u\partial_s=\frac{\alpha_s}{\alpha}u,\\
        \nabla^{t}_{\partial_s}v&=\nabla^{t}_v\partial_s=\frac{\beta_s}{\beta}v,\\
        \nabla^{t}_u u'&=-\frac{\alpha_s\alpha}{\gamma^2}\langle u,u'\rangle_{S^a}+\nabla^{S^a}_u u',\\
        \nabla^{t}_v v'&=-\frac{\beta_s\beta}{\gamma^2}\langle v,v'\rangle_{S^b}+\nabla^{S^b}_v v',\\
        \nabla^{t}_u v&=0.
    \end{align*}
    Here $u,u'\in TS^a$ and $v,v'\in TS^b$.
    
    Since the Ricci tensor is invariant under scalar multiplication, we obtain the Ricci curvatures of the metric $g_t$ from Lemma \ref{L:Ricci_doubly_warped}. Inserting this into the formulae of Lemma \ref{L:CURV_FORM} then results in the Ricci curvatures of the metric $h$ as claimed.
\end{proof}

The following functions will be useful in the proof of Theorems \ref{T:CONN_SUMS} and \ref{T:HIGHER_SURG} to construct suitable doubly warped product metrics.
\begin{lemma}\label{L:warping_fcts}
    For any $\lambda\in(0,1)$, $\varepsilon,r>0$, $a>0$ and $b>1$, there exist $t_0>0$ and functions $\beta,\gamma\colon[0,t_0]\to(0,\infty)$ satisfying the differential inequalities
    \begin{align*}
        &-a\frac{\gamma''}{\gamma}-b\frac{\beta''}{\beta}>0,\\
        &-\frac{\beta''}{\beta}+(b-1)\frac{1-{\beta'}^2}{\beta^2}-a\frac{\gamma'\beta'}{\gamma\beta}>0,
    \end{align*}
    and the boundary conditions
    \begin{align*}
        \gamma(0)&=1, & \beta(0)&= r,\\
        \gamma'(0)&\leq \varepsilon, & \beta'(0)&= 0,\\
        \gamma'(t_0)&\geq 0, & \beta'(t_0)&\geq \lambda.
    \end{align*}
\end{lemma}
These functions are constructed in \cite[Sections 3.2 and 3.3]{Re23}. Notice that, compared to the corresponding statement in \cite{Re23}, we changed the notation and denoted $h$, $f$, $\lambda$ and $\cos(R/N)$ in \cite{Re23} by $\beta$, $\gamma$, $\varepsilon$ and $\lambda$, respectively. We also omitted conclusions (3.7) and (3.10) being positive in \cite{Re23}, as we do not need them here.

Finally, we construct a weighted metric of $\BE_\infty>0$ on a cylinder that transitions between two given metrics of positive Ricci curvature.
\begin{lemma}\label{L:isot_conc}
    Let $(g_t)_{t\in[0,1]}$ be a smoothly varying family of Riemannian metrics of positive Ricci curvature on a closed manifold $M$. Then for any $\lambda\in\R$ there exists a weighted Riemannian metric $(g,e^{-f})$ of $\BE_\infty>0$ on $[0,1]\times M$ such that
    \begin{enumerate}
        \item the induced metrics of $g$ on $\{0\}\times M$ and $\{1\}\times M$ are isometric to $g_0$ and $g_1$, respectively,
        \item the boundary components $\{0\}\times M$ and $\{1\}\times M$ are both totally geodesic, and
        \item the function $f$ is constant at both $t=0$ and $t=1$, and the normal derivative of $f$ at $\{1\}\times M$ is a constant $\lambda$, i.e.\ $df(\nu)=\lambda$ where $\nu$ is the outward unit normal.
    \end{enumerate}
\end{lemma}
\begin{proof}
    Since $M$ is compact, there exist $c,C>0$ so that for any $t\in[0,1]$ and any unit vectors $u,v\in TM$ with respect to $g_t$ all of the expressions
    \[ |\textrm{tr}_{g_t}g_t'|,\, |v(\textrm{tr}_{g_t}g_t')|,\,|(\nabla^{t}_{u}g_t')(v,u)|,\,|g_t'(u,v)|,\,|g_t''(u,v)|,\,|\textrm{tr}_{g_t}g_t''| \]
    are all bounded by $C$ and
    \[ \Ric^{g_t}(u,u)\geq c. \]
    Now for $a\in(0,1)$, let $\chi_a\colon[0,\frac{3}{a}]\to[0,1]$ be a smooth function with the following properties:
    \begin{enumerate}
        \item $\chi_a(0)=0$, $\chi_a(\frac{3}{a})=1$,
        \item $\chi_a'(0)=\chi_a'(\frac{3}{a})=0$,
        \item $|\chi_a'|,|\chi_a''|\leq a$.
    \end{enumerate}
    Such a function can for example be obtained by smoothing the $C^1$-function
    \[ \chi_a(t)=\begin{cases}
        \frac{a}{4}t^2,\quad & t\in[0,1],\\
        \frac{a}{2}t-\frac{a}{4},\quad &t\in[1,\frac{2}{a}],\\
        -\frac{a}{4}\left( t-\frac{2+a}{a} \right)^2+1,\quad &t\in[\frac{2}{a},\frac{2+a}{a}],\\
        1,\quad & t\in[\frac{2+a}{a},\frac{3}{a}].
    \end{cases} \]
    We now define the weighted metric $(g,e^{-f})$ on $[0,\frac{3}{a}]\times M$ by imposing
    \begin{align*}
        g&=dt^2+g_{\chi_a(t)},\\
        f'(t,x)&=2Ca\left(t-\frac{3}{a}\right)+\lambda,\quad f(0)=0.
    \end{align*}
    We have
    \[ \tfrac{\partial}{\partial t} g_{\chi_a(t)}=\chi_a'(t)g'_{\chi_a(t)},\quad \tfrac{\partial^2}{\partial t^2} g_{\chi_a(t)}=\chi_a'(t)^2 g''_{\chi_a(t)}+\chi_a''(t)g'_{\chi_a(t)}. \]
    Hence, by Lemmas \ref{L:CURV_FORM} and \ref{L:HESSIAN}, we can estimate the weighted Ricci curvatures of $(g,e^{-f})$ as follows:
    \begin{align*}
        \BE_\infty(\partial_t,\partial_t)&\geq -\frac{1}{2}C(a^2+a)+2Ca >Ca,\\
        |\BE_\infty(\partial_t,v)|&\leq \frac{n+1}{2}aC,\\
        \BE_\infty(v,v)&\geq c-\frac{1}{2}C(a^2+a)-\frac{1}{2}nC^2a^2-\frac{1}{4}C^2a^2-\frac{1}{2}Ca(6C+\lambda)\geq c-C'a,
    \end{align*}
    where $C'>0$ is a suitable constant (which depends on $\lambda$ and $C$, but not on $a$). Thus, the weighted Ricci curvatures are positive if and only if
    \[ Ca(c-C'a)>\left(\frac{n+1}{2}a C\right)^2, \]
    i.e.\ if and only if
    \[ c-C'a>\left(\frac{n+1}{2}\right)^2 aC, \]
    which is satisfied for all $a$ sufficiently small.

    Finally, we obtain a weighted metric on $[0,1]\times M$ by pulling back $(g,e^{-f})$ along a diffeomorphism $(t,x)\mapsto(\phi(t),x)$, where $\phi\colon [0,1]\to[0,\frac{3}{a}]$ is a diffeomorphism satisfying $\phi(1)=\frac{3}{a}$.
\end{proof}

\section{Perelman's Gluing Theorem}\label{S:gluing}

The goal of this section is to prove Theorem \ref{T:GLUING}. Before doing so, we first recall the corresponding result in the Riemannian case:

\begin{theorem}[{\cite{Pe97}, see also \cite{BWW19}, \cite{RW23}}]\label{T:gluing_Riem}
    Let $(M_1,h_1)$ and $(M_2,h_2)$ be two Riemannian $n$-manifolds of positive Ricci curvature and suppose that there exists an isometry $\phi\colon\partial_c M_1\to \partial_c M_2$ between two boundary components $\partial_c M_1\subseteq \partial M_1$ and $\partial_c M_2\subseteq \partial M_2$ such that $\II_{\partial_c M_1}+\phi^*\II_{\partial_c M_2}\geq 0$. Then there exists a Riemannian metric of positive Ricci curvature on $M_1\cup_\phi M_2$ that coincides with $h_i$ on $M_i$ outside an arbitrarily small neighbourhood of $\partial_c M_i$.
\end{theorem}

Theorem \ref{T:gluing_Riem} is proven by smoothing the $C^0$-metric $h_1\cup_\phi h_2$ on $M_1\cup_\phi M_2$ using spline interpolation. This is achieved in three main steps: First, a spline interpolation of degree 3 produces a $C^1$-regular metric, from which one then obtains a $C^2$-regular metric by a subsequent spline interpolation of degree 5. Finally, general smoothing results give a metric of $C^\infty$-regularity. The condition on the second fundamental forms ensures that the Ricci curvature is positive after the first step, while there is no additional assumption needed to preserve positive Ricci curvature in the other subsequent steps. 

To prove Theorem \ref{T:GLUING}, we will simultaneously smooth the metric and the weight function according to a similar smoothing process, mostly following the arguments presented in \cite[Section 2]{BWW19}. To simplify the proof we will combine the second and third step and construct a $C^\infty$-regular metric and function after having obtained $C^1$-regularity by using mollifying techniques as in \cite{RW23}.

We note that an alternative proof of Theorem \ref{T:gluing_Riem} was obtained by Schlichting \cite{Sc12}, which is based on a gluing result for Alexandrov spaces due to Kosovskii \cite{Ko02}.

In our setting, we start by considering two weighted Riemannian manifolds $(M_1,h_1,e^{-f_1})$ and $(M_2,h_2,e^{-f_2})$ and suppose that there exists a weighted isometry $\phi\colon \partial M_1\to\partial M_2$, i.e.\ $\phi$ is a diffeomorphism satisfying
\begin{itemize}
    \item $\phi^*h_2\big\vert_{\partial M_2}=h_1\big\vert_{\partial M_1}$ and 
    \item $f_2\circ\phi=f_1\big\vert_{\partial M_1}$.
\end{itemize}
At this point we do not impose any conditions on the second fundamental forms or weighted mean curvature.

As described in \cite[Section 2.1]{BWW19}, the glued space $W=M_1\cup_\phi M_2$ carries a canonical smooth structure so that the embeddings $M_i\hookrightarrow W$ are smooth. We will identify each $M_i$ with its image in $W$ and set $X=\partial M_1=\partial M_2$.

By the hypotheses on $h_i$ and $f_i$, we then obtain a continuous weighted Riemannian metric $(\hat{h},e^{-\hat{f}})$ on $W$ which is smooth on $W\setminus X$ and coincides with $(h_i,e^{-f_i})$ on $M_i$. Further, by considering normal coordinates around $X$, we obtain a diffeomorphism between a neighbourhood of $X$ in $M_i$ and the product $[0,\delta)\times X$ for some $\delta>0$ on which $h_i$ takes the form $h_i=dt^2+h_i(t)$, where $h_i(t)$ is a smoothly varying family of Riemannian metrics on $X$. Hence, on a tubular neighbourhood of $X$ in $W$, which we identify with $(-\delta,\delta)\times X$, the weighted $C^0$-metric $(\hat{h},e^{-\hat{f}})$ takes the form $(dt^2+\hat{g}(t),e^{-\hat{f}_t})$, where $(\hat{g}(t),\hat{f}_t)=(h_1(-t),f_1(-t,\cdot))$ resp.\ $(h_2(t),f_2(t,\cdot))$ for $t\leq 0$ resp.\ $t\geq 0$.

Our aim is now to smooth the weighted Riemannian metric $(\hat{h},e^{-\hat{f}})$ on $(-\delta,\delta)\times X$ while preserving $\BE_q>0$. For that, we first consider a $C^1$-smoothing. As in Subsection \ref{SS:cylinder}, an added $'$ will denote a derivative in the $t$-direction.

\begin{lemma}\label{Lemma 3.2}
    Consider a pair of smooth $n$-dimensional weighted Riemannian manifolds $(M_1,h_1,e^{-f_1})$ and $(M_2,h_2,e^{-f_2})$ as above. For all $\varepsilon>0$ sufficiently small, there is a $C^1$-regular metric $\tilde{h}$ and a $C^1$-regular weight function $\tilde{f}$ over the smooth glued manifold $W=M_1\cup_\phi M_2$ that differ only inside an $\varepsilon$-neighbourhood of $X$ from the continuous $\hat{h}$ and $\hat{f}$. Moreover, on the $\varepsilon$-neighbourhood of $X$, which we identify with $[-\varepsilon,\varepsilon]\times X$, the metric $\tilde{h}$ takes the form $dt^2+\tilde{g}(t)$, and we have the following:
    \begin{enumerate}
        \item $\tilde{g}(t)$ and $\tilde{f}$ converge pointwise to $\hat{g}(t)$ and $\hat{f}$, respectively, as $\varepsilon\to 0$,
        \item The first $t$-derivatives $\tilde{g}'(t)$ and $\tilde{f}'$ are uniformly bounded independent of $\varepsilon$,
        \item The second $t$-derivatives $\tilde{g}''(t)$ and $\tilde{f}''$ are linear in $t$ and satisfy
        \[ \varepsilon\tilde{g}''(\pm\varepsilon)\to\frac{1}{2}\left(h_2'(0)-h_1'(0)  \right)\quad \text{ and }\quad \varepsilon \tilde{f}''(\pm\varepsilon,\cdot)\to\frac{1}{2}\left( f_2'(0,\cdot)-f_1'(0,\cdot) \right) \]
        as $\varepsilon\to0$.
    \end{enumerate} 
\end{lemma}
\begin{proof}


Recall that we identified a tubular neighbourhood of $X$ in $W$ with $(-\delta,\delta)\times X$. Our aim is to construct a new metric $\tilde{h}$ on $[-\varepsilon,\varepsilon]\times X$ of the form $dt^2+\tilde{g}(t)$ which joins with $h_1$ for $t\le -\varepsilon$, and with $h_2$ for $t\ge \varepsilon$. Likewise, we want to construct a new weight function $\tilde{f}$ on $[-\varepsilon,\varepsilon]\times X$ which joins with $f_1$ for $t\le-\varepsilon$, and with $f_2$ for $t\ge \varepsilon$. As pointed out in the claim, we require both the metric $\tilde{h}$ and the function $\tilde{f}$ be at least $C^1$-regular on $W$.

The construction of the metric $\tilde{h}$ via a cubic spline interpolation is carried out in \cite[Lemma 3 and subsequent paragraph]{BWW19}, see also \cite[Section 2.1]{RW23b}. We are therefore left with constructing the function $\tilde{f}$. Here we take a similar approach and define $\tilde{f}$ as follows.
Let us denote by $f_i(t)=f_i(t,\cdot)$ the induced weight function of $f_i$ on the hypersurface $\lbrace t\rbrace\times X$ at constant distance $t\in [-\varepsilon,\varepsilon]$ from $X$, with $i=1$ for $t\leq 0$ and $i=2$ for $t\geq 0$. We then get our desired weight function $\tilde{f}(t):=\tilde{f}(t,\cdot)$ via a cubic spline interpolation between $f_1(-\varepsilon)$ and $f_2(\varepsilon)$, that is
\begin{align*}
    \tilde{f}(t)=&\frac{t+\varepsilon}{2\varepsilon}f_2(\varepsilon)-\frac{t-\varepsilon}{2\varepsilon}f_1(-\varepsilon)+\frac{(t-\varepsilon)^2(t+\varepsilon)}{4\varepsilon^2}\left[f_1'(-\varepsilon)-\frac{f_2(\varepsilon)-f_1(-\varepsilon)}{2\varepsilon}\right]+\\&+\frac{(t+\varepsilon)^2(t-\varepsilon)}{4\varepsilon^2}\left[f_2'(\varepsilon)-\frac{f_2(\varepsilon)-f_1(-\varepsilon)}{2\varepsilon}\right]\,.
\end{align*}
The $t$-derivative of this function is then
\begin{align*}
    \tilde{f}'(t)=&\frac{f_2(\varepsilon)-f_1(-\varepsilon)}{2\varepsilon}+\frac{2(t^2-\varepsilon^2)+(t-\varepsilon)^2}{4\varepsilon^2}\left[f_1'(-\varepsilon)-\frac{f_2(\varepsilon)-f_1(-\varepsilon)}{2\varepsilon}\right]+\\&+\frac{2(t^2-\varepsilon^2)+(t+\varepsilon)^2}{4\varepsilon^2}\left[f_2'(\varepsilon)-\frac{f_2(\varepsilon)-f_1(-\varepsilon)}{2\varepsilon}\right]\,,
\end{align*}
and it is straightforward to check that the weight function $\tilde{f}$ forms a $C^1$-join with the $f_i$ at $t=\pm\varepsilon$ and converges to $\hat{f}$ as $\varepsilon\to 0$.

We are now left to investigate the $t$-derivatives of $\tilde{f}$. In particular, we are interested in the limiting behaviour of $\tilde{f}(\pm\varepsilon)$ as $\varepsilon\to 0$. 
By differentiating twice along $t$, we get the expression
\begin{align*}
    \tilde{f}''(t)=\frac{6t-2\varepsilon}{4\varepsilon^2}\left[f_1'(-\varepsilon)-\frac{f_2(\varepsilon)-f_1(-\varepsilon)}{2\varepsilon}\right]+\frac{6t+2\varepsilon}{4\varepsilon^2}\left[f_2'(\varepsilon)-\frac{f_2(\varepsilon)-f_1(-\varepsilon)}{2\varepsilon}\right]\,.
\end{align*}
Applying de l'H\^opital, we have
that the limiting behaviour of $\tilde{f}''(\pm\varepsilon)$ as $\varepsilon\to 0$  is given by
\begin{align*}
    \varepsilon\cdot\tilde{f}''(\pm\varepsilon)=&\pm\frac{3}{2}\left[f_1'(-\varepsilon)+f_2'(\varepsilon )-\frac{f_2(\varepsilon)-f_1(\varepsilon)}{\varepsilon}\right]+\frac{1}{2}\left[f_2'(\varepsilon)-f_1'(-\varepsilon)\right]\\&\longrightarrow\frac{1}{2}\left[f_2'(0)-f_1'(0)\right]\,
\end{align*}
as $\varepsilon\to 0$. Finally, using the explicit formulae above, it is a straightforward computation to check that $\tilde{f}'$ is uniformly bounded. 
\end{proof}
\begin{lemma}\label{Lemma 3.3}
    Working under the same setting as in Lemma \ref{Lemma 3.2}, if we further assume conditions 1) and 2) of Theorem \ref{T:GLUING} in a strict sense, i.e.\ that     \begin{enumerate}
        \item $H_{\partial M_1}^{f_1}+H_{\partial M_2}^{f_2}\circ\phi> 0$, and
        \item $\II_{\partial M_1}+\phi^*\II_{\partial M_2}> 0$,
    \end{enumerate} then for any $A>0$ there exists $\hat\varepsilon=\hat\varepsilon(A,h_1,h_2,f_1,f_2)>0$ such that for any $\varepsilon<\hat{\varepsilon}$
    \begin{align*}
        \tilde{g}''(t)(u,u)<-A\cdot \tilde{g}(t)(u,u)\ \text{ and }\tfrac{1}{2}\mathrm{tr}_{\tilde{g}(t)}\tilde{g}''(t)-\tilde{f}''(t)<-A\ ,
    \end{align*}
    for all $t\in[-\varepsilon,\varepsilon]$ and all tangent vectors $u$ tangent to $\{t\}\times X$.
\end{lemma}
\begin{proof}
Let us consider the hypersurface $\lbrace t\rbrace\times X$ obtained by slicing the collar neighbourhood of $X$ at any $t\in [-\varepsilon,\varepsilon]$. Applying Lemma \ref{L:CURV_FORM} to the weighted mean curvatures at $\partial M_1$ and $\partial M_2$ with respect to their outward normal directions (see Definition \ref{weighted mean curv}), we have
\begin{align*}
    H^{f_1}=\frac{1}{2}\mathrm{tr}_{h_1(0)}h_1'(0)-f_1'(0),\text{ and }H^{f_2}=-\frac{1}{2}\mathrm{tr}_{h_2(0)}h_1'(0)+f_2'(0)\,.
\end{align*}
By (1) , we have that $H^{f_1}+H^{f_2}>0$, and therefore the limit of $\varepsilon\cdot\big[\tfrac{1}{2}\mathrm{tr}_{\tilde{g}(\pm\varepsilon)}\tilde{g}''(\pm\varepsilon)-\tilde{f}''(\pm\varepsilon)\big]$ as $\varepsilon\to 0$ is negative by Lemma \ref{Lemma 3.2}. Since the second $t$-derivatives of $\tilde{g}(t)$ and $\tilde{f}(t)$ are linear in $t$, we obtain the required bound.

Finally, in order to obtain that
\begin{align*}
    \lim_{\varepsilon\to 0}\varepsilon\cdot \tilde{g}''(\pm \varepsilon)<0\,,
\end{align*}
one can proceed in a similar fashion as before, by using that $\II_{\partial_c M_1}+\phi^*\II_{\partial_c M_2}> 0$ by assumption (2) (see also \cite[Lemma 4]{BWW19} for an explicit proof of the result).
\end{proof}

After having established Lemmas \ref{Lemma 3.2} and \ref{Lemma 3.3} we can now prove Theorem \ref{T:GLUING}. As indicated above, we proceed in two steps.

\begin{proposition}[From continuous to $C^1$]\label{P: C1 mfld positive weighted} Let $(M_1,h_1,e^{-f_1})$ and $(M_2,h_2,e^{-f_2})$ be two weighted Riemannian manifolds with $\BE_q>0$ for some $q\in(0,\infty]$, and suppose there exists an isometry $\phi\colon \partial_c M_1\to\partial_c M_2$ between two boundary components $\partial_c M_1\subseteq \partial M_1$ and $\partial_c M_2\subseteq \partial M_2$ such that $f_1|_{\partial_c M_1}=f_2\circ\phi$. If
    \begin{enumerate}
        \item $H_{\partial_c M_1}^{f_1}+H_{\partial_c M_2}^{f_2}\circ\phi\geq 0$, and
        \item $\II_{\partial_c M_1}+\phi^*\II_{\partial_c M_2}\geq 0$,
    \end{enumerate}
    then for $\varepsilon>0$ sufficiently small the weighted $C^1$-metric $(\tilde{h},e^{-\tilde{f}})$ defined in Lemma \ref{Lemma 3.2} has $\BE_q>0$.
\end{proposition}
\begin{proof}
    We first slightly deform the metric and weight function on one of $M_1$ and $M_2$ while preserving $\BE_q>0$, so that inequalities (1) and (2) hold strictly (e.g.\ as in \cite[Proposition 1.2.11]{Bu19a}).

    By Lemma \ref{Lemma 3.2} it suffices now to consider the $\varepsilon$-neighbourhood $[-\varepsilon,\varepsilon]\times X$ of $X$. Then, by Lemma \ref{Lemma 3.3}, we can bound the values of $\tilde{g}''(t)$ and $\tfrac{1}{2}\mathrm{tr}_{\tilde{g}(t)}\tilde{g}''(t)-\tilde{f}''(t)$ from above by any negative value by choosing $\varepsilon$ sufficiently small. Further, by Lemma \ref{Lemma 3.2}, all first order $t$-derivatives of $\tilde{g}(t)$ and $\tilde{f}(t)$ are bounded independently of $\varepsilon$.
    
    Hence, by Lemmas \ref{L:CURV_FORM} and \ref{L:HESSIAN}, all terms in $\BE_q(\partial_t,\partial_t)$ and $\BE_q(u,u)$, where $u$ is tangent to $X$, that contain a second order $t$-derivative can be bounded below by any positive constant, while all other terms, as well as the mixed curvatures $\BE_q(\partial_t,u)$ are bounded independently of $\varepsilon$. Hence, for $\varepsilon$ sufficiently small, we have $\BE_q>0$.
\end{proof}
 \begin{proposition}[From $C^1$ to smooth] Let $(M_1,h_1,e^{-f_1})$ and $(M_2,h_2,e^{-f_2})$ be two weighted Riemannian manifolds with $\BE_q>0$ for some $q\in(0,\infty]$, and suppose there exists an isometry $\phi\colon \partial_c M_1\to\partial_c M_2$ between two boundary components $\partial_c M_1\subseteq \partial M_1$ and $\partial_c M_2\subseteq \partial M_2$ such that $f_1|_{\partial_c M_1}=f_2\circ\phi$. If
    \begin{enumerate}
        \item $H_{\partial_c M_1}^{f_1}+H_{\partial_c M_2}^{f_2}\circ\phi\geq 0$, and
        \item $\II_{\partial_c M_1}+\phi^*\II_{\partial_c M_2}\geq 0$,
    \end{enumerate}
    then it is possible to endow the glued manifold $M_1\cup_\phi M_2$ with a structure of a smooth weighted Riemannian manifold of $\BE_q>0$ that differs from the weighted $C^0$-metric $(\hat{h},e^{-\hat{f}})$ only in an arbitrarily small neighbourhood of the gluing area. 
\end{proposition} 
\begin{proof}
    The statement follows directly once we apply the smoothing result of \cite[Lemma 3.1]{RW23b} to the $C^1$-regular weight function and the metric of the glued manifold $(M\cup_\phi M_2,\tilde{h},e^{-\tilde{f}})$ constructed in Proposition \ref{P: C1 mfld positive weighted}.
\end{proof}

This finishes the proof of Theorem \ref{T:GLUING}.

Finally, we consider the special case of a doubly warped product metric.

\begin{corollary}\label{C:gluing_warped}
    Let $\alpha,\beta\colon I\to(0,\infty)$, $f\colon I\to\R$ be continuous functions that are smooth except at finitely many points $t_1,\dots,t_\ell\in I$. Suppose that for the metric
    \[ g=dt^2+\alpha(t)^2ds_a^2+\beta(t)^2ds_b^2 \]
    on $I\times S^a\times S^b$ the weighted Riemannian metric $(g,e^{-f})$ has $\BE_q>0$. Then we can smooth the functions $\alpha,\beta,f$ in arbitrarily small neighbourhoods of each $t_i$ and obtain a smooth weighted Riemannian metric of $\BE_q>0$, provided that for each $i\in\{1,\dots,\ell\}$ the following inequalities are satisfied:
    \begin{align*}
        \alpha'_-(t_i)&\geq \alpha'_+(t_i),\\
        \beta'_-(t_i)&\geq \beta'_+(t_i),\\
        a\frac{\alpha'_-(t_i)}{\alpha(t_i)}+b\frac{\beta'_-(t_i)}{\beta(t_i)}-f'_-(t_i)&\geq a\frac{\alpha'_+(t_i)}{\alpha(t_i)}+b\frac{\beta'_+(t_i)}{\beta(t_i)}-f'_+(t_i).
    \end{align*}
\end{corollary}
\begin{proof}
    This directly follows from Theorem \ref{T:GLUING} and Lemmas \ref{L:CURV_FORM} and \ref{L:HESSIAN}. Note that from the explicit form of the metric $\tilde{h}$ and weight function $\tilde{f}$ in the proof of Theorem \ref{T:GLUING}, it follows that the smoothed metric is again a doubly warped product metric and the weight function only depends on $t$.
\end{proof}

\section{Connected Sums}
\label{S:conn_sums}

In this section, we consider connected sums in the presence of weighted metrics of $\BE_q>0$ and prove Theorem \ref{T:CONN_SUMS}. First, let us recall the techniques used in the Riemannian case. 

\begin{definition}[{\cite{Bu19}, based on \cite{Pe97}}]\label{D:core}
    A Riemannian metric $g$ on an $n$-dimensional manifold $M$ is called a \emph{core metric}, if there exists an isometric embedding $\varphi\colon D^n\hookrightarrow M$, where we consider $D^n$ being equipped with the induced metric of a hemisphere in the round sphere of radius $1$.
\end{definition}
We note that this definition slightly differs from the definition introduced in \cite{Bu19}. However, a core metric in the sense of Definition \ref{D:core} can always be deformed into a core metric in the sense of \cite{Bu19} and vice versa, see \cite[Lemma 2.4]{Re24}.

Core metrics are of interest in the context of connected sums due to the following theorem, which is a consequence of Perelman's gluing theorem (Theorem \ref{T:gluing_Riem}) together with the existence of a metric of positive Ricci curvature on $S^n\setminus (\sqcup_\ell D^{n})^\circ$ with small second fundamental form on each boundary component.

\begin{theorem}[{\cite{Pe97},\cite[Theorem B]{Bu19}}]\label{T:CORE_CONN_SUM}
    Let $M_1^n,\dots,M_\ell^n$ be manifolds with $n\geq 4$ that admit core metrics. Then the connected sum $M_1\#\dots\# M_\ell$ admits a Riemannian metric of positive Ricci curvature.
\end{theorem}

An immediate consequence is that a closed manifold with a core metric is simply-connected. This follows from Theorem \ref{T:CORE_CONN_SUM} in combination with the theorem of Bonnet--Myers, or alternatively from a result of Lawson \cite[Theorem 1]{La70}, which only requires that the mean curvature of $\varphi(S^{n-1})$ is non-negative.

Core metrics are difficult to construct in general. To the best of our knowledge, the following are all the known examples of manifolds admitting a core metric:
\begin{enumerate}
    \item[\listlabel{EQ:core1}{(C1)}] the sphere $S^n$ and the compact rank one symmetric spaces $\C P^n$, $\mathbb{H}P^n$ and $\mathbb{O}P^2$ (see \cite{Pe97},\cite{Bu19}),
    \item[\listlabel{EQ:core2}{(C2)}] linear sphere bundles and projective bundles with fibre $\C P^n$, $\mathbb{H}P^n$ or $\mathbb{O}P^2$ over manifolds with core metrics (see \cite{Bu20},\cite{Re23},\cite{Re24}),
    \item[\listlabel{EQ:core3}{(C3)}] products of manifolds with core metrics (see \cite{Re24}),
    \item[\listlabel{EQ:core4}{(C4)}] connected sums of manifolds with core metrics (see \cite{Bu20a}),
    \item[\listlabel{EQ:core5}{(C5)}] manifolds obtained as boundaries of certain plumbings (see \cite{Bu19a}),
    \item[\listlabel{EQ:core6}{(C6)}] certain manifolds that decompose as the union of two disc bundles, such as the Wu manifold $W^5$ (see \cite{Re24}).
\end{enumerate}

A natural generalisation of core metrics are weighted core metrics (Definition \ref{D:weighted_core}). By using Theorem \ref{T:GLUING}, we have the following equivalent characterisations.
\begin{lemma}\label{L:weighted_core_equ}
    Let $M^n$ be a manifold and let $q\in(0,\infty]$. Then the following are equivalent.
    \begin{enumerate}
        \item $M$ admits a weighted core metric with respect to $q$.
        \item $M$ admits a weighted metric of $\BE_q>0$ and an embedding $D^n\subseteq M$ such that on the boundary $\partial (M\setminus{D^n}^\circ)$ we have the following:
        \begin{enumerate}
            \item the induced metric is round and the weight function is constant;
            \item the second fundamental form and the weighted mean curvature are positive.
        \end{enumerate}
        \item $M$ admits a weighted metric of $\BE_q>0$ and an embedding $D^n\subseteq M$ such that on the boundary $\partial (M\setminus{D^n}^\circ)$ we have the following:
        \begin{enumerate}
            \item the induced metric is round and the weight function is constant;
            \item the second fundamental form and the weighted mean curvature are non-negative.
        \end{enumerate}
    \end{enumerate}
\end{lemma}
\begin{proof}
    The proof goes along the same lines as the proof of \cite[Lemma 2.18]{Re24} by using Theorem \ref{T:GLUING} instead of Theorem \ref{T:gluing_Riem}.
\end{proof}

It is clear that a core metric defines a weighted core metric with respect to any $q\in(0,\infty]$ by choosing a constant weight function. Furthermore, just like for core metrics, closed manifolds with a weighted core metric are simply-connected. This follows from a result of Moore--Woolgar \cite[Theorem 1.5]{MW23}, which only requires the weighted mean curvature on the boundary be non-negative, or, alternatively, from Theorem \ref{T:CONN_SUMS} in combination with Proposition \ref{P:BE_TOP}. A further connection between core metrics and weighted core metrics is given in Proposition \ref{P:WEIGHTED_CORE_BDL} below.

By adapting the proof of Theorem \ref{T:CORE_CONN_SUM} using Theorem \ref{T:GLUING} instead of the gluing theorem for positive Ricci curvature, we could directly generalise it to the weighted setting and show that the connected sum of manifolds with weighted core metrics with respect to $q$ admits a weighted metric of $\BE_q>0$. Note, however, that Theorem \ref{T:CONN_SUMS} is more general as it allows one additional summand that does not need to admit a weighted core metric.

The main ingredient in the proof of Theorem \ref{T:CONN_SUMS} is the following proposition.

\begin{proposition}\label{P:ALMOST_TOT_GEOD}
    Let $(M^n,g,e^{-f})$ be a weighted Riemannian manifold of $\BE_q>0$ for some $q\in(0,\infty]$ and let $x\in M$ with $\nabla f_x=0$. Then for any $\varepsilon>0$ and any $r>0$ sufficiently small there exists a weighted Riemannian metric $(g',e^{-f'})$ of $\BE_q>0$ on $M\setminus {B_{\frac{r}{2}}(x)}^\circ$ which, up to a positive constant factor, coincides with $(g,e^{-f})$ on $M\setminus {B_{r}(x)}^\circ$ and such that
    \begin{enumerate}
        \item the induced metric of $g'$ on the boundary $\partial B_{\frac{r}{2}}(x)\cong S^{n-1}$ is given by $R^2\cdot ds_{n-1}^2$ for some $R>0$;
        \item the principal curvatures of $g'$ at $\partial B_{\frac{r}{2}}(x)$ with respect to the inward normal of $\partial B_{\frac{r}{2}}(x)\subseteq B_{\frac{r}{2}}(x)$ are bounded from below by $-\frac{\varepsilon}{R}$,
        \item the weight function $f'$ is constant on $\partial B_{\frac{r}{2}}(x)$ and has vanishing normal derivative.
    \end{enumerate}
\end{proposition}
The balls $B_r(x)$ and $B_{\frac{r}{2}}(x)$ used in Proposition \ref{P:ALMOST_TOT_GEOD} are determined using $g$. The weighted metric $(g',e^{-f'})$ we will construct in the proof will differ drastically from $(g,e^{-f})$ on $B_r(x)\setminus B_{\frac{r}{2}}(x)^\circ$. In fact, the diameter of $B_r(x)\setminus B_{\frac{r}{2}}(x)^\circ$ with respect to the metric $g'$ goes to $\infty$ as $\varepsilon\to 0$.

Given Proposition \ref{P:ALMOST_TOT_GEOD}, we can now prove Theorem \ref{T:CONN_SUMS}.
\begin{proof}[Proof of Theorem \ref{T:CONN_SUMS}]
    We assume $\ell=1$. The statement for general $\ell$ then follows inductively by applying $\ell$-times the result for $\ell=1$. We denote by $(g_0,e^{-f_0})$ the weighted metric of $\BE_q>0$ on $M_0$ and by $(g_1,e^{-f_1})$ the weighted core metric on $M_1$.

    Since $M_0$ is closed, there exists a point $x\in M_0$ with $\nabla {f_1}_x=0 $. Hence, we can apply Proposition \ref{P:ALMOST_TOT_GEOD} and, after rescaling, obtain for any $\varepsilon>0$ a weighted metric of $\BE_q>0$ on $M_0\setminus {D^n}^\circ$ such that the boundary is round with principal curvatures bounded from below by $-\varepsilon$ and such that the weight function is constant on the boundary with vanishing normal derivatives.

    To glue this weighted metric to $M_1\setminus{D^n}^\circ$, we apply Lemma \ref{L:weighted_core_equ} to obtain a weighted metric of $\BE_q>0$ on $M_1\setminus{D^n}^\circ$ with round boundary and strictly positive second fundamental form and weighted mean curvature, and such that the weight function is constant on the boundary. Hence, by Theorem \ref{T:GLUING}, after shifting one of the weight functions by a suitable constant, for $\varepsilon$ sufficiently small we can glue $M_0\setminus{D^n}^\circ$ to $M_1\setminus{D^n}^\circ$ along the boundary and obtain a weighted metric of $\BE_q>0$ on the connected sum $M_0\# M_1$.
\end{proof}

It now remains to prove Proposition \ref{P:ALMOST_TOT_GEOD}. The first step in the proof is to deform the weighted metric locally around the point $x$ with $\nabla f_x=0$ into a weighted metric with constant weight function and constant sectional curvature equal to $1$. This is a direct adaptation of the corresponding statement for positive Ricci curvature first established in \cite{GY86} (where it is shown for negative Ricci curvature, but the arguments work similarly for positive Ricci curvature), see also \cite[Theorem 1.10]{Wr02} and \cite[Lemma 4.3]{RW23a}. We follow the line of arguments given in \cite[Lemma 4.3 and Corollary 4.4]{RW23a} and begin by establishing the following more general deformation result first.

\begin{lemma}\label{L:LOCAL_DEFORM}
    Let $(M^n,g_0,e^{-f_0})$ be a weighted Riemannian manifold of $\BE_q>0$ and let $N^p\subseteq M$ be a compact embedded submanifold. Let $(g_1,e^{-f_1})$ be a weighted metric of $\BE_q>0$ defined on an open neighbourhood $U$ of $N$. If the $1$-jets of $g_0$ and $g_1$ and the $1$-jets of $f_0$ and $f_1$ coincide on $N$, then there exists a weighted metric $(\bar{g},e^{-\bar{f}})$ of $\BE_q>0$ on $M$ that coincides with $(g_0,e^{-f_0})$ on $M\setminus U$ and coincides with $(g_1,e^{-f_1})$ on an open neighbourhood of $N$ (which is contained in $U$).
\end{lemma}
\begin{proof}
    Consider for $t\in[0,1]$ the weighted metric $(g_t,e^{-f_t})$ on $U$ defined by
    \begin{align*}
        g_t&=(1-t)g_0+tg_1,\\
        f_t&=(1-t)f_0+tf_1.
    \end{align*}
    Since the $1$-jets of $(g_0,e^{-f_0})$ and $(g_1,e^{-f_1})$ coincide on $N$ and the sectional curvature (and therefore also the Ricci curvatures) and the hessian depend linearly on the second derivatives of the metric and the weight function respectively, we have that the weighted Ricci curvature $\BE_{q}^{g_t,f_t}$ on $N$ is given by
    \[ \BE_{q}^{g_t,f_t}=(1-t)\BE_{q}^{g_0,f_0}+t\BE_{q}^{g_1,f_1}. \]
    In particular, $\BE_{q}^{g_t,f_t}>0$ on $N$ and, by compactness, this also holds in a small neighbourhood of $N$. By local flexibility, see \cite[Theorem 1.2]{BH22}, we can extend $(g_t,e^{-f_t})$ to a global deformation of $(g_0,e^{-f_0})$ which is constant on $M\setminus U$ and coincides with $(g_t,e^{-f_t})$ on a neighbourhood of $N$.
\end{proof}
The special case where $N$ is $0$-dimensional gives the following consequence:
\begin{corollary}\label{C:LOCAL_DEFORM}
    Let $(M^n,g,e^{-f})$ be a weighted Riemannian manifold of $\BE_q>0$ and let $x\in M$ with $\nabla f_x=0$. Then, for any open neighbourhood $U$ of $x$, the weighted metric $(g,e^{-f})$ can be deformed into a weighted metric of $\BE_q>0$ that coincides with $(g,e^{-f})$ on $M\setminus U$ and has constant weight function and constant sectional curvature $1$ on a small neighbourhood of $x$.
\end{corollary}
\begin{proof}
    By considering normal coordinates around $x$, we can write the metric $g$ locally as $g_{ij}=\delta_{ij}+O(r^2)$, where $r$ denotes the distance from $x$. In particular, the $1$-jets of $g$ coincide with the $1$-jets of the round metric of radius $1$ on a sphere in normal coordinates. Furthermore, since $\nabla f_x=0$, the $1$-jets of $f$ at $x$ coincide with the $1$-jets of a constant function. Hence, by Lemma \ref{L:LOCAL_DEFORM} where we set $N=\{x\}$, we can deform $(g,e^{-f})$ as required.
\end{proof}

\begin{proof}[Proof of Proposition \ref{P:ALMOST_TOT_GEOD}]
    We apply Corollary \ref{C:LOCAL_DEFORM} to deform the weighted metric in a small neighbourhood of $x$ to have constant sectional curvature $1$ and constant weight function, that is, we can write $g$ on $B_r(x)$ for some $r>0$ sufficiently small as the warped product
    \[ g=dt^2+\sin^2(t)ds_{n-1}^2. \]
    Here we identified $B_r(x)$ with the space obtained from $[0,r]\times S^{n-1}$ by collapsing $\{0\}\times S^{n-1}$ to a point. We set $\lambda=\cos(r)$. We will now modify the weighted metric on this part to satisfy the required conditions.

    To do so, we consider for some $t_0>0$ a weighted metric $(g_{\beta},e^{-f})$ on $[0,t_0]\times S^{n-1}$ of the form
    \[ g_{\beta}=dt^2+\beta(t)^2ds_{n-1}^2,\quad f=-q\ln(\gamma(t)) \]
    where $\beta,\gamma\colon[0,t_0]\to(0,\infty)$ are two smooth functions. By Lemma \ref{L:Ricci_doubly_warped}, the weighted Ricci curvatures for $(g_{\beta},e^{-f})$ are given as follows (here we set $\alpha$ to be constant):
    \begin{align*}
        \BE_q(\partial_t,\partial_t)&=-(n-1)\frac{\beta''}{\beta}-q\frac{\gamma''}{\gamma},\\
        \BE_q(\tfrac{v}{\beta},\tfrac{v}{\beta})&=-\frac{\beta''}{\beta}+(n-2)\frac{1-{\beta'}^2}{\beta^2}-q\frac{\beta'\gamma'}{\beta\gamma},\\
        \BE_q(\partial_t,\tfrac{v}{\beta})&=0,
    \end{align*}
    where $v$ is a unit tangent vector of $(S^{n-1},ds_{n-1}^2)$.
    
    The second fundamental form of the hypersurface $\{t\}\times S^{n-1}$ with respect to the metric metric $g_\beta$ and the unit normal $\partial_t$ is given by
    \[ \II(\tfrac{u}{\beta},\tfrac{u}{\beta})=\frac{\beta'}{\beta}, \]
    see Lemma \ref{L:CURV_FORM}. Hence, to glue the weighted metric $(g_\beta,e^{-f})$ using Theorem \ref{T:GLUING}, with the weighted metric on $M\setminus B_r(x)$ the following boundary conditions at $t=t_0$ are sufficient:
    \begin{align*}
        \beta'(t_0)&\geq \lambda,\\
        \gamma'(t_0)&\geq 0.
    \end{align*}
    Note that we do not need to prescribe the value of $\beta$ at $t=t_0$, as we can globally rescale the metric (i.e.\ replacing $t_0,\beta, \gamma$ by $\mu t_0, \mu\beta\left(\frac{\cdot}{\mu}\right),\mu\gamma\left(\frac{\cdot}{\mu}\right)$, respectively) to satisfy this condition. Furthermore, the value of $\gamma$ at $t=t_0$ also does not need to be prescribed as we can always add a constant function to $f$.
    
    For the required conditions (1)--(3) to hold, the following boundary conditions at $t=0$ need to be satisfied:
    \begin{align*}
        \beta'(0)&\leq \varepsilon,\\
        \gamma'(0)&=0.
    \end{align*}

    We start by setting
    \[ \beta(t)=N\cos\left(\frac{t-t'}{N} \right) \]
    on $[0,t']$, where $N,t'>0$ are chosen so that $\beta(t')=1$ and $\beta'(0)=\varepsilon$, and we choose $t'$ as the smallest such value (to ensure $\beta>0$). In particular, we have $\beta'(t')=0$. If we define $\gamma$ to be constant on $[0,t']$, it is easily verified that the weighted Ricci curvatures are positive. Hence, the same holds if we slightly perturb $\gamma$ to have vanishing derivative at $t=0$ and strictly positive derivative at $t=t'$.

    Next, we extend the functions $\beta$, $\gamma$ by the functions obtained in Lemma \ref{L:warping_fcts} (which we need to shift by $t'$), where the parameters $a,b,\lambda,\varepsilon,r$ in Lemma \ref{L:warping_fcts} are set as $q,(n-1),\lambda, \gamma'(t'), \gamma(t')$, respectively. It then follows from Lemma \ref{L:warping_fcts} that the weighted Ricci curvatures are positive and the boundary conditions at $t=t_0$ are satisfied.
\end{proof}

\section{Higher Surgeries}\label{S:higher_surg}

In this section, we prove Theorem \ref{T:HIGHER_SURG}. We assume that $(M^n,g,e^{-f})$ is a weighted Riemannian manifold of $\BE_q>0$ and $\varphi\colon S^a\times D^{b+1}\hookrightarrow M$ is an embedding with $a+b+1=n$. As a first step, we deform the metric and weight functions near the image of $\varphi$ into a standard form.

\begin{proposition}\label{P:deform_tot_geod}
    Suppose that $\varphi(S^a\times\{0\})$ is round and totally geodesic and the weight function $f$ is constant on $\varphi(S^a\times \{0\})$ with vanishing normal derivatives. Then there exists a weighted Riemannian metric $(\bar{g},e^{-\bar{f}})$ of $\BE_q>0$ on $M$ and an embedding $\bar{\varphi}\colon S^a\times D^{b+1}\hookrightarrow M$ isotopic to $\varphi$ such that the following holds:
    \begin{enumerate}
        \item The weighted metric $(\bar{g},e^{-\bar{f}})$ coincides with $(g,e^{-f})$ outside an arbitrarily small neighbourhood of $\varphi(S^a\times\{0\})$,
        \item The pull-back $\bar{\varphi}^*\bar{g}$ on $S^a\times D^{b+1}$ is a metric of $\Ric>0$ such that the projection onto $(S^a,ds_a^2)$ is a Riemannian submersion with $SO(b+1)$-connection and totally geodesic fibres isometric to the induced metric of a ball of radius $\varepsilon$ in the round sphere of radius $r$, for some $\varepsilon,r>0$,
        \item The weight function $\bar{f}$ is constant on $\bar{\varphi}(S^a\times D^{b+1})$.
    \end{enumerate}
\end{proposition}

To prove Proposition \ref{P:deform_tot_geod}, we first make a general consideration. Let $(M,g)$ be a Riemannian manifold and let $N\subseteq M$ be an embedded submanifold. Let $\nu(N)\to N$ be the normal bundle of $N$ in $M$. Recall that the normal connection $\nabla^\perp$ on $\nu(N)\to N$ is defined by
\[ \nabla^\perp_X U=(\nabla_X U)^\perp, \]
where $\nabla$ is the Levi--Civita connection of $M$, $X$ is tangent to $N$, $U\in\nu(N)$ and $(\cdot)^\perp$ denotes the orthogonal projection onto $\nu(N)$.

Let $D_\varepsilon\nu(N)\to N$ denote the disc bundle of $\nu(N)\to N$ of radius $\varepsilon>0$. Then for any $r>0$, there exists a unique Riemannian metric $g_{N,r}$ on $D_\varepsilon\nu(N)$ so that
\begin{enumerate}
    \item $(D_\varepsilon\nu(N),g_{N,r})\to (N,g|_N)$ is a Riemannian submersion with totally geodesic fibres isometric to the induced metric of a ball of radius $\varepsilon>0$ in the round sphere of radius $r$,
    \item The connection of this Riemannian submersion
    coincides with $\nabla^\perp$, in particular, it is induced by a $SO(m)$-connection of the corresponding principal $SO(m)$-bundle,
\end{enumerate}
see \cite[9.59]{Be87}.

The 1-jets of the metrics $g$ and $g_{N,r}$ on $N$ differ by the second fundamental form of $g$ on $N$. In particular, if $N$ is totally geodesic, we have the following:
\begin{lemma}[{\cite{La07}}]\label{L:tot_geod_1-jet}
    Let $(M,g)$ be a Riemannian manifold and let $N\subseteq M$ be an embedded totally geodesic submanifold. For $r>0$ consider the metric $g_{N,r}$ defined above. Then, when we identify $D_\varepsilon\nu(N)$ with a tubular neighbourhood of $N$ in $M$ via the exponential map, the $1$-jets of the metrics $g$ and $g_{N,r}$ coincide on $N$.
\end{lemma}

Using this, we can now prove Proposition \ref{P:deform_tot_geod}.
\begin{proof}[Proof of Proposition \ref{P:deform_tot_geod}]
    To simplify notation, we will write $\varphi(S^a)$ instead of $\varphi(S^p\times\{0\})$. Let $g_{r,\varepsilon}$ denote the metric $g_{\varphi(S^a),r}$ on $D_\varepsilon\nu(\varphi(S^a))$ considered above. For $r$ sufficiently small, this metric has positive Ricci curvature, see e.g.\ \cite[9.70]{Be87}. Following Lemma \ref{L:tot_geod_1-jet}, we view $g_{r,\varepsilon}$ as a metric on a tubular neighbourhood of $\varphi(S^a)$. Therefore, Lemma \ref{L:tot_geod_1-jet} implies that the $1$-jets of the weighted metrics $(g,e^{-f})$ and $(g_{r,\varepsilon},e^{-f_0})$, where $f_0$ is the constant value of $f$ on $\varphi(S^a)$, coincide on $\varphi(S^a)$. Using weighted metrics of $\BE_q>0$, one can then deform $(g,e^{-f})$ into a new weighted metric $(\bar{g},e^{-\bar{f}})$ that coincides with $(g_{r,\varepsilon},e^{-f_0})$ on a neighbourhood of $\varphi(S^a)$, and with $(g,e^{-f})$ outside a (bigger) neighbourhood of $\varphi(S^a)$. We stress that these neighbourhoods can be chosen arbitrarily small, see \cite[Theorem 1.2]{BH22} and \cite[Theorem 1.10]{Wr02}.
    
    Thus, the restriction of $(\bar{g},e^{-\bar{f}})$ to $D_{\varepsilon'}\nu(\varphi(S^a))$ is given by $(g_{r,\varepsilon'},e^{-f_0})$ for $\varepsilon'>0$ sufficiently small. Since the normal bundle $\nu(\varphi(S^a))$ is trivial, with a trivialisation given by the differential $\varphi_*|_{S^a}$, we can identify $D_{\varepsilon'}\nu(\varphi(S^a))$ with $S^a\times D^{b+1}$. In other words, we obtain an embedding $\bar{\varphi}\colon S^a\times D^{b+1}\hookrightarrow M$ on which the weight function $\bar{f}$ is constant and the metric $\bar{g}$ is a submersion metric, as stated in (2) and (3).

    Finally, by the uniqueness of tubular neighbourhoods, see e.g.\ \cite[Theorem 4.5.3]{Hi76}, there exists a smooth map $A\colon S^a\to SO(b+1)$, such that $\varphi$ is isotopic to $\bar{\varphi}\circ \phi$, where $\phi\colon S^a\times D^{b+1}\to S^a\times D^{b+1}$ is the diffeomorphism defined by
    \[ \phi(x,y)=(x,A_x y). \]
    Hence, replacing $\bar{\varphi}$ by $\bar{\varphi}\circ\phi$ results in the required embedding.
\end{proof}

Conclusion (2) of Proposition \ref{P:deform_tot_geod} implies that the Riemannian submersion $(S^a\times D^{b+1},\bar{\varphi}^*\bar{g})\xrightarrow{\pi}(S^a,ds_a^2)$, where $\pi=\mathrm{pr}_{S^a}$, restricts to a Riemannian submersion on $S^a\times S^b$ and is determined by it. In other words, if $\mathcal{H}$ denotes the horizontal distribution of the Riemannian submersion $S^a\times S^b\to S^a$, as in Subsection \ref{SS:cylinder}, we can write $\varphi^*\bar{g}$ as
\begin{equation}\label{EQ:subm_metric}
    \bar{\varphi}^*\bar{g}=dt^2+\mathcal{H}^*\pi^*ds_a^2+r^2\sin^2\left(\frac{t}{r} \right)ds_b^2
\end{equation}
with $t\in[0,\varepsilon]$.

We will now modify the weighted metric $(\bar{\varphi}^*\bar{g},e^{-\bar{f}\circ\varphi})$ on $[0,\varepsilon]\times S^a\times S^b$ in several steps to collapse the $S^a$-factor instead of the $S^b$-factor. The first step consists of transitioning to a totally geodesic boundary. In the following, $\mathcal{H}$ will always denote an arbitrary horizontal distribution for the submersion $S^a\times S^b\to S^a$ and we fix $q\in(0,\infty]$.

\begin{lemma}\label{L:tot_geod}
    For any $r>0$, $t_0\in(0,r\frac{\pi}{2})$, $f_0\in\R$ and $\mu>0$ sufficiently small, there exist $t_1<t_0$ and a weighted Riemannian metric $(g_1,e^{-f_1})$ of $\BE_q>0$ on $[t_1,t_0]\times S^a\times S^{b}$ such that the following holds:
    \begin{enumerate}
        \item $g_1$ is isometric to $dt^2+\mathcal{H}^*\pi^*ds_a^2+r^2\sin^2\left(\frac{t}{r}\right)ds^2_{b}$ near $t=t_0$ and $f\equiv f_0$ near $t=t_0$,
        \item At $t=t_1$, the boundary $\{t_1\}\times S^a\times S^{b}$ is totally geodesic and the induced metric is  $\mathcal{H}^*\pi^*ds_a^2+\mu^2ds^2_{b}$,
        \item At $t=t_1$, the weight function $f_1$ is constant and has constant normal derivative.
    \end{enumerate}
\end{lemma}
\begin{proof}
    Let $\lambda\in(\cos(\frac{t_0}{r}),1)$ so that for any unit horizontal vector $u$ and any unit vertical vector $v$ we have
    \[ \left((a-1)-2 r^2\sin^2\left(\frac{t_\lambda}{r} \right)(A_u,A_u) \right)\frac{b-1}{r^2}>r^2\sin^2\left(\frac{t_\lambda}{r} \right)\left((\check{\delta}A)u,v\right)^2.\]
    Here $t_\lambda\in(0,t_0)$ is defined by imposing $\cos(\frac{t_\lambda}{r})=\lambda$, and $A$ denotes the $A$-tensor of $\mathcal{H}$ with respect to the metric $\mathcal{H}^*\pi^*ds_a^2+ds_b^2$ (cf.\ Section \ref{S:prelim}). Since $t_\lambda\to 0$ as $\lambda\to 1$, this inequality is satisfied for all $\lambda$ sufficiently close to $1$. We will use this estimate later to show that the weighted metric we construct has $\BE_q>0$.
    
    Now fix $t_0'\in(t_\lambda,t_0)$, so that $\cos(\frac{t_0'}{r})<\lambda$. As a first step, we replace the function $r\sin\left(\frac{\cdot}{r}\right)$ by a function whose derivative does not exceed $\lambda$. For that, let $\tilde{\beta}_1\colon[t_1',t_0']\to[0,\infty)$ be a smooth function with the following properties:
    \begin{enumerate}
        \item $\tilde{\beta}_1(t_1')=0$ and $\tilde{\beta}_1'(t_1')=\lambda$,
        \item $\tilde{\beta}_1''<0$,
        \item $\tilde{\beta}_1(t_0')= r\sin(\frac{t_0'}{r})$ and $\tilde{\beta}_1'(t_0')=\cos(\frac{t_0'}{r})$.
    \end{enumerate}
    Such a function can for example be obtained by modifying the $C^1$-function
    \[ t\mapsto\begin{cases}
        r\sin\left(\frac{t}{r}\right),\quad & t\in[t_\lambda,t_0'],\\
        \lambda(t-t_\lambda)+r\sin\left(\frac{t_\lambda}{r}\right),\quad & t\in[t_1',t_\lambda]
    \end{cases} \]
    with $t_1'=-\frac{r}{\lambda}\sin(\frac{t_\lambda}{r})+t_\lambda$, into a smooth function with strictly negative second derivative.
    
    We then consider the metric 
    \[\tilde{g}_1=dt^2+\mathcal{H}^*\pi^*ds_a^2+\tilde{\beta}_1(t)^2ds_{b}^2\]
    on $[t_1',t_0']\times S^a\times S^{b}$. We now use Lemma \ref{L:Ricci_doubly_warped_sub} to analyse its Ricci curvatures. First note that, since $\tilde{\beta}''<0$, we have positive Ricci curvature in $t$-direction. Thus, the metric $\tilde{g}_1$ has positive Ricci curvature if and only if
    \[ \Ric(u,u)\Ric(v,v)>\Ric(u,v)^2 \]
    for all horizontal vectors $u$ and vertical vectors $v$. Since $\tilde{\beta}_1\leq r\sin(\frac{t_\lambda}{r})$ and $\tilde{\beta}_1'\leq \lambda=\cos(\frac{t_\lambda}{r})$, we have
    \[ \frac{1-\tilde{\beta}_1'^2}{\tilde{\beta}_1^2}\geq\frac{1-\cos^2\left(\frac{t_\lambda}{r}\right)}{r^2\sin^2\left(\frac{t_\lambda}{r}\right)}=\frac{1}{r^2}.\]
    Hence, by the choice of $\lambda$, it follows from Lemma \ref{L:Ricci_doubly_warped_sub} that $\tilde{g}_1$ has positive Ricci curvature.

    Next, consider the functions $\beta, \gamma\colon[0,t_0]\to(0,\infty)$ obtained in Lemma \ref{L:warping_fcts}, where the parameters $a,b$ are set to $q,b$, respectively (and we replace $q$ by a finite value in case $q=\infty$). To avoid confusion with existing notation, here we denote by $t_1''$ the variable $t_0>0$ of Lemma \ref{L:warping_fcts}, and the $\varepsilon, r$ are arbitrary.

    For $r'>0$ consider the functions
    \[ \beta_{r'}(t)=r'\beta\left(\frac{t}{r'}\right), \quad \gamma_{r'}=r'\gamma\left(\frac{t}{r'}\right) \]
    and define the metric
    \[\bar{g}_1=dt^2+\mathcal{H}^*\pi^*ds_a^2+\beta_{r'}(t)^2ds_b^2 \]
    and the weight function
    \[ \bar{f}_1=-q\ln\left(\gamma_{r'} \right) \]
    on $[0,r't_1'']\times S^a\times S^b$. It then follows from Lemmas \ref{L:Ricci_doubly_warped_sub} and \ref{L:warping_fcts} that for all $r'$ sufficiently small the weighted Riemannian metric $(\bar{g}_1,\bar{f}_1)$ has $\BE_q>0$.

    Now, by possibly choosing $r'$ smaller, we can assume that $\beta_{r'}(r't_1'')\leq \tilde{\beta}_1(t_0')$, so there exists $t_1'''$ with $\beta_{r'}(r't_1'')=\tilde{\beta}_1(t_1''')$. Thus, by shifting the interval $[0,r't_1'']$ by $-r't_1''+t_1'''$, we can glue the weighted metric $(\bar{g}_1,\bar{f}_1-\bar{f}_1(r_1't_1'')+f_0)$ on $[-r't_1''+t_1''',t_1''']\times S^a\times S^b$ with $(\tilde{g}_1,f_0)$ on $[t_1''',t_0]\times S^a\times S^b$ using Corollary \ref{C:gluing_warped} to obtain a weighted metric $(g_1,e^{-f_1})$ of $\BE_q>0$ satisfying the required conditions. By choosing $t'$ smaller, we can realise any sufficiently small value of $\mu$.



\end{proof}

Having achieved a totally geodesic boundary, we now proceed by \textquotedblleft untwisting\textquotedblright\ the bundle.

\begin{lemma}\label{L:untwist}
    For any $\lambda_2$ and any $\mu>0$ sufficiently small  there exists a weighted Riemannian metric $(g_2,e^{-f_2})$ of $\BE_\infty>0$ on $[0,1]\times S^a\times S^{b}$ with the following properties:
    \begin{enumerate}
        \item the induced metric of $g_2$ on $\{1\}\times S^a\times S^{b}$ is given by $\mathcal{H}^*\pi^*ds_a^2+\mu^2ds_{b}^2$,
        \item the induced metric of $g_2$ on $\{0\}\times S^a\times S^{b}$ is given by the product $ds_a^2+\mu^2ds_{b}^2$,
        \item the boundaries $\{0\}\times S^a\times S^{b}$ and $\{1\}\times S^a\times S^{b}$ are totally geodesic,
        \item the function $f$ and its normal derivative are constant at both $t=0$ and $t=1$ and the normal derivative at $t=1$ is given by $-\lambda_2$.
    \end{enumerate}
\end{lemma}
\begin{proof}
    Let $\mathcal{H}_t$, $t\in[0,1]$, be a smoothly varying path of horizontal distributions with $\mathcal{H}_0=TS^m$ and $\mathcal{H}_1=\mathcal{H}$. For each $t\in[0,1]$, we define  the metric
    \[ g_t=\mathcal{H}_t^*\pi^*ds_a^2+\mu^2ds_{b}^2. \]
    By compactness, for $\mu$ sufficiently small, each metric $g_t$ has positive Ricci curvature. The claim now directly follows from Lemma \ref{L:isot_conc}.
\end{proof}

Finally, we collapse the sphere $S^a$.

\begin{lemma}\label{L:collapse}
    For any $\lambda_3,\mu>0$, there exist $t_3<0$ and smooth functions $\alpha_3,\beta_3\colon[t_3,0]\to [0,\infty)$, $f_3\colon[t_3,0]\to\R$ such that for the Riemannian metric
    \[ g_3=dt^2+\alpha_3(t)^2ds_a^2+\beta_3(t)^2ds_{b}^2 \]
    the following holds:
    \begin{enumerate}
        \item $\alpha_3$ is odd at $t=t_3$ with $\alpha_3'(t_3)=1$ with $\alpha_3|_{(t_3,0)}>0$, and both $\beta_3$ and $f_3$ are even at $t=t_3$ with $\beta_3>0$,
        \item $\alpha_3(0)=1$ and $\alpha_3'(0)=0$,
        \item $\beta_3(0)=\mu$ and $\beta_3'(0)=0$,
        \item $f_3'(0)=-\lambda_3$,
        \item the weighted Riemannian metric $(g_3,e^{-f_3})$ has $\BE_\infty>0$.
    \end{enumerate}
\end{lemma}
\begin{proof}
    We will simply set $\beta_3\equiv \mu$. To obtain $\alpha_3$, we start by defining, for $\varepsilon>0$, the function $\tilde{\alpha}_3$ on $(-\infty,0]$ to be a smooth function with the following properties:
    \begin{enumerate}
        \item $\tilde{\alpha}_3(0)=1$ and $\tilde{\alpha}_3'(0)=0$,
        \item $\tilde{\alpha}_3'\in[0,\varepsilon]$,
        \item $\tilde{\alpha}_3''<0$.
    \end{enumerate}
    Further, we define $\tilde{f}_3\colon (-\infty,0]\to\R$ by $\tilde{f}_3(t)=-\lambda_3 t$.

    Consider the metric
    \[\tilde{g}_3=dt^2+\tilde{\alpha}_3(t)^2 ds_a^2+\beta_3(t)^2 ds_{b}^2\]
    on $[\tilde{t}_3,0]$, where $\tilde{t}_3$ is the unique value with $\tilde{\alpha}_3(\tilde{t}_3)=0$. We now use Lemma \ref{L:Ricci_doubly_warped} to analyse the weighted Ricci curvatures of $(\tilde{g}_3,e^{-\tilde{f}_3})$.

    Since $\tilde{f}_3''=0$, $\beta_3'=\beta_3''\equiv0$ and $\tilde{\alpha}_3''<0$, we directly obtain that $\BE_\infty(\partial_t,\partial_t)$ and $\BE_\infty(\frac{v}{\beta_3},\frac{v}{\beta_3})$ are both positive. Thus, it remains to consider $\BE_\infty(\frac{u}{\tilde{\alpha}_3},\frac{u}{\tilde{\alpha}_3})$. Here we obtain
    \begin{align*}
        \BE_\infty\left(\frac{u}{\tilde{\alpha}_3},\frac{u}{\tilde{\alpha}_3}\right)&=-\frac{\tilde{\alpha}_3''}{\tilde{\alpha}_3}+(a-1)\frac{1-\tilde{\alpha}_3'^2}{\tilde{\alpha}_3^2}-\lambda_3\frac{\tilde{\alpha}_3'}{\tilde{\alpha}_3}\\
        &> (a-1)\frac{1-\varepsilon^2}{\tilde{\alpha}_3^2}-\lambda_3\frac{\varepsilon}{\tilde{\alpha}_3}\geq \frac{1}{\tilde{\alpha}_3}\left((a-1)(1-\varepsilon^2)-\lambda_3\varepsilon \right),
    \end{align*}
    where we used that $\tilde{\alpha}_3\leq 1$. For $\varepsilon$ sufficiently small, this expression is positive.

    Let now $t_\varepsilon$ be the value for which $\tilde{\alpha}_3(t_\varepsilon)=\varepsilon$ and consider for $t_3=t_\varepsilon-\arcsin(\varepsilon)$ the functions
    \[ \alpha_3(t)=\begin{cases}
        \sin(t-t_3),\quad & t\in[t_3,t_\varepsilon],\\
        \tilde{\alpha}_3(t),\quad & t\in[t_\varepsilon,0]
    \end{cases} \]
    and
    \[ f_3(t)=\begin{cases}
        \tilde{f}_3(t_\varepsilon),\quad &t\in[t_3,t_\varepsilon],\\
        \tilde{f}_3(t),\quad & t\in[t_\varepsilon,0].
    \end{cases} \]
    For the metric
    \[ g_3=dt^2+\alpha_3(t)^2ds_a^2+\beta_3(t)^2ds_{b}^2 \]
    we then have that $(g_3,e^{-f_3})$ has $\BE_\infty>0$ and all the required boundary conditions are satisfied. It remains to smooth the functions $\alpha_3$ and $f_3$ at $t=t_\varepsilon$ using Corollary \ref{C:gluing_warped}, i.e.\ we need to consider the following expressions:
    \begin{align*}
        {\alpha_3'}_-(t_\varepsilon)-{\alpha_3'}_+(t_\varepsilon)&=\cos(\arcsin(\varepsilon))-\tilde{\alpha}'_3(t_\varepsilon)\geq \cos(\arcsin(\varepsilon))-\varepsilon,\\
        a\frac{{\alpha_3'}_-(t_\varepsilon)}{\alpha_3(t_\varepsilon)}-a\frac{{\alpha_3'}_+(t_\varepsilon)}{\alpha_3(t_\varepsilon)}-\lambda_3&=a\frac{\cos(\arcsin(\varepsilon))}{\varepsilon}-a\frac{\tilde{\alpha}_3'(t_\varepsilon)}{\varepsilon}-\lambda_3\geq a\frac{\cos(\arcsin(\varepsilon))}{\varepsilon}-a-\lambda_3.
    \end{align*}
    Both expressions are strictly positive for $\varepsilon$ sufficiently small, so we can apply Corollary \ref{C:gluing_warped}.
\end{proof}

\begin{proof}[Proof of Theorem \ref{T:HIGHER_SURG}]
    Suppose $(M^n,g,e^{-f})$ is a weighted Riemannian manifold of $\BE_\infty>0$ and $\varphi\colon S^a\times D^{b+1}\hookrightarrow M$, is an embedding satisfying the hypotheses of Theorem \ref{T:HIGHER_SURG}.

    By Proposition \ref{P:deform_tot_geod}, we can assume that on $\varphi(S^a\times D^{b+1})$ the metric $g$ is of the form \eqref{EQ:subm_metric} for some $r,\varepsilon>0$, horizontal distribution $\mathcal{H}$, and the weight function $f$ is constant. As a first step, for $t_0\in(0,\min(\varepsilon,r\frac{\pi}{2})]$, we replace $g$ on $[0,t_0]\times S^a\times S^b$ by the weighted metric of $\BE_\infty>0$ on $[t_1,t_0]\times S^a\times S^b$ constructed in Lemma \ref{L:tot_geod}. By item (1) of Lemma \ref{L:tot_geod}, the resulting weighted metric is again smooth and by items (2) and (3) the new boundary $\{t_1\}\times S^a\times S^b$ is totally geodesic with induced metric given by
    \[\mathcal{H}^*\pi^*ds_a^2+\mu^2ds_b^2 \]
    for all $\mu>0$ sufficiently small. Further, the weight function is constant along the boundary with constant normal derivative.

    Next, by possibly choosing $\mu$ smaller, we attach a cylinder $[0,1]\times S^a\times S^b$ equipped with the weighted metric constructed in Lemma \ref{L:untwist} using Theorem \ref{T:GLUING}. Thus, we now have the same conclusion as before, but we additionally obtain that the metric is a product metric
    \[ ds_a^2+\mu^2 ds_b^2 \]
    on the boundary.

    Finally, by again possibly choosing $\mu$ smaller, we attach another cylinder $[t_3,0]\times S^a\times S^b$ equipped with the weighted metric constructed in Lemma \ref{L:collapse} using Theorem \ref{T:GLUING}. By item (1) of Lemma \ref{L:collapse}, this defines in fact a metric on the space obtained from $[t_3,0]\times S^a\times S^b$ by collapsing $\{t_3\}\times S^a\times \{x\}$ for all $x\in S^b$, that is, on $D^{a+1}\times S^b$. Hence, we have performed a surgery operation along $\varphi$.
\end{proof}

\section{Highly-connected manifolds}\label{S:5-mfds}

In this section we apply Theorem \ref{T:HIGHER_SURG} to $(2m-1)$-connected $(4m+1)$-manifolds with $m\geq1$ and prove the following result:
\begin{theorem}\label{T:highly-conn}
    Let $M^{4m+1}$ be a closed, $(2m-1)$-connected $2m$-parallelisable manifold. Then there exists a homotopy sphere $\Sigma^{4m+1}$ such that $M\#\Sigma$ admits a weighted Riemannian metric of $\BE_{\infty}>0$.
\end{theorem}
Since there exists no exotic sphere in dimension 5, and since a closed, simply-connected 5-manifold is 2-parallelisable if and only if it is spin, we obtain Theorem \ref{T:dim-5} from Theorem \ref{T:highly-conn} by setting $m=1$.

For $m\geq2$, it was shown by Crowley--Wraith \cite{CW22} that Theorem \ref{T:highly-conn} holds if one replaces the condition $\BE_\infty>0$ with the intermediate condition of 2-positive Ricci curvature. Like in \cite{CW22}, we will use Wall's classification of handlebodies \cite{Wa62} to identify the manifolds we construct. In dimension 5, one can alternatively also use the classification of closed, simply-connected 5-manifolds by Smale \cite{Sm62} and Barden \cite{Ba65}.

Our techniques could also be used to construct metrics of $\BE_\infty>0$ on $(2m-2)$-connected $(4m-1)$-manifolds. We note, however, that this case is already covered by Crowley--Wraith \cite{CW17} for the stronger condition of $\Ric>0$.

\begin{remark}
    Similarly as in Corollary \ref{C:5-mfs}, it follows from Theorem \ref{T:highly-conn} that for any such manifold $M$ there exists a homotopy sphere $\Sigma^{4m+1}$ such that for all $q\in\N$ sufficiently large, the manifold $(M\#\Sigma)\times S^q$ admits a Riemannian metric of $\Ric>0$. It remains an open question whether the connected sum with $\Sigma$ is necessary, i.e.\ whether the manifold $M\times S^q$ admits a Riemannian metric of $\Ric>0$. This would follow directly if one could show that $(M\#\Sigma)\times S^q$ is in fact diffeomorphic to $M\times S^q$.
    
    If we identify $M\times S^q$ and $(M\#\Sigma)\times S^q$ as topological manifolds in the obvious way, easy arguments in smoothing theory show that the corresponding smooth structures are not concordant, and hence not isotopic, when $\Sigma $ is not the standard sphere. However, this does not rule out the existence of a diffeomorphism between the two smooth structures. To the best of our knowledge, this problem is open.
\end{remark}

The strategy for the proof of Theorem \ref{T:highly-conn} is as follows. By the work of Wall \cite{Wa62} and Crowley--Wraith \cite{CW17,CW22}, all manifolds $M$ in Theorem \ref{T:highly-conn}, after possibly a connected sum with a homotopy sphere, can be realised as the boundary of a \emph{handlebody}, i.e.\ a manifold obtained by attaching $(2m+1)$-handles to the disc $D^{4m+2}$. Since the effect of a handle attachment to the boundary is a surgery operation, we can use Theorem \ref{T:HIGHER_SURG} to obtain a weighted metric of $\BE_\infty>0$ on the boundary, provided the assumptions on the metric required in Theorem \ref{T:HIGHER_SURG} are satisfied. The embeddings for the surgery operation we use will be obtained from intersections of the sphere $S^{4m+1}$ with a $(2m+1)$-dimensional affine subspace in $\R^{4m+2}$. It then remains to construct a metric on $S^{4m+1}$ for which we can apply Theorem \ref{T:HIGHER_SURG}, and, by using classification results of Wall \cite{Wa62}, to show that these embeddings realise all possibly handlebodies.

\subsection{Wall's classification of handlebodies}

In this section, we recall Wall's classification of handlebodies \cite{Wa62}. We also refer to \cite[Section 3]{CW17}, \cite[Section 2]{CW22}. A \emph{handlebody of dimension $4m+2$} as defined by Smale \cite{Sm62} is a manifold obtained from the disc $D^{4m+2}$ by attaching handles $D^{2m+1}\times D^{2m+1}$ along the boundary $S^{2m}\times D^{2m+1}$. The set of diffeomorphism classes of handlebodies of dimension $4m+2$ is denoted by $\mathcal{H}(4m+2)$.

Following \cite[Section 3]{CW17}, we define the \emph{quadratic module} $\pi_{2m}\{SO(2m+1)\}$ as the quadruple
\[ \pi_{2m}\{SO(2m+1)\}=(\pi_{2m}(SO(2m+1)),\Z,\mathrm{h},\mathrm{p}), \]
where $\mathrm{h}\colon \pi_{2m}(SO(2m+1))\to\Z$ and $\mathrm{p}\colon\Z\to\pi_{2m}(SO(2m+1))$ are the maps
\[ \mathrm{h}(\xi)=e(\xi),\quad \mathrm{p}(k)=k\cdot\tau_{S^{2m+1}}. \]
Here, we have identified $\pi_{2m}(SO(2m+1))$ with the set of isomorphism classes of vector bundles of rank $(2m+1)$ over the sphere $S^{2m+1}$, and $e(\xi)$ denotes the Euler number of the bundle $\xi$ and $\tau_{S^{2m+1}}$ the tangent bundle of $S^{2m+1}$. In particular, the group $\pi_{2m}(SO(2m+1))$ is finite so the map $\mathrm{h}$ is trivial and we do not need to consider it.

An \emph{extended quadratic form} over $\pi_{2m}\{SO(2m+1)\}$ is a triple $(H,\lambda,\mu)$ where $H$ is a finitely generated free abelian group, $\lambda\colon H\times H\to\Z$ is a skew-symmetric bilinear form, and $\mu\colon H\to\pi_{2m}(SO(2m+1))$ is a map satisfying
\[ \mu(x+y)=\mu(x)+\mu(y)+\mathrm{p}(\lambda(x,y)). \]
For a given handlebody $W$ of dimension $(4m+2)$, we can define an extended quadratic form $(H_W,\lambda_W,\mu_W)$ by setting $H_W:=H^{2m+1}(W,\partial W)$, $\lambda_W$ as the intersection form of $W$, and $\mu_W(x)$ as the isomorphism class of the normal bundle of an embedding of $S^{2m+1}$ representing $x\in H^{2m+1}(W,\partial W)\cong H_{2m+1}(W)$.

The classification of handlebodies is now as follows:
\begin{theorem}[{\cite{Wa62}, \cite[Theorem 2.2]{CW22}}]\label{T:handlebody_class}
    The assignment $W\mapsto (H_W,\lambda_W,\mu_W)$ defines a bijection between $\mathcal{H}(4m+2)$ and the set of isomorphism classes of extended quadratic forms over $\pi_{2m}\{SO(2m+1)\}$. Moreover, every isomorphism of extended quadratic forms $(H_{W_1},\lambda_{W_1},\mu_{W_1})\cong(H_{W_2},\lambda_{W_2},\mu_{W_2})$ is realised by a diffeomorphism $W_1\cong W_2$.
\end{theorem}

We can determine the associated extended quadratic form of a handlebody from the attaching maps as follows. Let $\varphi_1,\dots,\varphi_\ell\colon S^{2m}\times D^{2m+1}\hookrightarrow S^{4m+1}$ be embeddings with pairwise disjoint images. We extend each embedding $\varphi_i|_{S^{2m}\times \{0\}}$ to a map $D^{2m+1}\hookrightarrow D^{4m+2}$. By the relative Whitney embedding theorem \cite[Theorem 5]{Wh36}, we can assume that these maps are embeddings, and by the uniqueness of tubular neighbourhoods (see e.g.\ \cite[Section 4.5]{Hi76}) we can extend them to embeddings $\bar{\varphi}_i\colon D^{2m+1}\times D^{2m+1}\hookrightarrow D^{4m+2}$ satisfying
\[ \varphi_i(x,y)=\bar{\varphi}_i(x,{\phi_i}_xy) \]
for smooth functions $\phi_i\colon S^{2m}\to SO(2m+1)$.

The following two lemmas are well-known. We inculde their proofs for convenience.
\begin{lemma}\label{L:handlebody_inv}
    Let $A\in\Z^{\ell\times\ell}$ be the oriented intersection matrix of the embeddings $\bar{\varphi}_i|_{D^{2m+1}\times\{0\}}$. Then the handlebody $W$ obtained by attaching handles along the embeddings $\varphi_i$ has associated extended quadratic form given by $(\Z^\ell,A,([\phi_1],\dots,[\phi_\ell]))$.
\end{lemma}
\begin{proof}
    Let $S_i\subseteq W$ be the $(2m+1)$-sphere consisting of the two discs $\bar{\varphi}_i(D^{2m+1}\times\{0\})$ and $D^{2m+1}\times\{0\}$ in the $i$-th attached handle $D^{2m+1}\times D^{2m+1}$. Then $W$ is homotopy equivalent to the one-point union
    \[S_1\vee\overset{\ell}{\cdots}\vee S_\ell, \]
    so that $H_W\cong\Z^\ell$ is generated by the spheres $S_1,\dots,S_\ell$. Since the spheres $S_i$ intersect each other only in $D^{4m+2}$, this shows that $\lambda_W$ is given by $A$. Moreover, the functions $\phi_i$ are the clutching functions of the normal bundle of $S_i$, which shows that $\mu_W=([\phi_1],\dots,[\phi_\ell])$.
\end{proof}

We will be interested in the boundary of a handlebody. For that, we have the following result:
\begin{lemma}\label{L:handlebody_examples}
    Let $W\in\mathcal{H}(4m+2)$ be a handlebody.
    \begin{enumerate}
        \item If $H_W\cong\Z$, then $\partial W$ is the total space of a linear $S^{2m+1}$-bundle over $S^{2m}$.
        \item If $H_W\cong\Z^2$ and there exists a basis of $H_W$ in which $\lambda_W$ is represented by the matrix
        \[ \begin{pmatrix}
            0 & 1\\-1 & 0
        \end{pmatrix}, \]
        then $\partial W$ is a homotopy sphere.
    \end{enumerate}
\end{lemma}
\begin{proof}
    First, suppose $H_W\cong\Z$. By Lemma \ref{L:handlebody_inv}, the manifold $W$ is obtained from $D^{4m+2}$ by attaching a single handle along an embedding $\varphi\colon S^{2m}\times D^{2m+1}\hookrightarrow S^{4m+1}$. We write the sphere $S^{4m+1}$ as
    \[ S^{4m+1}\cong (S^{2m}\times D^{2m+1})\cup_{S^{2m}\times S^{2m}}(D^{2m+1}\times S^{2m}). \]
    Then, by the Wu--Whitney embedding theorem \cite{Wu58} and the uniqueness of tubular neighbourhoods, the embedding $\varphi$ is isotopic to an embedding
    \[ (x,y)\mapsto(x,\phi_x y) \]
    into the first factor, where $\phi\colon S^{2m}\to SO(2m+1)$ is a smooth map. Hence, the manifold $\partial W$ is diffeomorphic to
    \[ (D^{2m+1}\times S^{2m})\cup_\phi(D^{2m+1}\times S^{2m}), \]
    i.e.\ the total space of the linear $S^{2m+1}$-bundle with clutching function $\phi$.

    Now suppose that $H_W\cong\Z^2$ and $\lambda_W$ is given by
    \[ \begin{pmatrix}
            0 & 1\\-1 & 0
        \end{pmatrix}. \]
    Since $\partial W$ is obtained from the sphere $S^{4m+1}$ by a sequence of $2m$-surgeries, it is simply-connected (see \cite{Mi61}). Further, since $W$ has non-trivial homology only in degrees $0$ and $(2m+1)$, it follows from Poincaré duality, the universal coefficient theorem and the long exact sequence of the pair $(W,\partial W)$, that $\partial W$ has possibly non-trivial homology groups only in degrees $0$, $2m$, $2m+1$ and $4m+1$. For degrees $2m$ and $2m+1$, we obtain the following exact sequence:
    \[ 0\longrightarrow H_{2m+1}(\partial W)\longrightarrow H_{2m+1}(W)\longrightarrow H_{2m+1}(W,\partial W)\longrightarrow H_{2m}(\partial W)\longrightarrow 0\,. \]
    By Poincaré duality, the map $H_{2m+1}(W)\to H_{2m+1}(W,\partial W)$ is given by $\lambda_W$. Since $\lambda_W$ is invertible, it follows that both $H_{2m+1}(\partial W)$ and $H_{2m}(\partial W)$ are trivial. It follows that $\partial W$ is a simply-connected homology sphere, and hence a homotopy sphere by the Hurewicz and Whitehead theorems.
\end{proof}

\subsection{Geometric setup}

In this section, we define a metric of $\Ric>0$ on the sphere $S^{4m+1}$ so that the intersections of certain $(2m+1)$-dimensional affine subspaces with $S^{4m+1}$ are round and totally geodesic.

We begin by defining such a metric for a single affine subspace that is \textquotedblleft close\textquotedblright\ to a linear subspace.

\begin{proposition}\label{P:unlink}
    Let $S\subseteq S^{2m+1}$ be a totally geodesic round $m$-sphere in a round $(2m+1)$-sphere, i.e.\ there exists an $(m+1)$-dimensional subspace $W\subseteq \R^{2m+2}$ with $S=W\cap S^{2m+1}$. Let $\hat{W}\subseteq \R^{2m+2}$ be a $(2m+1)$-dimensional subspace containing $W$ with unit sphere $\hat{S}=S^{2m+1}\cap \hat{W}$. For $\varepsilon>0$, let $S_\varepsilon\subseteq S^{2m+1}$ be the submanifold obtained by moving $S$ by distance $\varepsilon$ (w.r.t\ the metric on $S^{2m+1}$) in orthogonal direction to $\hat{S}$, i.e.\
    \[ S_\varepsilon=\cos(\varepsilon)S+\sin(\varepsilon)N, \]
    where $N\in\hat{W}^\perp$ is a unit normal. Then, for any $\delta>0$ and all $\varepsilon>0$ sufficiently small, there exists a Riemannian metric $g$ of positive Ricci curvature on $S^{2m+1}$ with the following properties:
    \begin{enumerate}
        \item $g$ is the round metric of radius $1$ on the complement of $A_{\varepsilon,6\varepsilon}(\hat{S})\cap B_{\delta}(S)$,
        \item the submanifold $S_{3\varepsilon}\subseteq S^{2m+1}$ is round and totally geodesic.
    \end{enumerate}
\end{proposition}

We first show that there exists a metric that satisfies item (2) and leaves the metric unchanged on the complement of $A_{\varepsilon,6\varepsilon}(\hat{S})$. Here $A_{\varepsilon,6\varepsilon}(\hat{S})$ denotes the annulus around $\hat{S}$ with inner and outer radii given by $\varepsilon$ and $6\varepsilon$, respectively.
\begin{lemma}\label{L:h_eps}
    For every $\varepsilon>0$ sufficiently small, there exists a smooth function $h_\varepsilon\colon[0,\frac{\pi}{2}]\to[0,\infty)$ such that the following holds:
    \begin{enumerate}
        \item $h_\varepsilon(t)=\cos(t)$ for all $t\not\in(\varepsilon,6\varepsilon)$.
        \item $h_\varepsilon$ converges to $t\mapsto\cos(t)$ as $\varepsilon\to 0$ in the $C^1$-norm.
        \item $h_\varepsilon'(3\varepsilon)=0$.
        \item The doubly warped product metric
        \[ dt^2+h_\varepsilon^2(t)ds_m^2+\sin^2(t)ds_m^2 \]
        on $S^{2m+1}$ has Ricci curvatures $\geq \rho$ for some $\rho>0$ independent of $\varepsilon$.
    \end{enumerate}
    \begin{proof}
        Let $\nu>1$ and consider the function $h_1\colon\R\to\R$ defined by
        \[ h_1(t)=\cos(\nu\varepsilon)\cosh\left(\frac{t-\nu\varepsilon}{\nu}\right)-\nu\sin(\nu\varepsilon)\sinh\left(\frac{t-\nu\varepsilon}{\nu}\right), \]
        i.e.\ $h_1$ is the unique function whose value and derivative at $t=\nu\varepsilon$ coincide with those of the cosine function at the same point, and satisfies $h_1''=\frac{1}{\nu^2}h_1$.

        A calculation now shows that $h_1$ has vanishing first derivative at
        \[ t_\varepsilon=\nu\,\mathrm{arctanh}(\nu\tan(\nu\varepsilon))+\nu\varepsilon. \]
        By l'Hôpital's rule, we obtain that $\frac{t_\varepsilon}{\varepsilon}$ converges to $\nu(\nu^2+1)$ as $\varepsilon\to 0$. In particular, for $\nu$ sufficiently close to $1$ (e.g.\ for $\nu^4<\frac{3}{2}$) and $\varepsilon$ sufficiently small, we have $t_\varepsilon<\frac{3}{\nu}\varepsilon$.

        Now consider the function $h_2\colon\R\to\R$ defined by
        \[ h_2(t)=\frac{\cos(\frac{6}{\nu}\varepsilon)-h_1(t_\varepsilon)}{3\nu\varepsilon(\frac{2}{\nu^2}-1)}(t-3\nu\varepsilon)+h_1(t_\varepsilon), \]
        i.e.\ $h_2$ is the unique linear function with $h_2(3\nu\varepsilon)=h_1(t_\varepsilon)$ and $h_2(\frac{6}{\nu}\varepsilon)=\cos(\frac{6}{\nu}\varepsilon)$. We then define $\tilde{h}_\varepsilon\colon[0,\frac{\pi}{2}]\to[0,\infty)$ by
        \[ \tilde{h}_\varepsilon(t)=\begin{cases}
            h_1(t),\quad & t\in[\nu\varepsilon,t_\varepsilon],\\
            h_1(t_\varepsilon),\quad & t\in[t_\varepsilon,3\nu\varepsilon],\\
            h_2(t),\quad & t\in[3\nu\varepsilon,\frac{6}{\nu}\varepsilon],\\
            \cos(t),\quad &\text{else.}
        \end{cases} \]
        A sketch of the graph of the function $\tilde{h}_\varepsilon$ is given in Figure \ref{fig:h_eps}.
        \begin{figure}
            \centering
            \includegraphics{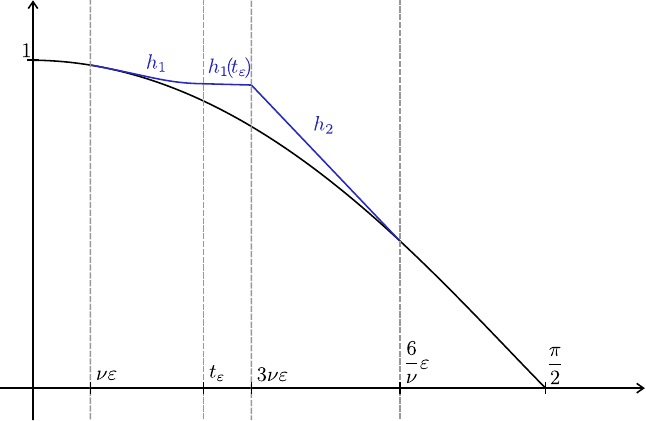}
            \caption{Sketch of the graph of the function $\tilde{h}_\varepsilon$.}
            \label{fig:h_eps}
        \end{figure}
        
        If $\varepsilon$ is sufficiently small and $\nu>1$ sufficiently close to $1$ so that $t_\varepsilon<\frac{3}{\nu}\varepsilon$, the function $\tilde{h}_\varepsilon$ is well-defined and continuous, and satisfies $\tilde{h}'_\varepsilon(3\varepsilon)=0$. Furthermore, we have $\tilde{h}_\varepsilon(t)=\cos(t)$ for all $t\not\in(\nu\varepsilon,\frac{6}{\nu}\varepsilon)$, which is strictly contained in $(\varepsilon,6\varepsilon)$. Moreover, except at the points $t=\nu\varepsilon,t_\varepsilon,3\nu\varepsilon,\frac{6}{\nu}\varepsilon$ it is smooth. We will now show that the function $\tilde{h}_\varepsilon$ satisfies the required properties and that we can smooth it so that these properties are preserved.

        
        For the $C^1$-convergence, we have that $\tilde{h}_\varepsilon$ is contained in the interval $[\cos(\frac{6}{\nu}\varepsilon),1]$ on $[0,\frac{6}{\nu}\varepsilon]$. Furthermore, the derivatives of $h_1$ are contained in the interval $[-\sin(\nu\varepsilon),0]$ on $[\nu\varepsilon,t_\varepsilon]$ and a calculation shows that
        \[\frac{h_2'(\frac{6}{\nu}\varepsilon)}{\sin(\frac{6}{\nu}\varepsilon)}\to \frac{\nu^6+\nu^4-36}{36(2-\nu^2)} \]
        as $\varepsilon\to 0$. For $\nu$ sufficiently close to $1$, this expression is strictly bigger than $-1$. In particular, for $\varepsilon$ sufficiently small and $\nu$ sufficiently close to $1$, the derivative of $h_2$ is contained in the interval $[-\sin(\frac{6}{\nu}\varepsilon),0]$. Thus, $\tilde{h}_\varepsilon'$ is contained in the interval $[-\sin(\frac{6}{\nu}\varepsilon),0]$ on $[0,\frac{6}{\nu}\varepsilon]$, showing the required convergence.

        For the Ricci curvatures, by Lemma \ref{L:Ricci_doubly_warped}, the following inequalities need to be satisfied:
        \begin{align*}
            &-m\frac{\tilde{h}_\varepsilon''}{\tilde{h}_\varepsilon}+m\geq\rho,\\
            &-\frac{\tilde{h}_\varepsilon''}{\tilde{h}_\varepsilon}+(m-1)\frac{1-{\tilde{h}_\varepsilon}^{\prime\,2}}{\tilde{h}_\varepsilon^2}-m\cot(t)\frac{\tilde{h}_\varepsilon'}{\tilde{h}_\varepsilon}\geq\rho,\\
            &m-m\cot(t)\frac{\tilde{h}_\varepsilon'}{\tilde{h}_\varepsilon}\geq\rho.
        \end{align*}
        The $C^1$-convergence, together with the fact that $\tilde{h}_\varepsilon''\leq \frac{1}{\nu^2}\tilde{h}_\varepsilon$, implies that for fixed $\nu$ and for $\varepsilon$ sufficiently small, such $\rho$ exists.

        Finally, we define $h_\varepsilon$ as the function obtained from $\tilde{h}_\varepsilon$ by smoothing at the points of non-smoothness. For that, around each of these points, we choose neighbourhoods of size so small that they do not contain the points $t=\varepsilon, 3\varepsilon,6\varepsilon$, and smooth out the function $\tilde{h}_\varepsilon$ using Corollary \ref{C:gluing_warped} (where we use a constant weight function), see also \cite[Corollary 3.2]{Re23}. To guarantee that the Ricci curvatures are still positive for the resulting function, we need to verify that at each of these points the left-hand side derivative is at least the right-hand side derivative. This is clear at the points $t=\nu\varepsilon,t_\varepsilon,3\nu\varepsilon$. At $t=\frac{6}{\nu}\varepsilon$ this follows from the fact that the quotient $\frac{h_2'(\frac{6}{\nu}\varepsilon)}{\sin(\frac{6}{\nu}\varepsilon)}$ converges to a value strictly bigger than $-1$ for $\nu$ sufficiently close to $1$ as seen above.
    \end{proof}
\end{lemma}

\begin{proof}[Proof of Proposition \ref{P:unlink}]
    The round metric on $S^{2m+1}$ can be expressed as
    \[ ds_{2m+1}^2=dt^2+\cos^2(t)ds_{2m}^2, \]
    where $t\in[-\frac{\pi}{2},\frac{\pi}{2}]$ is the signed distance from $\hat{S}$. Moreover, the round metric $ds_{2m}^2$ on $S^{2m}$ can be written as
    \[ ds_{2m}^2=ds^2+\cos^2(s)ds_m^2+\sin^2(s)ds_{m-1}^2, \]
    where $s\in[0,\frac{\pi}{2}]$ is the distance from $S_t$ in $(\hat{S}_t,ds_{2m}^2)$. Hence, we can write
    \[ ds_{2m+1}^2=dt^2+\cos^2(t)\left( ds^2+\cos^2(s)ds_m^2+\sin^2(s)ds_{m-1}^2 \right) \]
    and the spheres $\hat{S}$ and $S_\varepsilon$ correspond to the sets $\{t=0\}$ and $\{t=\varepsilon,s=0\}$, respectively.
    
    For given $\delta,\varepsilon>0$ with $\varepsilon$ sufficiently small we will now modify this metric on the set $\{\varepsilon\leq t\leq 6\varepsilon,0\leq s\leq \delta' \}$, where $\varepsilon^2+{\delta'}^2\leq\delta^2$. By the triangle inequality this set is contained in $A_{\varepsilon,6\varepsilon}(\hat{S})\cap B_{\delta}(S)$. To simplify the notation we will write $\delta$ instead of $\delta'$ in the following.
    
    Let $\tilde{\alpha}\colon[-\frac{\pi}{2},\frac{\pi}{2}]\times [0,\frac{\pi}{2}]\to[0,\infty)$ be a smooth function and consider the metric
    \[ g_{\tilde{\alpha}}=dt^2+\cos^2(t)ds^2+\tilde{\alpha}(t,s)^2\cos^2(s)ds_m^2+\cos^2(t)\sin^2(s)ds_{m-1}^2. \]
    Then, by Lemma \ref{L:Ricci_triple_warped}, the Ricci curvatures of the metric $g_{\tilde{\alpha}}$ are given as follows:
    \begin{align*}
        \Ric(\partial_t,\partial_t)=&m\left(1-\frac{\tilde{\alpha}_{tt}}{\tilde{\alpha}}\right),\\
        \Ric(\partial_t,\tfrac{\partial_s}{\gamma})=&\frac{m}{\cos(t)}\left( -\frac{\tilde{\alpha}_{st}}{\tilde{\alpha}}+\tan(s)\frac{\tilde{\alpha}_t}{\tilde{\alpha}}-\tan(t)\frac{\tilde{\alpha}_s}{\tilde{\alpha}}+\tan(s)\tan(t) \right),\\
        \Ric(\tfrac{\partial_s}{\gamma},\tfrac{\partial_s}{\gamma})=&\frac{m}{\cos^2(t)}\left( 1-\frac{\tilde{\alpha}_{ss}}{\tilde{\alpha}}+2\tan(s)\frac{\tilde{\alpha}_s}{\tilde{\alpha}} \right)+m\left( 1+\tan(t)\frac{\tilde{\alpha}_t}{\tilde{\alpha}} \right),\\
        \Ric(\tfrac{u}{\alpha},\tfrac{u}{\alpha})=&\frac{1}{\cos^2(t)}\left( -\frac{\tilde{\alpha}_{ss}}{\tilde{\alpha}}+2m\tan(s)\frac{\tilde{\alpha}_s}{\tilde{\alpha}}-(m-1)\frac{\tilde{\alpha}_s^2}{\tilde{\alpha}^2}-(m-1)\tan^2(s)-(m-1)\cot(s)\frac{\tilde{\alpha}_s}{\tilde{\alpha}}+m \right)\\
        &-\frac{\tilde{\alpha}_{tt}}{\tilde{\alpha}}+(m-1)\frac{1-\tilde{\alpha}_t^2\cos^2(s)}{\tilde{\alpha}^2\cos^2(s)}+m\tan(t)\frac{\tilde{\alpha}_t}{\tilde{\alpha}},\\
        \Ric(\tfrac{v}{\beta},\tfrac{v}{\beta})=& m\left( 1+\frac{1}{\cos^2(t)}\left( 1-\cot(s)\frac{\tilde{\alpha}_s}{\tilde{\alpha}} \right)+\tan(t)\frac{\tilde{\alpha}_t}{\tilde{\alpha}} \right),\\
        \Ric(\partial_t,\tfrac{u}{\alpha})=&\Ric(\partial_t,\tfrac{v}{\beta})=\Ric(\tfrac{\partial_s}{\gamma},\tfrac{u}{\alpha})=\Ric(\tfrac{\partial_s}{\gamma},\tfrac{v}{\beta})=\Ric(\tfrac{u}{\alpha},\tfrac{v}{\beta})=0.
    \end{align*}
    Here we set $\gamma(t)=\cos(t)$, $\alpha(t,s)=\tilde{\alpha}(t,s)\cos(s)$, $\beta(t,s)=\cos(t)\sin(s)$ and $u$ and $v$ are unit tangent vectors of $(S^m,ds_m^2)$ and $(S^{m-1},ds_{m-1}^2)$, respectively.

    Now let $\chi\colon\R\to[0,1]$ be a smooth function with $\chi|_{(-\infty,0]}\equiv 1$ and $\chi|_{[1,\infty)}\equiv 0$ (which necessarily is non-constant on $[0,1]$) and set
    \[ \tilde{\alpha}(t,s)=\chi\left(\frac{s}{\delta} \right)h_\varepsilon(t)+\left(1-\chi\left(\frac{s}{\delta}\right)\right)\cos(t), \]
    where $h_\varepsilon$ is the function obtained in Lemma \ref{L:h_eps}. Then for $t\not\in(\varepsilon,6\varepsilon)$ or $s>\delta$ we have $\tilde{\alpha}(t,s)=\cos(t)$, so that $g_{\tilde{\alpha}}$ coincides with the round metric at these points. Furthermore, the second fundamental form of the submanifold $\{t\}\times \{s\}\times S^m\times\{v\}$ is given by
    \[ \II\left(\frac{u}{\alpha},\frac{u}{\alpha}\right)=\frac{\alpha_t}{\alpha}\partial_t+\frac{\alpha_s}{\alpha}\frac{\partial_s}{\gamma}, \]
    which is given at $(t,s)=(3\varepsilon,0)$ by
    \[ \II\left(\frac{u}{\alpha},\frac{u}{\alpha}\right)=\frac{h_\varepsilon'(3\varepsilon)}{h_\varepsilon(3\varepsilon)}\partial_t+\frac{\chi'(0)}{\delta}\frac{h_\varepsilon(3\varepsilon)-\cos(3\varepsilon)}{h_\varepsilon(3\varepsilon)}\frac{\partial_s}{\gamma}=0, \]
    so $S_{3\varepsilon}$ is totally geodesic in $(S^{2m+1},g_{\tilde{\alpha}})$.

    It remains to show that for $\varepsilon$ sufficiently small, the metric $g_{\tilde{\alpha}}$ has positive Ricci curvature. Since $h_\varepsilon$ converges to $t\mapsto\cos(t)$ in $C^1$ as $\varepsilon\to 0$, we can bound $|\tilde{\alpha}_t+\sin(t)|$, $|\tilde{\alpha}_s|$, $|\tilde{\alpha}_{ss}|$ and $|\tilde{\alpha}_{st}|$ uniformly by any positive constant by choosing $\varepsilon$ sufficiently small. The same holds for $|\cot(s)\tilde{\alpha}_s|$, since $\chi'\left(\frac{s}{\delta}\right)\cdot\cot(s)$ converges to $0$ as $s\to 0$.
    
    We now calculate as follows:
    \[ -\frac{\tilde{\alpha}_{tt}}{\tilde{\alpha}}=\chi\left(\frac{s}{\delta}\right) \frac{-h_\varepsilon''}{h_\varepsilon}\frac{h_\varepsilon}{\tilde{\alpha}}+\left(1-\chi\left(\frac{s}{\delta}\right)  \right) \frac{\cos(t)}{\tilde{\alpha}}. \]
    When $h_\varepsilon''(t)\leq 0$, this expression is non-negative. Otherwise we take $\varepsilon>0$ small enough so $\frac{h_\varepsilon(t)}{\cos(t)}<\frac{m-\frac{\rho}{2}}{m-\rho}$, where we assume $\rho<m$. Then, using that $-\frac{h_\varepsilon''}{h_\varepsilon}\geq \frac{\rho-m}{m}$ by Lemma \ref{L:h_eps} and that $\tilde{\alpha}(t)\geq\cos(t)$, we have
    \begin{align*}
        \frac{-h_\varepsilon''}{h_\varepsilon}\frac{h_\varepsilon}{\tilde{\alpha}}>\frac{\rho-m}{m}\frac{m-\frac{\rho}{2}}{m-\rho}\geq -1+\frac{\rho}{2m}.
    \end{align*}
    Since $\left(1-\chi\left(\frac{s}{\delta}\right)  \right) \frac{\cos(t)}{\tilde{\alpha}}\geq 0$, we obtain for all $t$ the estimate
    \[ -\frac{\tilde{\alpha}_{tt}}{\tilde{\alpha}}\geq -1+\frac{\rho}{2m}.\]


    Hence, the Ricci curvature $\Ric(\partial_t,\partial_t)$ is bounded from below by a positive constant that is independent of $\varepsilon$.

    Further, by choosing $\varepsilon$ small enough, we can bound $|\Ric(\partial_t,\frac{\partial_s}{\gamma})|$ uniformly by any positive constant, while the Ricci curvatures $\Ric(\frac{\partial_s}{\gamma},\frac{\partial_s}{\gamma})$ and $\Ric(\frac{v}{\beta},\frac{v}{\beta})$ converge to $2m$ as $\varepsilon\to 0$. Finally, using the estimate $-\frac{\tilde{\alpha}_{tt}}{\tilde{\alpha}}\geq -1+\frac{\rho}{2m}$, we obtain that the Ricci curvature $\Ric(\frac{u}{\alpha},\frac{u}{\alpha})$ is bounded from below by a function that converges to
    \[ (m-1)\left(-\frac{\tan^2(s)}{\cos^2(t)}+\frac{1-\sin^2(t)\cos^2(s)}{\cos^2(t)\cos^2(s)}\right)+m\left(\frac{1}{\cos^2(t)}-\tan(t)\frac{\sin(t)}{\cos(t)}\right)-1+\frac{\rho}{2m}=2m-1+\frac{\rho}{2m} \]
    as $\varepsilon\to 0$. Hence, all the Ricci curvatures are positive for $\varepsilon$ sufficiently small.
\end{proof}

We now consider a finite collection $W_1,\dots,W_{\ell}\subseteq\R^{4m+2}$ of oriented $(2m+1)$-dimensional subspaces such that each pairwise intersection $W_i\cap W_j$, $i\neq j$ has dimension at most $1$. We define the intersection matrix $A=(a_{ij})\in\{-1,0,1\}^{\ell\times\ell}$ as follows:
    \[
        a_{ij}=\begin{cases}
            1,\quad & \dim(W_i\cap W_j)=0\text{ and the oriented intersection number of }W_i\text{ and }W_j\text{ is }1,\\
            -1,\quad & \dim(W_i\cap W_j)=0\text{ and the oriented intersection number of }W_i\text{ and }W_j\text{ is }-1,\\
            0,\quad &\text{else.}
        \end{cases}
    \]
    Recall that the oriented intersection number of $W_i$ and $W_j$ is obtained by the sign of the determinant of the matrix
    \[ \begin{pmatrix}
        B_i & B_j
    \end{pmatrix}, \]
    where $B_i$, resp.\ $B_j$, is a $(4m+2)\times(2m+1)$-matrix whose columns form an oriented basis of $W_i$, resp.\ $W_j$.

    Our goal is now to slightly move the subspaces $W_i$  using Proposition \ref{P:unlink} so that the corresponding spheres do not intersect each other. To ensure that we can apply Proposition \ref{P:unlink} for each intersection of subspaces, we consider a graph $G_A=(V,E)$, where we define $V=\{1,\dots,\ell\}$ as the set of vertices and
    \[ E=\{\{i,j\}\mid i\neq j\text{ and }a_{ij}=0 \} \]
    as the set of edges. Recall that a \emph{clique} of a graph is a complete subgraph.

    \begin{proposition}\label{P:W_i_graph}
        Let $W_1,\dots,W_\ell\subseteq\R^{4m+2}$ be oriented $(2m+1)$-dimensional linear subspaces, let $A$ denote their intersection matrix and $G_A=(V_A,E_A)$ the corresponding graph. Suppose the following:
        \begin{enumerate}
            \item Every simple closed path in $G_A$ is contained in a clique,
            \item For every clique of $G_A$ the corresponding subspaces are contained in a subspace of codimension 1.
        \end{enumerate}
        Then there exist oriented $(2m+1)$-dimensional affine subspaces $W_1',\dots,W_\ell'\subseteq \R^{4m+2}$ with intersection matrix $A$ and a metric of positive Ricci curvature on $S^{4m+1}$ such that the intersections $W_i'\cap S^{4m+1}$ are pairwise disjoint spheres that are round and totally geodesic.
    \end{proposition}

    \begin{proof}
        We will apply Proposition \ref{P:unlink} to slightly move the subspaces $W_{i}$ so that they only intersect in at most one point, and hence the corresponding spheres do not intersect each other, such that the oriented intersection matrix of the subspaces $W_i$ is given by $A$.
        
        We will assume that $G_A$ is connected, otherwise we apply the same argument to each connected component. We now pick a subspace $W_{i_0}$ in $G_A$ and construct a sequence of trees $T_0\subseteq T_1\subseteq \dots$ as follows.
        
        The vertices of $T_j$ are a subset of the set of maximal cliques of $G_A$ together with an additional vertex, which will be the root of all $T_j$. The trivial tree $T_0$ consists of only the root. $T_1$ is obtained from $T_0$ by adding all maximal cliques of $G_A$ that contain the vertex $i_0$ and connecting these to the root. Next, we construct inductively the tree $T_{j+1}$ from $T_j$ by adding all maximal cliques of $G_A$ that intersect a leaf of $T_j$ as vertices and connecting them in $T_{j+1}$. Since $G_A$ is finite, we have $T_j=T_{j+1}$ for all $j$ sufficiently large, so we obtain a finite tree $T=\bigcup_j T_j$, in which the subtree of vertices of distance at most $j$ from the root is given by $T_j$.
        
        We now apply the separation process of Proposition \ref{P:unlink} repeatedly to the subspaces $W_1,\dots,W_\ell$ by using the tree $T$ as follows. We start with a leaf of $T$, which is a maximal clique in $G_A$. For this maximal clique there exists precisely one vertex that also belongs to a different maximal clique, since otherwise we could construct a simple closed path in $G_A$ that is not entirely contained in maximal clique, which would contradict (1). We fix the corresponding subspace to this vertex and apply Proposition \ref{P:unlink} to all other subspaces corresponding to vertices in this maximal clique with respect to the same codimension-1 subspace $\hat{W}$, which exists by (2). By choosing the values of $\varepsilon$ and $\delta$ in Proposition \ref{P:unlink} for each subspace $W_i$ accordingly, we can achieve that the neighbourhoods $A_{\varepsilon,6\varepsilon}(\hat{S})\cap B_{\delta}(S_i)$ do not intersect each other and also do not intersect any of the remaining subspaces. In this way, all subspaces $W_i$ corresponding to vertices in this maximal clique do not intersect each other anymore after this process, all the while the intersection number with any other subspace remains unchanged.

        Next, we remove the chosen leaf from $T$ and pick a new leaf. Property (1) now again ensures that precisely one vertex in this maximal clique is contained in a different maximal clique that is a vertex in $T$, thus we can apply the same arguments. We repeat this process until only the root is left. The required embeddings $\varphi_i\colon S^{2m}\hookrightarrow S^{4m+1}$ are now given by the intersections of each subspace $W_i$ with $S^{4m+1}$, which concludes the proof.
    \end{proof}

\subsection{Antisymmetric integer matrices}

In this section, we determine the normal forms of certain antisymmetric integer matrices.

Recall that any antisymmetric integer matrix $A\in\Z^{N\times N}$ is equivalent to a block-diagonal matrix
\[ D_N(n_1,\dots,n_k)=\begin{pmatrix}
    K_{n_1} &        &         &   &        &     \\
            & \ddots &         &   &        &     \\
            &        & K_{n_k} &   &        &     \\
            &        &         & 0 &        &     \\
            &        &         &   & \ddots &     \\
            &        &         &   &        & 0   
\end{pmatrix} \]
where each $K_{n_j}$ is of the form
\[ K_{n}=\begin{pmatrix}
    0 & n\\
    -n & 0
\end{pmatrix}\ \text{ with }n\in\Z\text{ positive.} \]

For $n,\ell\in\N$ with $n\leq 2\ell-2$ if $n$ is even and $n\leq 2\ell-1$ if $n$ is odd, we define the matrix $A_{n,\ell}$ as the $(2\ell\times2\ell)$-matrix
\[ A_{n,\ell}=\begin{pmatrix}
    S_{2\ell-1} & v_{n,\ell}\\
    -v_{n,\ell}^T & 0
\end{pmatrix}, \]
where $S_{2\ell-1}$ is the antisymmetric $(2\ell-1)\times(2\ell-1)$-matrix with all entries above the diagonal equal to $1$, and $v_{n,\ell}\in\{-1,0,1\}^{2\ell-1}$ is defined by
\[ (v_{n,\ell})_i=\begin{cases}
    (-1)^{i-1},\quad & i\leq n,\\
    0,\quad & i=2\ell-1\text{ and }n\text{ even},\\
    1,\quad & \text{else.}
\end{cases} \]

\begin{lemma}\label{L:Anl}
    The matrix $A_{n,\ell}$ is equivalent to the diagonal matrix $D_{2\ell}(1,\overset{\ell-1}{\cdots},1,n)$.
\end{lemma}
\begin{proof}
    By performing simultaneous row and column operations, we obtain that a matrix of the form
    \[ \begin{pmatrix}
        S_{2\ell-1} & v\\
        -v^T & 0
    \end{pmatrix} \]
    for any $v=(v_1,\dots,v_{2\ell-1})\in\R^{2\ell-1}$ is equivalent to
    \[
        \begin{pmatrix}
            0 & 1 &   & \\
            -1 & 0 &  & \\
            & & S_{2\ell-3} & v'\\
             &  & -v'^T & 0
      \end{pmatrix},
    \]
    where $v'=(v_3+(v_1-v_2),\dots,v_{2\ell-1}+(v_1-v_2))\in\R^{2\ell-3}$. Applying this $\ell-1$ times results in the matrix $D_{2\ell}(1,\overset{\ell-1}{\cdots},1,-\sum_{i=1}^{2\ell-1}(-1)^i v_i)$.
    
    In the case of the matrix $A_{n,\ell}$, we have
    \begin{align*}
        -\sum_{i=1}^{2\ell-1}(-1)^i (v_{n,\ell})_i&=-\sum_{i=1}^n(-1)^i(v_{n,\ell})_i-\sum_{i=n+1}^{2\ell-1}(-1)^i(v_{n,\ell})_i=\begin{cases}
        n-\sum_{i=n+1}^{2\ell-1}(-1)^i,\quad & n\text{ odd},\\
        n-\sum_{i=n+1}^{2\ell-2}(-1)^i,\quad & n\text{ even},
    \end{cases}\\
    &=n.
    \end{align*}
\end{proof}

Now let $\nu=(n_1,\dots,n_k)\in\Z^{k}_{>0}$. For $\ell\geq \frac{1}{2}\max\{n_j+1,k\}$, we construct the $(2\ell k\times 2\ell k)$-matrix $B_{\nu,\ell}$ inductively by setting $B_{\nu,\ell}=A_{n_1,\ell}$ if $k=1$ and
\[ 
    B_{\nu,\ell}=\begin{pmatrix}
        B_{\nu',\ell} & C_{\nu',\ell} \\
        -C_{\nu',\ell}^T & A_{n_k,\ell}
    \end{pmatrix}
\]
where $\nu'=(n_1,\dots,n_{k-1})$ and $C_{\nu',\ell}$ is the $2\ell(k-1)\times2\ell$-matrix where each column is equal to the $(k-1)$-st column of $B_{\nu',\ell}$.

\begin{lemma}\label{L:B_normal_form}
    The matrix $B_{\nu,\ell}$ is equivalent to $D_{2\ell k}(1,\overset{k(\ell-1)}{\cdots},1,n_1,\dots,n_k)$.
\end{lemma}
\begin{proof}
    By subtracting the $(k-1)$-st column of $B_{\nu,\ell}$ from the $i$-th column for all $2\ell(k-1)+1\leq i\leq 2\ell k$, and similarly subtracting the $(k-1)$-st row from the $i$-th row, we can eliminate $C_{\nu',\ell}$. Since the $(k-1)$-st row of $C_{\nu',\ell}$ (and therefore also the $(k-1)$-st column of $-C_{\nu',\ell}^T$) consists entirely of zeros, this operation does not affect $A_{n_k,\ell}$. Hence, the matrix $B_{\nu,\ell}$ is equivalent to
    \[ 
    \begin{pmatrix}
        B_{\nu',\ell} &  \\
         & A_{n_k,\ell}
    \end{pmatrix}.
    \]
    Repeating this argument $(k-1)$-times then results in the matrix
    \[
        \begin{pmatrix}
            A_{n_1,\ell} & & &\\
             & A_{n_2,\ell} & & &\\
             & & \ddots & \\
             & & & A_{n_k,\ell}
        \end{pmatrix}.
    \]
    Since each $A_{n_i,\ell}$ is equivalent to $D_{2\ell}(1,\overset{\ell-1}{\cdots},1,n_i)$ by Lemma \ref{L:Anl}, the claim follows.
\end{proof}

\subsection{Intersection matrices}

In this section, we show that the matrices $B_{\nu,\ell}$ can be realised as intersection matrices of linear subspaces that satisfy the hypotheses of Proposition \ref{P:unlink}.

\begin{lemma}\label{L:W_iB}
    Let $\nu=(n_1,\dots,n_k)\in\Z_{>0}^k$ and $\ell\geq \frac{1}{2}\max\{n_j+1,k\}$. Then there exist oriented $(2m+1)$-dimensional subspaces $W_1,\dots,W_{2\ell k}$ of $\R^{4m+2}$ with oriented intersection matrix given by $B_{\nu,\ell}$.
\end{lemma}

\begin{proof}
    Note that, after applying an automorphism of $\R^{4m+2}$, for every finite set of subspaces $W_i$, every $W_i$ has a basis given by the columns of the matrix
    \[ \begin{pmatrix}
        P_i\\I_{2m+1},
    \end{pmatrix} \]
    where $P_i\in\R^{(2m+1)\times (2m+1)}$, and $I_{2m+1}\in\R^{(2m+1)\times(2m+1)}$ is the identity matrix. One therefore defines the subspaces $W_i$ by specifying matrices $P_i$, and the orientation will be induced by the columns of the above matrix. Note also that a matrix of the form
    \[ \begin{pmatrix}
        P & Q\\
        I_{2m+1} & I_{2m+1}
    \end{pmatrix} \]
    with $P,Q\in\R^{(2m+1)\times(2m+1)}$ has determinant $\det(P-Q)$, so that the oriented intersection number of the subspaces $W_i$ and $W_j$ is given by $\mathrm{sgn}(\det(P_i-P_j))$.

    We start by constructing subspaces $W_1,\dots,W_{2\ell}$ with oriented intersection matrix $A_{n,\ell}$. For that, we set
    \[ P_{i,n}=\begin{cases}
        \mathrm{diag}(-i,i,\dots,i),\quad & i\text{ odd and }i\leq n,\\
        \mathrm{diag}(i,-i,i,\dots,i),\quad & i\text{ even and }i\leq n,\\
        \mathrm{diag}(i,i,\frac{1}{i},\dots,\frac{1}{i}),\quad & n+1\leq i\leq 2\ell-2,\\
        \mathrm{diag}(2\ell-1,2\ell-1,\frac{1}{2\ell-1},\dots,\frac{1}{2\ell-1}),\quad & i=2\ell-1\text{ and }n\text{ odd},\\
        \mathrm{diag}(2\ell-1,2\ell-1,0,\dots,0),\quad & i=2\ell-1\text{ and }n\text{ even},\\
        \mathrm{diag}(-2\ell,0,\dots,0),\quad & i=2\ell \text{ and }n\text{ odd}.\\
        \mathrm{diag}(-2\ell+1,0,\dots,0),\quad & i=2\ell\text{ and }n\text{ even},\\
    \end{cases} \]
    A computation shows that we obtain the matrix $A_{n,\ell}$ with this choice of subspaces.

    For a vector $v=(v_1,\dots,v_{2m+1})\in\R^{2m+1}$ we define the matrix $Q_v$ by
    \[ Q_v=\begin{pmatrix}
        v_1 & 1 & & & \\
        v_1v_2 & v_2 & & &\\
        & & v_3 & &\\
        & & & \ddots &\\
        & & & & v_{2m+1}
    \end{pmatrix}. \]
    For $1\leq i\leq 2\ell$, we then define $Q_{i,n}$ as $Q_v$, where $v$ consists of the diagonal entries of $P_{i,n}$, i.e.\ $P_{i,n}=\mathrm{diag}(v_1,\dots,v_{2m+1})$.

    We now set $P_i=P_{i,n_1}$ for $1\leq i\leq 2\ell$. Given $\varepsilon>0$, for each $j\in\{1,\dots,k-1\}$ and $2j\ell+1\leq i\leq 2(j+1)\ell$, we define
    \[ P_i=P_{j,n_1}+\varepsilon Q_{i-2j\ell,n_{j+1}}. \]
    For $\varepsilon$ sufficiently small, we then obtain the intersection matrix $B_{\nu,\ell}$. This can be seen from the inductive definition of $B_{\nu,\ell}$ as follows.

    Set $\nu_j=(n_1,\dots,n_j)$ and recall that $B_{\nu_{j+1},\ell}$ is obtained from $B_{\nu_{j},\ell}$ by setting
    \[ 
    B_{\nu_{j+1},\ell}=\begin{pmatrix}
        B_{\nu_{j},\ell} & C_{\nu_{j},\ell} \\
        -C_{\nu_{j},\ell}^T & A_{n_{j+1},\ell}
    \end{pmatrix},
\]
    where every colummn of $C_{\nu_{j},\ell}$ is defined as the $j$-th column of $B_{\nu_{j},\ell}$. Since $P_i\to P_{j,n_1}$ as $\varepsilon\to 0$ and since the subspaces defined by $P_i$ and $P_{j,n_1}$ intersect (as $\det(Q_v)=0$ for any $v$), we obtain that the oriented intersection matrix is indeed of the form
    \[ \begin{pmatrix}
        B_{\nu_{j},\ell} & C_{\nu_{j},\ell}\\
        -C_{\nu_{j},\ell}^T & A'
    \end{pmatrix}\]
    for $\varepsilon$ sufficiently small, where $A'\in\R^{(2m+1)\times (2m+1)}$. Furthermore, since $\det(P_i-P_{i'})$ for $2j\ell\leq i,i'\leq 2(j+1)\ell$ is given by 
    \[ \det(P_i-P_{i'})=\varepsilon^{2m+1}\det(Q_{i-2j\ell,n_{j+1}}-Q_{i'-2j\ell,n_{j+1}})=\varepsilon^{2m+1}\det(P_{i-2j\ell,n_{j+1}}-P_{i'-2j\ell,n_{j+1}}), \]
    we obtain the matrix $A'=A_{n_{j+1},\ell}$ .
\end{proof}

\begin{lemma}\label{L:conn_comp}
    The connected components of the graph associated to $B_{\nu,\ell}$ are of the form
    \begin{center}
    \raisebox{-.5\height}{
    \begin{tikzpicture}
			\begin{scope}[every node/.style={minimum height=2em}]
                \node (Godd) at (-0.8,0) {$G_{1}=$};
                \node[circle, fill, draw, inner sep=0, minimum size=4pt] (u) at (0,0) {};
                \node[circle, fill, draw, inner sep=0, minimum size=4pt] (v1) at (1,0) {};
			 \end{scope}
            \path[-](u) edge (v1);
		\end{tikzpicture}
  }
    \hspace{0.5cm} , \hspace{0.5cm}
    \raisebox{-.5\height}{
    \begin{tikzpicture}
			\begin{scope}[every node/.style={minimum height=2em}]
                \node (Godd) at (-0.8,0) {$G_{odd}=$};
                \node[circle, fill, draw, inner sep=0, minimum size=4pt] (u) at (0,0) {};
                \node[circle, fill, draw, inner sep=0, minimum size=4pt] (v1) at (2,1.3) {};
                \node[circle, fill, draw, inner sep=0, minimum size=4pt] (v2) at (2,1) {};
                \node[draw=none] (dots) at (1,0) {$\rvdots_{\;2\ell}$};
                \node[circle, fill, draw, inner sep=0, minimum size=4pt] (vk) at (2,-1) {};
			 \end{scope}
            \path[-](u) edge (v1);
            \path[-](u) edge (v2);
            \path[-](u) edge (vk);
		\end{tikzpicture}
  }
        \hspace{0.5cm} or \hspace{0.5cm}
    \raisebox{-.5\height}{
        \begin{tikzpicture}
			\begin{scope}[every node/.style={minimum height=2em}]
                \node (Godd) at (-1.8,0) {$G_{ev}=$};
                \node[circle, fill, draw, inner sep=0, minimum size=4pt] (u1) at (0,0) {};
                \node[circle, fill, draw, inner sep=0, minimum size=4pt] (u2) at (-1,0.7) {};
                \node[circle, fill, draw, inner sep=0, minimum size=4pt] (u3) at (-1,-0.7) {};
                \node[circle, fill, draw, inner sep=0, minimum size=4pt] (v1) at (2,1.3) {};
                \node[circle, fill, draw, inner sep=0, minimum size=4pt] (v2) at (2,1) {};
                \node[draw=none] (dots) at (1,0) {$\rvdots_{\;2\ell-2}$};
                \node[circle, fill, draw, inner sep=0, minimum size=4pt] (vk) at (2,-1) {};
			 \end{scope}
            \path[-](u1) edge (v1);
            \path[-](u1) edge (v2);
            \path[-](u1) edge (vk);
            \path[-](u1) edge (u2);
            \path[-](u1) edge (u3);
            \path[-](u2) edge (u3);
		\end{tikzpicture}
    },
    \end{center}
    where $G_1$ appears if and only if $n_1$ is even, and the number of connected components of the form $G_{odd}$ and $G_{ev}$ is given by the number of odd and even numbers among $n_2,\dots,n_k$, respectively.
\end{lemma}
\begin{proof}
    By construction, the matrix $A_{n,\ell}$ only has an entry above the diagonal equal to zero when $n$ is even, in which case the zero entry is at position $(2\ell-1,2\ell)$. Moreover, the only zero entries of the matrix $C_{\nu_j,\ell}$ are at position $(j,i)$, where $1\leq i\leq 2\ell$.

    Thus, for every $j\in\{1,\dots,k-1\}$, we obtain zero entries in the matrix $B_{\nu,\ell}$ at the positions $(j,2\ell j+i)$ for all $1\leq i\leq 2\ell$, and additionally also at $(2\ell (j+1)-1,2\ell (j+1))$ when $n_{j+1}$ is even.
    
    Hence, we obtain a connected component of the form $G_1$ when $n_1$ is even. In this case, the two vertices correspond to the subspaces $W_{2\ell-1}$ and $W_{2\ell}$. Furthermore, for any $j\in\{1,\dots,k-1\}$ we obtain a connected component of the form $G_{odd}$ when $n_{j+1}$ is odd, where the vertex on the left-hand side is represented by $W_j$, and a graph of the form $G_{ev}$ when $n_{j+1}$ is even, where the vertex in middle is represented by $W_j$, and the vertices on the left-hand side by $W_{2\ell(j+1)-1}$ and $W_{2\ell (j+1)}$.
\end{proof}

\begin{lemma}\label{L:Wi_admissible}
    The subspaces $W_1,\dots,W_{2\ell k}$ constructed in Lemma \ref{L:W_iB} satisfy the requirements of Proposition~\ref{P:W_i_graph}.
\end{lemma}
\begin{proof}
    The first property of Proposition \ref{P:W_i_graph} follows from Lemma \ref{L:conn_comp}. For the second property, by (the proof of) Lemma \ref{L:conn_comp}, we need to consider the subspaces $W_j, W_{2\ell(j+1)-1}, W_{2\ell(j+1)}$ for all $j\in\{1,\dots,k-1\}$ for which $n_{j+1}$ is even. These are defined by the matrices $P_{j,n_1}$, $P_{2\ell(j+1)-1}=P_{j,n_1}+\varepsilon Q_{2\ell-1,n_{j+1}}$ and $P_{2\ell(j+1)}=P_{j,n_1}+\varepsilon Q_{2\ell,n_{j+1}}$, respectively. Hence, we need to determine the rank of the matrix
    \[
        \begin{pmatrix}
            P_{j,n_1} & P_{j,n_1}+\varepsilon Q_{2\ell-1,n_{j+1}} & P_{j,n_1}+\varepsilon Q_{2\ell,n_{j+1}}\\
            I_{2m+1} & I_{2m+1} & I_{2m+1}
        \end{pmatrix}.
    \]
    By applying column operations, we obtain that this matrix has the same rank as the matrix
    \[
        \begin{pmatrix}
            P_{j,n_1} & \varepsilon Q_{2\ell-1,n_{j+1}} & \varepsilon Q_{2\ell,n_{j+1}}\\
            I_{2m+1} & 0 & 0
        \end{pmatrix}.
    \]
    By the definitions of $P_{2\ell-1,n_{j+1}}$ and $P_{2\ell,n_{j+1}}$, the matrices $Q_{2\ell-1,n_{j+1}}$ and $Q_{2\ell,n_{j+1}}$ only have non-zero entries in the upper-left $2\times 2$-block. Hence, this matrix has rank at most $2m+3\leq 4m+1$.    
\end{proof}

\subsection{Proof of Theorem \ref{T:highly-conn}}

Let $M^{4m+1}$ be a closed, $(2m-1)$-connected $2m$-parallelisable manifold. By \cite[Theorem 7.1]{CW17}, there exists a homotopy sphere $\Sigma^{4m+1}$ such that $M\# \Sigma$ is the boundary of a handlebody $W$. We will now construct embeddings $\varphi_i\colon S^{2m}\times D^{2m+1}\hookrightarrow S^{4m+1}$ that induce the same invariants as $W$, such that we can perform surgeries along these embeddings while preserving $\BE_\infty>0$.

Since $\lambda_W$ is an antisymmetric bilinear form, there exists a basis of $H_W$ in which $\lambda_W$ is given by
\[ D_N(n_1,\dots,n_k)=\begin{pmatrix}
    K_{n_1} &        &         &   &        &     \\
            & \ddots &         &   &        &     \\
            &        & K_{n_k} &   &        &     \\
            &        &         & 0 &        &     \\
            &        &         &   & \ddots &     \\
            &        &         &   &        & 0   
\end{pmatrix}, \]
where $N=\dim(H_W)$. We set $\nu=(n_1,\dots,n_k)$, and, for $\ell$ sufficiently large, consider the matrix $B_{\nu,\ell}$.

By Lemma \ref{L:B_normal_form}, there exists a matrix $T\in GL(2\ell k,\Z)$ that carries the matrix $B_{\nu,\ell}$ into
\[D=D_{2\ell k}(1,\overset{k(\ell-1)}{\cdots},1,n_1,\dots,n_k).\]
Furthermore, by Lemmas \ref{L:W_iB} and \ref{L:Wi_admissible} and Proposition \ref{P:W_i_graph}, there exists a metric of positive Ricci curvature on $S^{4m+1}$ and embeddings $\bar{\varphi}_i\colon D^{2m+1}\hookrightarrow D^{4m+2}$, $1\leq i\leq 2\ell k$, with oriented intersection matrix $B_{\nu,\ell}$ and such that each restriction $\bar{\varphi}_i|_{S^{2m}}\colon S^{2m}\hookrightarrow S^{4m+1}$ is round and totally geodesic. Moreover, the embeddings $\bar{\varphi}_i|_{S^{2m}}$ have pairwise disjoint image. We extend each embedding $\bar{\varphi}_i|_{S^{2m}}$ to an embedding $\varphi_i\colon S^{2m}\times D^{2m+1}\hookrightarrow S^{4m+1}$ such that the invariants of the handlebody obtained from $\varphi_1,\dots,\varphi_\ell$ are given by
\[ (\Z^{2\ell k},B_{\nu,\ell},T^{-1}\mu), \]
where $\mu=(0,\overset{k(\ell-1)}{\cdots},0,\mu_W)$.
 
By Theorem \ref{T:HIGHER_SURG}, the manifold $M_0$ obtained from $S^{4m+1}$ by surgery along the embeddings $\varphi_i$ admits a weighted Riemannian metric of $\BE_\infty>0$. Further, it is the boundary of the handlebody with invariants
\[ (\Z^{2\ell k},B_{\nu,\ell},T^{-1}\mu)\sim (\Z^{2\ell k},D,\mu). \]
By Theorem \ref{T:handlebody_class} and Lemma \ref{L:handlebody_examples}, there exists a manifold $M_1$ which is the connected sum of total spaces of linear $S^{2m+1}$-bundle over $S^{2m}$, and a homotopy sphere $\Sigma'$ such that $M\#\Sigma\#\Sigma'$ is diffeomorphic to $M_0\# M_1$. Finally, by Theorem \ref{T:CONN_SUMS} and \ref{EQ:core2}, the manifold $M_0\# M_1$ admits a weighted Riemannian metric of $\BE_\infty>0$.

\hfill $\square$

\subsection{Simply-connected 5-manifolds}\label{SS:5-manifolds}

Theorem \ref{T:highly-conn} implies that every closed, simply-connected spin 5-manifold admits a weighted Riemannian metric of $\BE_\infty>0$. In this subsection, we consider an extension of this result to certain non-spin manifolds.

For that, we first recall the classification of closed, simply-connected 5-manifolds by Barden \cite{Ba65} and Smale \cite{Sm62}. For $j\in\{-1,0,\dots,\infty\}$ there exists a closed, simply-connected 5-manifold $X_j$ \cite[Section 1]{Ba65} satisfying
\[ H_2(X_j)\cong \bigslant{\Z}{2^j}\oplus\bigslant{\Z}{2^j} \]
for $0\leq j<\infty$, and $H_2(X_{-1})\cong\Z/2$, $H_2(X_\infty)\cong\Z$. Furthermore, the second Stiefel--Whitney class $w_2(X_j)$ is nontrivial if and only if $j\neq0$. The classification is now given as follows:
\begin{theorem}[{\cite{Sm62}, \cite[Theorem 2.3]{Ba65}}]\label{T:class_dim5}$\phantom{=}$
    \begin{enumerate}
        \item Every closed, simply-connected spin 5 manifold is uniquely determined by its second homology group. A finitely generated abelian group $G$ can be realised by such a manifold if and only if there exists a finite abelian group $G_T$ such that $\mathrm{Tors}(G)\cong G_T\oplus G_T$.
        \item Every closed, simply-connected 5-manifold uniquely splits as $M\cong X_j\# M_0$ where $M_0$ is spin.
    \end{enumerate}
\end{theorem}

In particular, the manifolds $X_0$ and $X_\infty$ are the sphere $S^5$ and the total space of the unique non-trivial linear $S^3$-bundle over $S^2$, respectively. The manifold $X_{-1}$ is the Wu manifold $SU(3)/SO(3)$ and $X_1=X_{-1}\# X_{-1}$. By Theorems \ref{T:CONN_SUMS} and \ref{T:dim-5} together with \ref{EQ:core2} and \ref{EQ:core6} we obtain the following result.

\begin{theorem}\label{T:5-mfds}
    Let $M$ be a closed, simply-connected spin 5-manifold. Then $X_j\# M$ admits a weighted Riemannian metric of $\BE_\infty>0$ for all $j\in\{-1,0,1,\infty\}$.
\end{theorem}

For comparison, we have the following known examples of closed, simply-connected 5-manifolds with a Riemannian metric of positive Ricci curvature:
\begin{enumerate}
    \item All manifolds of the form $X_j\# M_0$ in Theorem \ref{T:class_dim5} where $j\in\{-1,0,1,\infty\}$ and $M_0$ is spin and has torsion-free homology (see \cite{SY91} or \cite{CG20} for $j\in\{0,\infty\}$, and \cite{Re24} for $j\in\{-1,1\}$).
    \item Closed, simply-connected 5-manifolds with positive Sasakian structures. These manifolds are all spin, have second Betti number at most 8 and torsion group of the form $(\Z/m)^{2\ell}$ (see \cite{Ko09}, \cite{BG02,BG06a} and \cite[Corollary 10.2.20, Theorem 10.2.25 and Table B.4.2]{BG08} for their classification).
\end{enumerate}

\appendix

\section{$\BE_q>0$ vs.\ $\Ric>0$}
\label{A:BE_Ric}

In this section, we collect results that allow us to construct Riemannian metrics of $\Ric>0$ from weighted Riemannian metrics of $\BE_q>0$ for some $q$. In general, we are interested in the following question.

\begin{question}\label{Q:BE>0_Ric>0}
    Given a closed, weighted Riemannian manifold $(M,g,e^{-f})$ with $\BE_q>0$ for some $q\in(0,\infty]$, does there exist a Riemannian metric $\tilde{g}$ on $M$ with $\Ric>0$?
\end{question}
To the best of our knowledge, there is no counterexample known for this question and it is known that the classical obstructions for $\Ric>0$ also hold for $\BE_q>0$:
\begin{proposition}\label{P:BE_TOP}
    Let $(M^n,g,e^{-f})$ be a weighted Riemannian manifold with $\BE_q>0$ such that $M$ is closed. Then
    \begin{enumerate}
        \item $M$ has finite fundamental group,
        \item If $M$ is spin, then the $\alpha$-invariant $\alpha(M)$ vanishes provided $q\leq 4$.
    \end{enumerate}
    In particular, if $M$ is simply-connected (spin or non-spin) with $n\neq 4$ and $q\leq 4$, then it admits a Riemannian metric of positive scalar curvature.
\end{proposition}
\begin{proof}
    Items (1) and (2) are shown in \cite[Theorem 1]{Lo03} and \cite[Corollary 4.4]{De21}, respectively. The last statement follows from the fact that any closed, simply-connected manifold of dimension at least $5$ admits a metric of positive scalar curvature if and only if it is non-spin \cite{GL80a} or spin with vanishing $\alpha$-invariant \cite{St92}, and the only closed, simply-connected manifolds in dimensions $2$ and $3$ are spheres.
\end{proof}
For further generalisations of results from $\Ric>0$ to $\BE_q>0$, such as the Bonnet--Myers theorem, the Cheeger--Gromoll splitting theorem, and the Bishop--Gromov volume comparison theorem,  we refer to \cite{Lo03}, \cite{Mo05}, \cite{WW09}, \cite{WY16}, and the references therein. We also refer to \cite{KW17}, \cite{KWY19}, \cite{Wy15} for results on positive weighted \emph{sectional} curvature.

Given the result of Proposition \ref{P:BE_TOP}, it is not clear, however, how one can construct a metric of positive scalar curvature from a weighted metric with positive weighted Ricci curvature. Note that in the special case where the weighted Ricci curvature is constant, the metric itself already has positive scalar curvature by \cite[Proposition 1.1]{CSW11}. In general, however, one can construct examples, where the metric $g$ even has negative sectional curvature, see e.g.\ \cite[Example 2.2]{WW09}. On the other hand, for closed manifolds, there exists at least a point of positive Ricci curvature:
\begin{theorem}
    Let $(M,g,e^{-f})$ be a closed, weighted Riemannian manifold of $\BE_q>0$. Then there exists a point in $M$ at which $g$ has $\Ric>0$.
\end{theorem}
This follows from the fact that $M\times S^p$ admits a submersion metric of $\Ric>0$ for all $p\geq\max\{2,q\}$ by \cite[Section 2]{Lo03} or Proposition \ref{P:BE_bdle} below (where we can replace $q$ by a finite value in case $q=\infty$ since $M$ is compact), together with \cite[Theorem 2]{PW14}.

A partial positive answer to question \ref{Q:BE>0_Ric>0} was given by Wylie and Yeroshkin \cite{WY16}. For a one form $\alpha$ on a Riemannian manifold $(M,g)$ they defined the torsion-free connection
\[\nabla^\alpha_X Y=\nabla_X Y-\alpha(X)Y-\alpha(Y)X, \]
where $\nabla$ denotes the Levi--Civita connection of $g$.

\begin{theorem}[{\cite[Theorem 2.15]{WY16}}]\label{T:weighted_conn}
    Let $(M^n,g,e^{-f})$ be a weighted Riemannian manifold with $\BE_{1-n}>0$, such that the holonomy of the connection $\nabla^{df}$ is compact. Then there exists a Riemannian metric $\tilde{g}$ on $M$ that is compatible with $\nabla^{df}$ and any such metric satisfies $\Ric>0$.
\end{theorem}
For this theorem, one needs to extend the definition of $\BE_q$ to $q<0$ in the obvious way. We also note that in \cite[Section 5.1]{WY16} several examples are given where the holonomy of $\nabla^{df}$ is not compact, showing that Theorem \ref{T:weighted_conn} does not provide a full answer to Question \ref{Q:BE>0_Ric>0}.

As a variation of Question \ref{Q:BE>0_Ric>0}, we can ask whether the existence of a weighted Riemannian metric of $\BE_q>0$ on a manifold $M$ implies the existence of a Riemannian metric of $\Ric>0$ on some higher-dimensional manifold obtained from $M$, such as a fibre bundle with base $M$. It was observed by Lott \cite[Section 2]{Lo03} that this holds in the special case of a product $M\times S^p$ whenever $p\geq \max\{2,q\}$, i.e.\ if $M$ admits a weighted Riemannian metric of $\BE_q>0$, then $M\times S^p$ admits a Riemannian metric of $\Ric>0$ whenever $p\geq \max\{2,q\}$. We can generalise this as follows.
\begin{proposition}\label{P:BE_bdle}
    Let $M$ be a closed manifold that admits a weighted Riemannian metric $(g,e^{-f})$ of $\BE_q>0$. Let $(N^p,\hat{g})$ be a manifold of positive Ricci curvature and let $E\xrightarrow{\pi} M$ be a fibre bundle with fibre $N$ such that the structure group of the bundle acts via isometries on $(N,\hat{g})$. If $p\geq q$, then $E$ admits a submersion metric of positive Ricci curvature.
\end{proposition}

In particular, Proposition \ref{P:BE_bdle} applies when $\pi$ is a linear sphere bundle, or a trivial bundle $M\times N\to M$ and $N$ is a closed manifold that admits a Riemannian metric of $\Ric>0$ (and in both cases we assume that the fibre dimension is at least $\max\{2,q\}$).

\begin{proof}
    Let $\mathcal{V}=\ker(d\pi)\subseteq TE$ be the vertical distribution of $\pi$. By choosing a principal connection on the associated principal $G$-bundle, where $G$ denotes the structure group of $\pi$, we obtain a Riemannian metric $\bar{g}$ on $E$ such that $(E,\bar{g})\xrightarrow{\pi}(M,g)$ is a Riemannian submersion with totally geodesic fibres isometric to $(N,\hat{g})$, see e.g.\ \cite[Theorem 9.59]{Be87}. We denote by $\mathcal{H}=\mathcal{V}^\perp\subseteq TE$ the corresponding horizontal distribution. In the following, we will denote by $u,u_1,u_2$ horizontal vectors, and by $v,v_1,v_2$ vertical vectors. We will also assume that all horizontal vector fields $u$ we consider are \emph{basic}, i.e.\ there exists a vector field $\check{u}$ on $M$ such that $\pi_*(u_x)=\check{u}_{\pi(x)}$ for all $x\in E$. Since every vector field on $M$ uniquely lifts to a basic vector field on $E$, we can identify vector fields on $M$ and basic vector fields on $E$ in this way.

    For a smooth function $F\colon M\to(0,\infty)$ we now define the metric $\bar{g}_F$ on $E$ as the metric obtained from $\bar{g}$ by scaling the fibres by $F^2$, i.e.\
    \[ \bar{g}_F|_{\mathcal{V}}=F(\pi)^2\bar{g},\quad \bar{g}_F|_{\mathcal{H}}=\bar{g}|_{\mathcal{H}},\quad \bar{g}_F(\mathcal{H},\mathcal{V})=0. \]
    Then $(E,\bar{g}_F)\xrightarrow{\pi}(M,g)$ is again a Riemannian submersion. However, the fibres do not need to be totally geodesic. Indeed, if $\overline{\nabla}$ and $\overline{\nabla}^F$ denote the Levi--Civita connections of $\bar{g}$ and $\bar{g}_F$, respectively, then it follows from the Koszul formula that
    \[ \bar{g}_F\left(\overline{\nabla}^F_{v_1}u,v_2\right)=\bar{g}_F\left(\overline{\nabla}_{v_1}u+\frac{u(F)}{F}v_1,v_2\right). \]
    Since fibres of $\pi$ are totally geodesic with respect to $\bar{g}$, it follows that $\bar{g}(\overline{\nabla}_{v_1}u,v_2)=0$, and hence the $T$-tensor $T^F$ of $\bar{g}_F$ (see e.g.\ \cite[Section 9.C]{Be87}) satisfies
    \[ T^F_v u=\frac{u(F)}{F}v. \]
    By the symmetries of $T^F$ (see \cite[9.18d]{Be87}) we also have
    \[ T^F_{v_1}v_2=-F\bar{g}(v_1,v_2)\nabla F. \]
    Moreover, we have $T^F_u v=T^F_{u_1}u_2=0$ (see \cite[9.18a]{Be87}).
    
    We use this to calculate the \emph{mean curvature vector} $\nu=\sum_i T^F_{v_i}v_i$, where $(v_i)$ is a vertical orthonormal basis with respect to $\bar{g}_F$, as follows:
    \[ \nu=-\sum_{i=1}^p\bar{g}_F(v_i,v_i)\frac{\nabla F}{F}=-p\frac{\nabla F}{F}\]
    (recall that $p$ is the dimension of $N$).
    
    Next, note that $[u,v]$ is vertical as it maps to $0$ under $\pi_*$ (here we need that $u$ is basic). We will assume for $x\in E$ that $v,v_i\in\mathcal{V}_x$ are vertical vectors at $x$ that are extended to vertical vector fields so that any covariant derivative at $x$ between two of these vector fields at $x$ is horizontal. This can for example be achieved by considering normal coordinates in the fibres and using that the Levi--Civita connections of the fibre metrics coincide with $\overline{\nabla}^F$ on $\mathcal{V}$ (see \cite[9.16]{Be87}). Then at $x$ we have the following equations (for the definitions of $\hat{\delta}$ and $T_v^F$, $A_u^F$ see \cite[9.33]{Be87}).
    \begin{align*}
        \bar{g}_F((\hat{\delta}T^F)v,u)&=-\sum_i\bar{g}_F\left( \left(\overline{\nabla}^F_{v_i}T^F \right)_{v_i}v,u \right)=-\sum_i\bar{g}_F\left( \overline{\nabla}^F_{v_i}\left(T^F_{v_i}v 
 \right),u \right)=\sum_i\bar{g}_F\left( \overline{\nabla}^F_{v_i}\left(F\bar{g}(v_i,v)\nabla F\right) ,u\right)\\
        &=\sum_i F\bar{g}(v_i,v)\bar{g}_F\left(\overline{\nabla}^F_{v_i}\nabla F,u\right)=\frac{1}{F}\bar{g}_F\left(\overline{\nabla}^F_{v}\nabla F,u \right)=-\frac{1}{2F}\bar{g}_F([\nabla F,u],v),\\
        \bar{g}_F\left(T^F_v,A^F_u\right)&=\sum_i\bar{g}_F\left( T^F_{v}v_i,A^F_u v_i \right)=-\sum_i\bar{g}_F\left( F\bar{g}(v,v_i)\nabla F,\overline{\nabla}^F_u v_i \right)=-\frac{1}{F}\bar{g}_F\left( \nabla F,\overline{\nabla}^F_u v \right)\\
        &=-\frac{1}{2F}\bar{g}_F([\nabla F,u],v),\\
        \bar{g}_F\left( T^F u_1,T^F u_2 \right)&=\sum_i\bar{g}_F\left( T^F_{v_i}u_1,T^F_{v_i}u_2 \right)=p\frac{u_1(F)u_2(F)}{F^2},\\
        \bar{g}_F\left( \overline{\nabla}^F_{u_1}\nu,u_2 \right)&=p\frac{
        u_1(F)u_2(F)}{F^2}-p\frac{\Hess(F)(u_1,u_2)}{F}.
    \end{align*}
    Here $A^F$ denotes the $A$-tensor of $\bar{g}_F$.

    We now use these equalities to analyse the Ricci curvatures of the metric $\bar{g}_F$ where we set
    \[ F=\sqrt{t}e^{-\frac{f}{p}} \]
    for some $t>0$. To simplify the notation we set $\bar{g}_t=\bar{g}_F$ and similarly $\overline{\nabla}^t=\overline{\nabla}^F$, $\nu^t=\nu^F$, $T^t=T^F$, and $A^t=A^F$. By \cite[9.36 and 9.69]{Be87}, we then have the following (note that the second summand in \cite[9.69h]{Be87}, which follows from \cite[9.69f]{Be87}, has the wrong sign).
    \begin{align*}
        \Ric^{\bar{g}_t}(v_1,v_2)&=\Ric^{\hat{g}}(v_1,v_2)-t\bar{g}_1\left(\nu^1,T^1_{v_1}v_2\right)+t^2\bar{g}_1\left(A^1v_1,A^1v_2\right)+t\left(\tilde{\delta}T^1 \right)(v_1,v_2),\\
        \Ric^{\bar{g}_t}(u,v)&=\bar{g}_1\left(\left(\hat{\delta}T^1\right)v,u \right)+\bar{g}_1\left(\overline{\nabla}^t_v \nu^1,u\right)-t\overline{g}_1\left(\left(\check{\delta}A^1\right)u,v\right)-(1+t)\overline{g}_1\left(A_u^1,T_v^1 \right)\\
        &=\bar{g}_1\left(\overline{\nabla}^t_v \nu^1,u\right)-t\overline{g}_1\left(\left(\check{\delta}A^1\right)u,v\right)-t\overline{g}_1\left(A_u^1,T_v^1 \right),\\
        \Ric^{\bar{g}_t}(u_1,u_2)&=\Ric^g(u_1,u_2)-2t\bar{g}_1\left(A^1_{u_1},A^1_{u_2} \right)-\bar{g}_1\left(T^1u_1,T^1u_2\right)+\frac{1}{2}\left( \bar{g}_1\left(\overline{\nabla}_{u_1}^t\nu^1,u_2\right)+\overline{g}_1\left(\overline{\nabla}^t_{u_2}\nu^1,u_1\right) \right)\\
        &=\Ric^g(u_1,u_2)-2t\bar{g}_1\left(A^1_{u_1},A_{u_2}^1 \right)-p\frac{\Hess(F)(u_1,u_2)}{F}.
    \end{align*}
    We have
    \begin{align*}
        \bar{g}_1\left(\overline{\nabla}^t_v \nu^1,u\right)=t\bar{g}_1\left( \overline{\nabla}^1_v \nu^1,u \right)
    \end{align*}
    (e.g.\ by \cite[9.69a]{Be87} or the Koszul formula). Further, with our choice of $F$, we have
    \begin{align*}
        \nabla F&=-\frac{\sqrt{t}e^{-\frac{f}{p}}}{p}\nabla f,\\
        \frac{\Hess(F)(u_1,u_2)}{F}&=\frac{1}{p^2}u_1(f)u_2(f)-\frac{1}{p}\Hess(f)(u_1,u_2).
    \end{align*}
    Hence, the Ricci curvatures of $\bar{g}_t$ can be written as follows:
    \begin{align*}
        \Ric^{\bar{g}_t}(v_1,v_2)&=\Ric^{\hat{g}}(v_1,v_2)+O(t),\\
        \Ric^{\bar{g}_t}(u,v)&=O(t),\\
        \Ric^{\bar{g}_t}(u_1,u_2)&=\BE{}_p^{g,f}(u_1,u_2)+O(t).
    \end{align*}
    Since $\hat{g}$ has $\Ric>0$ and $(g,e^{-f})$ has $\BE_p>0$, it follows that $\bar{g}_t$ has $\Ric>0$ for all $t$ sufficiently small.

\end{proof}

Finally, we obtain an analogous result for weighted core metric.

\begin{proposition}\label{P:WEIGHTED_CORE_BDL}
    Let $M$ be a closed manifold that admits a weighted core metric $(g,e^{-f})$ with respect to $q\in(0,\infty)$. Let $N^p$ be a closed manifold that admits a core metric $\hat{g}'$ and let $E\xrightarrow{\pi}M$ be a fibre bundle with fibre $N$ such that the structure group of the bundle acts via isometries on $(N,\hat{g})$, where $\hat{g}$ is a metric of positive Ricci curvature on $N$ that lies in the same path component as $\hat{g}'$ in the space of Ricci-positive metrics on $N$. If $p\geq \max\{3,q\}$, then $E$ admits a core metric.
\end{proposition}
In particular, the assumptions of Proposition \ref{P:WEIGHTED_CORE_BDL} are satisfied when $\pi$ is the trivial bundle $M\times N\to M$ and $N$ admits a core metric (then we can set $\hat{g}=\hat{g}'$), or when $\pi$ is a linear sphere bundle (then we can set $\hat{g}=\hat{g}'=ds_p^2$). Further, it can be applied to projective bundles with fibre $\C P^n$, $\mathbb{H}P^n$ or $\mathbb{O}P^2$, see \cite[Section 5.2]{Re24}.
\begin{proof}
    We consider the same submersion metric $\bar{g}_t$ as in the proof of Proposition \ref{P:BE_bdle}, which has positive Ricci curvature for all $t$ sufficiently small. Note that we can freely choose the principal connection on the corresponding principal $G$-bundle. Hence, if we choose a principal connection that is flat over the embedded hemisphere $\varphi(D^p)\subseteq M$, the metric $\bar{g}_t$ is a product
    \[ds_p^2|_{D^p}+te^{-\frac{f_0}{p}}\hat{g}\]
    on $\pi^{-1}(\varphi(D^p))\cong D^p\times N$, where $f_0$ is the constant value of $f$ on $\varphi(D^p)$. In particular, the boundary $\partial\pi^{-1}(\varphi(D^p))\cong S^{p-1}\times N$ is totally geodesic.

    We now consider the manifold $E\setminus\pi^{-1}(\varphi(D^p))^\circ$ equipped with the induced metric. By \cite[Proposition 1.2.11]{Bu19a} we can deform the metric $\bar{g}_t$ preserving $\Ric>0$ so that the second fundamental form on the boundary is strictly positive, and for any $r>0$, by \cite[Theorem C]{Bu20}, we can assume that the metric on the boundary is given by $ds_{p-1}^2+r^2\hat{g}'$. Then, by \cite[Theorem 4.1]{Re24}, it follows that we can glue back in $D^p\times N$ and obtain a core metric on $E$.
\end{proof}

\bibliographystyle{plainurl}
\bibliography{References}

\begin{thebibliography}{10}

\bibitem{BE85}
D.~Bakry and Michel \'{E}mery.
\newblock Diffusions hypercontractives.
\newblock In {\em S\'{e}minaire de probabilit\'{e}s, {XIX}, 1983/84}, volume
  1123 of {\em Lecture Notes in Math.}, pages 177--206. Springer, Berlin, 1985.
\newblock \href {https://doi.org/10.1007/BFb0075847}
  {\path{doi:10.1007/BFb0075847}}.

\bibitem{BH22}
Christian B\"{a}r and Bernhard Hanke.
\newblock Local flexibility for open partial differential relations.
\newblock {\em Comm. Pure Appl. Math.}, 75(6):1377--1415, 2022.
\newblock \href {https://doi.org/10.1002/cpa.21982}
  {\path{doi:10.1002/cpa.21982}}.

\bibitem{Ba65}
D.~Barden.
\newblock Simply connected five-manifolds.
\newblock {\em Ann. of Math. (2)}, 82:365--385, 1965.
\newblock \href {https://doi.org/10.2307/1970702} {\path{doi:10.2307/1970702}}.

\bibitem{Be87}
Arthur~L. Besse.
\newblock {\em Einstein manifolds}, volume~10 of {\em Ergebnisse der Mathematik
  und ihrer Grenzgebiete (3) [Results in Mathematics and Related Areas (3)]}.
\newblock Springer-Verlag, Berlin, 1987.

\bibitem{BWW19}
Boris Botvinnik, Mark~G. Walsh, and David Wraith.
\newblock Homotopy groups of the observer moduli space of {R}icci positive
  metrics.
\newblock {\em Geom. Topol.}, 23(6):3003--3040, 2019.
\newblock \href {https://doi.org/10.2140/gt.2019.23.3003}
  {\path{doi:10.2140/gt.2019.23.3003}}.

\bibitem{BG02}
Charles~P. Boyer and Krzysztof Galicki.
\newblock Rational homology 5-spheres with positive {R}icci curvature.
\newblock {\em Math. Res. Lett.}, 9(4):521--528, 2002.
\newblock \href {https://doi.org/10.4310/MRL.2002.v9.n4.a12}
  {\path{doi:10.4310/MRL.2002.v9.n4.a12}}.

\bibitem{BG06a}
Charles~P. Boyer and Krzysztof Galicki.
\newblock Erratum and addendum for: ``{R}ational homology 5-spheres with
  positive {R}icci curvature'' [{M}ath. {R}es. {L}ett. 9 (2002), no. 4,
  521--528; mr1928872].
\newblock {\em Math. Res. Lett.}, 13(2-3):463--465, 2006.
\newblock \href {https://doi.org/10.4310/MRL.2006.v13.n3.a10}
  {\path{doi:10.4310/MRL.2006.v13.n3.a10}}.

\bibitem{BG08}
Charles~P. Boyer and Krzysztof Galicki.
\newblock {\em Sasakian geometry}.
\newblock Oxford Mathematical Monographs. Oxford University Press, Oxford,
  2008.

\bibitem{Bu19a}
Bradley~Lewis Burdick.
\newblock {\em Metrics of Positive Ricci Curvature on Connected Sums:
  Projective Spaces, Products, and Plumbings}.
\newblock ProQuest LLC, Ann Arbor, MI, 2019.
\newblock Thesis (Ph.D.)--University of Oregon.
\newblock URL:
  \url{http://gateway.proquest.com/openurl?url_ver=Z39.88-2004&rft_val_fmt=info:ofi/fmt:kev:mtx:dissertation&res_dat=xri:pqm&rft_dat=xri:pqdiss:13898429}.

\bibitem{Bu19}
Bradley~Lewis Burdick.
\newblock Ricci-positive metrics on connected sums of projective spaces.
\newblock {\em Differential Geom. Appl.}, 62:212--233, 2019.
\newblock \href {https://doi.org/10.1016/j.difgeo.2018.11.005}
  {\path{doi:10.1016/j.difgeo.2018.11.005}}.

\bibitem{Bu20}
Bradley~Lewis Burdick.
\newblock Metrics of positive {R}icci curvature on the connected sums of
  products with arbitrarily many spheres.
\newblock {\em Ann. Global Anal. Geom.}, 58(4):433--476, 2020.
\newblock \href {https://doi.org/10.1007/s10455-020-09732-7}
  {\path{doi:10.1007/s10455-020-09732-7}}.

\bibitem{Bu20a}
Bradley~Lewis Burdick.
\newblock The space of positive {R}icci curvature metrics on spin manifolds.
\newblock {\em arXiv e-prints}, 2020.
\newblock \href {http://arxiv.org/abs/2009.06199} {\path{arXiv:2009.06199}}.

\bibitem{CSW11}
Jeffrey Case, Yu-Jen Shu, and Guofang Wei.
\newblock Rigidity of quasi-{E}instein metrics.
\newblock {\em Differential Geom. Appl.}, 29(1):93--100, 2011.
\newblock \href {https://doi.org/10.1016/j.difgeo.2010.11.003}
  {\path{doi:10.1016/j.difgeo.2010.11.003}}.

\bibitem{CG20}
Diego Corro and Fernando Galaz-Garc\'{\i}a.
\newblock Positive {R}icci curvature on simply-connected manifolds with
  cohomogeneity-two torus actions.
\newblock {\em Proc. Amer. Math. Soc.}, 148(7):3087--3097, 2020.
\newblock \href {https://doi.org/10.1090/proc/14961}
  {\path{doi:10.1090/proc/14961}}.

\bibitem{CW17}
Diarmuid Crowley and David Wraith.
\newblock Positive {R}icci curvature on highly connected manifolds.
\newblock {\em J. Differential Geom.}, 106(2):187--243, 2017.
\newblock \href {https://doi.org/10.4310/jdg/1497405625}
  {\path{doi:10.4310/jdg/1497405625}}.

\bibitem{CW22}
Diarmuid Crowley and David~J. Wraith.
\newblock Intermediate curvatures and highly connected manifolds.
\newblock {\em Asian J. Math.}, 26(3):407--454, 2022.
\newblock \href {https://doi.org/10.4310/ajm.2022.v26.n3.a3}
  {\path{doi:10.4310/ajm.2022.v26.n3.a3}}.

\bibitem{De21}
Jialong Deng.
\newblock Curvature-dimension condition meets {G}romov's {$n$}-volumic scalar
  curvature.
\newblock {\em SIGMA Symmetry Integrability Geom. Methods Appl.}, 17:Paper No.
  013, 20, 2021.
\newblock \href {https://doi.org/10.3842/SIGMA.2021.013}
  {\path{doi:10.3842/SIGMA.2021.013}}.

\bibitem{GY86}
L.~Zhiyong Gao and S.-T. Yau.
\newblock The existence of negatively {R}icci curved metrics on
  three-manifolds.
\newblock {\em Invent. Math.}, 85(3):637--652, 1986.
\newblock \href {https://doi.org/10.1007/BF01390331}
  {\path{doi:10.1007/BF01390331}}.

\bibitem{GL80a}
Mikhael Gromov and H.~Blaine Lawson, Jr.
\newblock The classification of simply connected manifolds of positive scalar
  curvature.
\newblock {\em Ann. of Math. (2)}, 111(3):423--434, 1980.
\newblock \href {https://doi.org/10.2307/1971103} {\path{doi:10.2307/1971103}}.

\bibitem{Hi76}
Morris~W. Hirsch.
\newblock {\em Differential topology}.
\newblock Springer-Verlag, New York-Heidelberg, 1976.
\newblock Graduate Texts in Mathematics, No. 33.

\bibitem{HNW23}
Erik Hupp, Aaron Naber, and Kai-Hsiang Wang.
\newblock Lower {R}icci curvature and nonexistence of manifold structure.
\newblock {\em arXiv e-prints}, 2023.
\newblock \href {http://arxiv.org/abs/2308.03909} {\path{arXiv:2308.03909}}.

\bibitem{KW17}
Lee Kennard and William Wylie.
\newblock Positive weighted sectional curvature.
\newblock {\em Indiana Univ. Math. J.}, 66(2):419--462, 2017.
\newblock \href {https://doi.org/10.1512/iumj.2017.66.6013}
  {\path{doi:10.1512/iumj.2017.66.6013}}.

\bibitem{KWY19}
Lee Kennard, William Wylie, and Dmytro Yeroshkin.
\newblock The weighted connection and sectional curvature for manifolds with
  density.
\newblock {\em J. Geom. Anal.}, 29(1):957--1001, 2019.
\newblock \href {https://doi.org/10.1007/s12220-018-0025-3}
  {\path{doi:10.1007/s12220-018-0025-3}}.

\bibitem{Ke24a}
Christian Ketterer.
\newblock Glued spaces and lower curvature bounds.
\newblock {\em arXiv e-prints}, 2024.
\newblock \href {http://arxiv.org/abs/2408.13137} {\path{arXiv:2408.13137}}.

\bibitem{Ke24}
Christian Ketterer.
\newblock Glued spaces and lower {R}icci curvature bounds.
\newblock {\em arXiv e-prints}, 2024.
\newblock \href {http://arxiv.org/abs/2308.06848} {\path{arXiv:2308.06848}}.

\bibitem{Ko09}
J\'{a}nos Koll\'{a}r.
\newblock Positive {S}asakian structures on 5-manifolds.
\newblock In {\em Riemannian topology and geometric structures on manifolds},
  volume 271 of {\em Progr. Math.}, pages 93--117. Birkh\"{a}user Boston,
  Boston, MA, 2009.
\newblock URL: \url{https://doi.org/10.1007/978-0-8176-4743-8_5}, \href
  {https://doi.org/10.1007/978-0-8176-4743-8\_5}
  {\path{doi:10.1007/978-0-8176-4743-8\_5}}.

\bibitem{Ko02}
N.~N. Kosovski\u{\i}.
\newblock Gluing of {R}iemannian manifolds of curvature {$\geq \kappa$}.
\newblock {\em Algebra i Analiz}, 14(3):140--157, 2002.

\bibitem{La07}
M.-L. Labbi.
\newblock On two natural {R}iemannian metrics on a tube.
\newblock {\em Balkan J. Geom. Appl.}, 12(2):81--86, 2007.

\bibitem{La70}
H.~Blaine Lawson, Jr.
\newblock The unknottedness of minimal embeddings.
\newblock {\em Invent. Math.}, 11:183--187, 1970.
\newblock \href {https://doi.org/10.1007/BF01404649}
  {\path{doi:10.1007/BF01404649}}.

\bibitem{Lo03}
John Lott.
\newblock Some geometric properties of the {B}akry-\'{E}mery-{R}icci tensor.
\newblock {\em Comment. Math. Helv.}, 78(4):865--883, 2003.
\newblock \href {https://doi.org/10.1007/s00014-003-0775-8}
  {\path{doi:10.1007/s00014-003-0775-8}}.

\bibitem{Mi61}
John Milnor.
\newblock A procedure for killing homotopy groups of differentiable manifolds.
\newblock In {\em Proc. {S}ympos. {P}ure {M}ath., {V}ol. {III}}, pages 39--55.
  Amer. Math. Soc., Providence, RI, 1961.

\bibitem{MW23}
Kenneth Moore and Eric Woolgar.
\newblock Bakry--\'{E}mery {R}icci curvature, {$X$}-minimal hypersurfaces, and
  near horizon geometries.
\newblock {\em J. Math. Phys.}, 64(2):Paper No. 022504, 17, 2023.
\newblock \href {https://doi.org/10.1063/5.0113859}
  {\path{doi:10.1063/5.0113859}}.

\bibitem{Mo05}
Frank Morgan.
\newblock Manifolds with density.
\newblock {\em Notices Amer. Math. Soc.}, 52(8):853--858, 2005.

\bibitem{Pe97}
Grigori Perelman.
\newblock Construction of manifolds of positive {R}icci curvature with big
  volume and large {B}etti numbers.
\newblock In {\em Comparison geometry ({B}erkeley, {CA}, 1993--94)}, volume~30
  of {\em Math. Sci. Res. Inst. Publ.}, pages 157--163. Cambridge Univ. Press,
  Cambridge, 1997.

\bibitem{PW14}
Curtis Pro and Frederick Wilhelm.
\newblock Riemannian submersions need not preserve positive {R}icci curvature.
\newblock {\em Proc. Amer. Math. Soc.}, 142(7):2529--2535, 2014.
\newblock \href {https://doi.org/10.1090/S0002-9939-2014-11960-5}
  {\path{doi:10.1090/S0002-9939-2014-11960-5}}.

\bibitem{Re23}
Philipp Reiser.
\newblock Generalized surgery on {R}iemannian manifolds of positive {R}icci
  curvature.
\newblock {\em Trans. Amer. Math. Soc.}, 376(5):3397--3418, 2023.
\newblock \href {https://doi.org/10.1090/tran/8789}
  {\path{doi:10.1090/tran/8789}}.

\bibitem{Re24}
Philipp Reiser.
\newblock Positive {R}icci curvature on connected sums of fibre bundles.
\newblock {\em arXiv e-prints}, 2024.
\newblock \href {http://arxiv.org/abs/2406.02274} {\path{arXiv:2406.02274}}.

\bibitem{RW23}
Philipp Reiser and David~J. Wraith.
\newblock A generalization of the {P}erelman gluing theorem and applications.
\newblock {\em arXiv e-prints}, 2023.
\newblock \href {http://arxiv.org/abs/2308.06996} {\path{arXiv:2308.06996}}.

\bibitem{RW23b}
Philipp Reiser and David~J. Wraith.
\newblock Intermediate {R}icci curvatures and {G}romov's {B}etti number bound.
\newblock {\em J. Geom. Anal.}, 33(12):Paper No. 364, 20, 2023.
\newblock \href {https://doi.org/10.1007/s12220-023-01423-6}
  {\path{doi:10.1007/s12220-023-01423-6}}.

\bibitem{RW23a}
Philipp Reiser and David~J. Wraith.
\newblock Positive intermediate {R}icci curvature on connected sums.
\newblock {\em Algebr. Geom. Topol. (to appear)}, 2024.
\newblock \href {http://arxiv.org/abs/2310.02746} {\path{arXiv:2310.02746}}.

\bibitem{Sc12}
Arthur Schlichting.
\newblock Gluing {R}iemannian manifolds with curvature operators at least
  $\kappa$.
\newblock {\em arXiv e-prints}, 2012.
\newblock \href {http://arxiv.org/abs/1210.2957} {\path{arXiv:1210.2957}}.

\bibitem{SY79}
Richard {Schoen} and Shing-Tung {Yau}.
\newblock On the structure of manifolds with positive scalar curvature.
\newblock {\em Manuscripta Math.}, 28(1-3):159--183, 1979.
\newblock \href {https://doi.org/10.1007/BF01647970}
  {\path{doi:10.1007/BF01647970}}.

\bibitem{SY91}
Ji-Ping Sha and DaGang Yang.
\newblock Positive {R}icci curvature on the connected sums of {$S^n\times
  S^m$}.
\newblock {\em J. Differential Geom.}, 33(1):127--137, 1991.
\newblock URL: \url{http://projecteuclid.org/euclid.jdg/1214446032}.

\bibitem{Sm62}
Stephen Smale.
\newblock On the structure of {$5$}-manifolds.
\newblock {\em Ann. of Math. (2)}, 75:38--46, 1962.
\newblock \href {https://doi.org/10.2307/1970417} {\path{doi:10.2307/1970417}}.

\bibitem{St92}
Stephan Stolz.
\newblock Simply connected manifolds of positive scalar curvature.
\newblock {\em Ann. of Math. (2)}, 136(3):511--540, 1992.
\newblock \href {https://doi.org/10.2307/2946598} {\path{doi:10.2307/2946598}}.

\bibitem{Wa62}
C.~T.~C. Wall.
\newblock Classification of {$(n-1)$}-connected {$2n$}-manifolds.
\newblock {\em Ann. of Math. (2)}, 75:163--189, 1962.

\bibitem{WW09}
Guofang Wei and Will Wylie.
\newblock Comparison geometry for the {B}akry-{E}mery {R}icci tensor.
\newblock {\em J. Differential Geom.}, 83(2):377--405, 2009.
\newblock \href {https://doi.org/10.4310/jdg/1261495336}
  {\path{doi:10.4310/jdg/1261495336}}.

\bibitem{Wh36}
Hassler Whitney.
\newblock Differentiable manifolds.
\newblock {\em Ann. of Math. (2)}, 37(3):645--680, 1936.
\newblock \href {https://doi.org/10.2307/1968482} {\path{doi:10.2307/1968482}}.

\bibitem{Wr97}
David Wraith.
\newblock Exotic spheres with positive {R}icci curvature.
\newblock {\em J. Differential Geom.}, 45(3):638--649, 1997.
\newblock URL: \url{http://projecteuclid.org/euclid.jdg/1214459846}.

\bibitem{Wr98}
David Wraith.
\newblock Surgery on {R}icci positive manifolds.
\newblock {\em J. Reine Angew. Math.}, 501:99--113, 1998.
\newblock \href {https://doi.org/10.1515/crll.1998.082}
  {\path{doi:10.1515/crll.1998.082}}.

\bibitem{Wr02}
David Wraith.
\newblock Deforming {R}icci positive metrics.
\newblock {\em Tokyo J. Math.}, 25(1):181--189, 2002.
\newblock \href {https://doi.org/10.3836/tjm/1244208944}
  {\path{doi:10.3836/tjm/1244208944}}.

\bibitem{Wu58}
Wen-ts\"un Wu.
\newblock On the isotopy of {$C\sp{r}$}-manifolds of dimension {$n$} in
  euclidean {$(2n+1)$}-space.
\newblock {\em Sci. Record (N.S.)}, 2:271--275, 1958.

\bibitem{Wy15}
William Wylie.
\newblock Sectional curvature for {R}iemannian manifolds with density.
\newblock {\em Geom. Dedicata}, 178:151--169, 2015.
\newblock \href {https://doi.org/10.1007/s10711-015-0050-3}
  {\path{doi:10.1007/s10711-015-0050-3}}.

\bibitem{WY16}
William Wylie and Dmytro Yeroshkin.
\newblock On the geometry of {R}iemannian manifolds with density.
\newblock {\em arXiv e-prints}, 2016.
\newblock \href {http://arxiv.org/abs/1602.08000} {\path{arXiv:1602.08000}}.

\end{thebibliography}

\end{document}